\documentclass{amsproc}
\usepackage{amsmath,amssymb,amsthm,amsfonts,latexsym}
\usepackage{euscript,color,textcomp,mathrsfs}
\usepackage{multirow}

\newcommand*{\pplus}{\unitlength=1.5pt
\begin{picture}(3,3)
\linethickness{1.5pt}\put(1.5,0){\line(0,1){3}}
\put(0,1.5){\line(1,0){3}}\end{picture}}

\newtheorem{thm}{Theorem}
\newtheorem{cor}[thm]{Corollary}
\newtheorem{lem}[thm]{Lemma}
\newtheorem{pro}[thm]{Proposition}
\newtheorem*{thmA}{Theorem A}
\newtheorem*{thmB}{Theorem B}
\newcounter{step}

\theoremstyle{remark}
\newtheorem{rem}[thm]{Remark}
\newtheorem{exa}[thm]{Example}

\theoremstyle{definition}

\DeclareMathOperator{\RE}{\mathfrak{Re}}
\DeclareMathOperator{\supp}{supp}

\newcommand*{\cbb}{\mathbb C}
\newcommand*{\D}{\mathrm d}
\newcommand*{\dcal}{\mathcal D}
\newcommand*{\dz}[1]{\mathscr D(#1)}
\newcommand*{\E}{\mathrm e}
\newcommand*{\ebf}{{\boldsymbol e}}
\newcommand*{\fcal}{\mathcal F}
\newcommand*{\Ge}{\geqslant}

\newcommand*{\hh}{\mathcal H}
\newcommand*{\is}[2]{\langle#1,#2\rangle}
\newcommand*{\ima}{{\mathfrak{Im}}}
\newcommand*{\kk}{\mathcal K}
\newcommand*{\I}{\mathrm{i}}
\newcommand*{\Le}{\leqslant}
\newcommand*{\llangle}{\langle\mbox{\hspace{-.6ex}}\langle}
\newcommand*{\llbra}{[\mbox{\hspace{-.27ex}}[}
\newcommand*{\minuss}{\text{\tiny{$-$}}}
\newcommand*{\nfr}{\mathfrak N}
\newcommand*{\nul}{\{0\}}
\newcommand*{\okr}{\stackrel{\scriptscriptstyle{\mathsf{def}}}{=}}
\newcommand*{\pcal}{\mathcal P}
\newcommand*{\pluss}{\text{\tiny{$+$}}}
\newcommand*{\rbb}{{\mathbb R}}
\newcommand*{\rea}{{\mathfrak{Re}}}
\newcommand*{\rrangle}{\rangle\mbox{\hspace{-.6ex}}\rangle}
\newcommand*{\rrbra}{]\mbox{\hspace{-.27ex}}]}
\newcommand*{\scal}{\mathcal S}
\newcommand*{\sfr}{\mathfrak S}

\newcommand*{\xbf}{\boldsymbol x}
\newcommand*{\xfr}{\mathfrak X}
\newcommand*{\zbb}{\mathbb Z}
\newcommand*{\zfr}{\mathfrak Z}

\begin{document}
   \title[Extending positive definiteness]{Extending positive definiteness}

   \author[D.\ Cicho\'n]{Dariusz Cicho\'n}
\address{Instytut Matematyki, Uniwersytet Jagiello\'nski,
ul. {\L}ojasiewicza 6, PL-30348 Kra\-k\'ow} \email{Dariusz.Cichon@im.uj.edu.pl}
   \author[J.\ Stochel]{Jan Stochel}
\address{Instytut Matematyki, Uniwersytet Jagiello\'nski,
ul. {\L}ojasiewicza 6, PL-30348 Kra\-k\'ow} \email{Jan.Stochel@im.uj.edu.pl}
   \author[F.\ H.\ Szafraniec]{Franciszek Hugon Szafraniec}
\address{Instytut Matematyki, Uniwersytet Jagiello\'nski,
ul. {\L}ojasiewicza 6, PL-30348 Kra\-k\'ow} \email{Franciszek.Szafraniec@im.uj.edu.pl}
  \thanks{This work was partially supported by the KBN grant
2 P03A 037 024 and by the MNiSzW grant N201 026 32/1350.
The third author also would like to acknowledge an
assistance of the EU Sixth Framework Programme for the
Transfer of Knowledge ``Operator theory methods for
differential equations'' (TODEQ) \# MTKD-CT-2005-030042.}
   \subjclass{Primary 43A35, 44A60 Secondary 47A20, 47B20}
   \keywords{Positive definite mapping, completely positive mapping, completely
f-positive mapping, complex moment problem,
multidimensional trigonometric moment problem, truncated
moment problems, the 17th Hilbert problem, subnormal
operator, unitary power dilation for several contractions}
    \dedicatory{To the memory of M.G.\ Kre\u{\i}n (1907-1989)}
   \begin{abstract}
The main result of the paper gives criteria for
extendibility of sesquilinear form-valued mappings
defined on symmetric subsets of $*$-semi\-groups
to positive definite ones. By specifying this we obtain new solutions~ of:
    \begin{itemize}
    \item the truncated complex moment problem,
    \item the truncated multidimensional trigonometric moment problem,
    \item the truncated two-sided complex moment problem,
    \end{itemize}
as well as characterizations of unbounded subnormality
and criteria for the existence of unitary power dilation.
   \end{abstract}
   \maketitle
   \subsection*{Introduction}
In \cite{sz-n} a fairly general concept of
$*$-semigroups, which includes groups, $*$-algebras
and quite a number of instances in between, and
positive definite functions on them has been
originated by Sz.-Nagy. On the other hand, there are
different notions which are related to positive
definiteness:\ positivity of sequences in the theory of
moments and complete positivity in $C^*$-algebras.
Positivity understood in the sense of Marcel Riesz and
Haviland\,\footnote{\;For the reader's convenience we
formulate both the real and the complex variants of
the Riesz-Haviland theorem as Theorems A and B in
Appendix.} usually ensures the sequence to be a moment
one while complete positivity works for dilations
(Stinespring, and what is equivalent, Sz.-Nagy, see
\cite{sza} for the argument).

If one goes beyond $C^*$-algebras the two notions,
positive definiteness and complete positivity, still
make sense but are no longer equivalent. This happens
when one deals with unbounded operator valued
functions and moment problems on unbounded sets.
Therefore, there is a need of common treatment of
these by means of forms over $*$-semigroups, like in
\cite{ark}. The aforesaid cases are represented in our
paper by unbounded subnormal operators and the complex
moment problem. In addition to this we also consider
the ``bounded'' case of unitary dilations of several
operators and, what is related to, of the
multidimensional trigonometric moment problem.

A topic attracting attention of quite a number of
mathematicians is extendibility of functions to either
positive definite or completely positive ones. The
most classical result in this matter concerns
\underline{groups}. It states that every continuous
positive definite function on a closed subgroup of a
locally compact abelian (or compact, not necessarily
abelian) group extends to a continuous positive
definite function on the whole group (cf.\
\cite[Section 34.48(a)\&(c)]{h-r} and \cite[Theorem
3.16]{bur}). However, leaving the topological
requirements aside, each positive definite function on
a subgroup of a group $G$ extends by zero to a
positive definite function on the whole $G$ (cf.\
\cite[Section 32.43(a)]{h-r}). This procedure is no
longer applicable to $*$-semigroups other than
groups. What is more, not every positive definite
function on a $*$-subsemigroup extends to a positive
definite function on the whole $*$-semigroup. This is
best exemplified by the interplay between the
$*$-semigroups $\nfr \subset \nfr_\pluss$ as shown in
\cite{st-sz1} in connection with the complex moment
problem (see Example \ref{nag} and Section \ref{s4}
for the definitions).

The situation becomes more complicated when one wants
to extend positive definiteness from subsets of more
relaxed structure, even in the case of groups. The
classical result of Krein \cite{krein} on automatic
extendibility of continuous positive definite
functions from a symmetric interval to the whole real
line suggests that \underline{symmetricity} of the
subset may be essential. This is somehow confirmed by
\cite{st-sz2} which contains a full characterization
of several contractions having commuting unitary
dilations. For more discussion of the role played by
symmetry we refer to Section \ref{bsym}. The results
contained in \cite{ba,ba-ti} and \cite{brudo} also
corroborate the importance of symmetry, the latter
 concerns extensions to indefinite forms with finite number of
negative squares.

One of the main ideas of the present paper is to
employ generalized polynomial functions to the
extendibility criteria invented in \cite{st-sz1} and
\cite{st-sz2}. What we get is strictly related to
complete positivity of associated linear mappings. The
original contribution consists in introducing complete
f-positivity (cf.\ Theorems \ref{main2} and
\ref{main2+}). This results, in particular, in the
complex variant of the Riesz-Haviland theorem (cf.\
Theorem \ref{zespmom}); the complete f-positivity is
now written in terms of positivity of the associated
linear functional on the set of all finite sums of
squares of moduli of very special rational functions
in variables $z$ and $\bar z$.

Carefully selected applications are chosen as follows.
Considering determining subsets of $\xfr_{\nfr_+}$
allows us to apply Theorems \ref{main2} and
\ref{main2+} to the complex moment problem (Theorem
\ref{zespmom}) as well as to subnormal operators
(Theorem \ref{chsub}). Theorem \ref{zespmom} can be
thought of as a truncated moment problem, however not
in the usual sense of \underline{finite} sections.
Analogous results are formulated for the truncated
multidimensional trigonometric moment problem and the
truncated two-sided complex moment problem (cf.\
Theorems \ref{trmpr} and \ref{zespmom+}). Moreover,
Theorem \ref{nagy} contains a new characterization of
finite systems of Hilbert space operators admitting
unitary power dilations.

Section \ref{appr} deals, inter alia, with
approximation of nonnegative polynomials in
indeterminates $z$ and $\bar z$ by sums of finitely
many squares of moduli of rational functions that are
bounded on a neighbourhood of the origin which is
assumed to be their only possible singularity (cf.\
Proposition \ref{dens1}). This can be compared to
Artin's solution of the 17th Hilbert problem stating
that every nonnegative polynomial in $z$ and $\bar z$
is a sum of finitely many squares of moduli of (a
priori arbitrary) rational functions. The above
approximation is no longer possible when considered on
proper closed subsets of $\cbb$, cf.\ Proposition
\ref{sumk} (see also Proposition \ref{sumk+} for the
case of multivariable trigonometric polynomials).
Similar approximation holds for multivariable
trigonometric polynomials (cf.\ Proposition
\ref{dens2}). A more detailed discussion relating the
theme to selected recent articles \cite{dr,g-w} is
contained in Section \ref{appr}.
   \section*{\sc General criteria}
Besides keeping $\zbb$, $\rbb$, $\cbb$ for standard sets of
integer, real and complex numbers, respectively, by $\zbb_\pluss$
we understand the set $\{0,1,2, \ldots\}$. Moreover, we adopt the
notation $\mathbb T = \{z \in \cbb \colon |z| =1\}$ and $\cbb_* =
\cbb \setminus \nul$. As usual, $\chi_\sigma$ stands for the
characteristic function of $\sigma$, a subset of a set
$\varOmega$. A system $\{z_\omega\}_{\omega \in \varOmega}$ of
complex numbers is said to be {\em finite} if the set $\{\omega
\in \varOmega \colon z_\omega \neq 0\}$ is finite.
   \subsection{\label{s1}Polynomials on dual $*$-semigroup}
Given a nonempty set $\varOmega$, we denote by
$\cbb^\varOmega$ the complex $*$-algebra of all complex
functions on $\varOmega$ with the algebra operations
defined pointwisely and the involution
   \begin{align*}
f^*(\omega) = \overline{f(\omega)}, \quad f \in \cbb^\varOmega, \,
\omega \in \varOmega.
   \end{align*}
The following fact reveals the idea standing behind the
known characterization of linear independence of Laplace
transforms of elements of a commutative semigroup (see
\cite{h-z} and \cite[Proposition 6.1.8]{b-ch-r} for a
pattern of the proof).
   \begin{lem} \label{hew}
Let $\varOmega$ be a nonempty set and let $Y\subset\cbb^\varOmega$
be a semigroup $($with respect to the multiplication of
$\cbb^\varOmega$$)$ containing the constant function $1$. Then the
following conditions are equivalent{\rm :}
   \begin{enumerate}
   \item[(i)] $Y$ separates points of $\varOmega$
$($equivalently{\em :} if $\widehat{\omega_1}|_Y =
\widehat{\omega_2}|_Y$, then $\omega_1 = \omega_2$$)$,
   \item[(ii)] the system $\{\widehat \omega|_Y\}_{\omega
\in \varOmega}$ is linearly independent in $\cbb^Y$,
   \end{enumerate}
where for $\omega \in \varOmega$ the function $\widehat \omega
\colon \cbb^\varOmega \to \cbb$ is defined by $\widehat \omega (f)
= f(\omega)$, $f \in \cbb^\varOmega$.
   \end{lem}
   If the $Y$ in Lemma \ref{hew} is not semigroup, then
the implication (i)$\Rightarrow$(ii) may not longer be true (cf.\
Example \ref{nag}).

As far as abstract semigroups are concerned, we adhere
to the multiplicative notation. A mapping $\sfr \ni s
\longmapsto s^* \in \sfr$ defined on a semigroup
$\sfr$ is called an {\it involution} if $(st)^* = t^*
s^*$ and $(s^*)^* = s$ for all $s,t \in \sfr$. A
semigroup $\sfr$ equipped with an involution is said
to be a {\it $*$-semigroup}. It is clear that if
$\sfr$ has a unit $\varepsilon$, then $\varepsilon^* =
\varepsilon$. The set $\{s \in \sfr \colon s=s^*\}$ of
all {\em Hermitian elements}\label{hermi} of a
$*$-semigroup $\sfr$ is denoted by $\sfr_{\mathfrak
h}$. For a nonempty subset $T$ of a $*$-semigroup
$\sfr$, we write $T^* = \{s^* \colon s \in T\}$; $T$
is said to be {\em symmetric} if $T=T^*$. Put $[T] =
\bigcup_{n=1}^\infty T^{[n]}$, where $T^{[n]}$ stands
for the set of all products of length $n$ with factors
in $T$. The set $[T]$ is the smallest subsemigroup of
$\sfr$ containing $T$. Under the assumption
of commutativity of $\sfr$, the set $\llbra T \rrbra
\okr \{u^*v \colon u, v \in [T]\}$ is a
$*$-subsemigroup of $\sfr$ which does not have
to contain any of the sets $T$ and $T^*$. Neither $[T]$ nor $\llbra T \rrbra$ has to contain the unit of $\sfr$ even if it exists.

Let $\sfr$ be a commutative $*$-semigroup with a unit
$\varepsilon$. A function $\chi\colon \sfr \to \cbb$
is called a {\it character} of $\sfr$ if
   \allowdisplaybreaks
   \begin{align} \label{chi1}
& \chi(st) = \chi(s) \chi(t), \quad s,t \in \sfr,
   \\
& \chi(s^*) = \overline {\chi(s)}, \quad s \in \sfr, \label{chi2}
   \\
& \chi(\varepsilon) = 1. \label{chi3}
   \end{align}
The set $\xfr_\sfr$ of all characters of $\sfr$ is a
$*$-semigroup with respect to the multiplication and
the involution of $\cbb^\sfr$; $\xfr_\sfr$ is called a
{\em dual} $*$-semigroup of $\sfr$. For every $s \in
\sfr$, we define the function $\hat s \colon \xfr_\sfr
\to \cbb$, modelled on the Fourier, Gelfand or Laplace
transform, via
   \begin{align*}
\hat s(\chi) = \chi(s), \quad \chi \in \xfr_\sfr.
   \end{align*}
The set $\{\hat s \colon s \in \sfr\}$ will be denoted by $\widehat
\sfr$. It follows from \eqref{chi1}, \eqref{chi2} and \eqref{chi3}
that $\widehat{st} = \hat s \cdot \hat t$ and $\widehat{s^*} =
\overline{\hat s}$ for all $s,t \in \sfr$, and $\hat\varepsilon
\equiv 1$. This means that $\widehat \sfr$ is a $*$-semigroup with
respect to the multiplication and the involution of
$\cbb^{\xfr_\sfr}$.

The set $\fcal_\sfr$ of all complex functions on $\sfr$ vanishing
off finite sets is a complex $*$-algebra with pointwise defined
linear algebra operations, the algebra multiplication of
convolution type
   $$(f \star g) (u) = \sum_{\substack {s,t \in \sfr \\ u = s
t}} f(s) g(t), \quad f,g \in \fcal_\sfr, \; u \in \sfr,
   $$
   and involution
   $$ f^* (s) = \overline {f(s^*)}, \quad f \in \fcal_\sfr, \; s
\in
\sfr.
   $$
For a nonempty subset $T$ of $\sfr$, we define the linear subspace
$\fcal_{\sfr,T}$ of $\fcal_\sfr$ via
   \begin{align*}
\fcal_{\sfr,T} = \{f \in \fcal_\sfr \colon f \text{ vanishes off
the set } T\}.
   \end{align*}
With the algebra operations defined above, $\fcal_{\sfr,T}$ is a
subalgebra of $\fcal_\sfr$ if and only if $T$ is a subsemigroup of
$\sfr$; what is more, $\fcal_{\sfr,T}$ is a symmetric subset of
$\fcal_\sfr$ if and only if $T$ is a symmetric subset of $\sfr$.
The reader should be aware of the fact that if $\sfr$ is finite,
then though the sets $\fcal_\sfr$ and $\cbb^\sfr$ coincide their
$*$-algebra structures differ unless $\sfr$ is a singleton.

Given a nonempty subset $Y$ of $\xfr_\sfr$, we write
$\pcal (Y)$ for the linear span of $\{\hat s|_Y \colon
s \in \sfr\}$ in $\cbb^Y$. Clearly, $\pcal(Y)$ is a
$*$-subalgebra of the $*$-algebra $\cbb^Y$. Notice
that there exists a unique $*$-algebra epimorphism
$\varDelta_Y \colon \fcal_\sfr \to \pcal(Y)$ such that
   \begin{align*}
\varDelta_Y(\delta_s) = \hat s|_Y, \quad s \in \sfr,
   \end{align*}
where $\delta_s \in \fcal_\sfr$ is the characteristic
function of $\{s\}$. For a nonempty subset $T$ of
$\sfr$, we write
$\pcal_T(Y)=\varDelta_Y(\fcal_{\sfr,T})$; the linear
space $\pcal_T(Y)$ coincides with the linear span of
$\{\hat s|_Y \colon s \in T\}$.

   In this paper we are interested in the case in
which the $*$-algebra homomorphism $\varDelta_Y$ is
injective. The following is a direct consequence of
Lemma \ref{hew}.
   \begin{pro}\label{char}
If $Y$ is a subsemigroup of $\xfr_\sfr$ containing the constant
function $1$, then the following conditions are equivalent
   \begin{enumerate}
   \item[(i)] $\varDelta_Y$ is a $*$-algebra isomorphism,
   \item[(ii)] the system $\{\hat s|_Y\}_{s \in \sfr}$
is linearly independent in $\cbb^Y$,
   \item[(iii)] $Y$ separates the points of $\sfr$
$($equivalently{\em :} if $\hat s|_Y = \hat t|_Y$, then $s =
t$$)$.
   \end{enumerate}
   \end{pro}
   \subsection{Determining sets}
A nonempty subset $Y$ of $\xfr_\sfr$ is called {\em determining
$($\/for $\pcal(\xfr_\sfr)$}) if it satisfies any of the equivalent
conditions (i) and (ii) of Proposition \ref{char}. If a subset $Y$
of $\xfr_\sfr$ is determining, then the mapping
   \begin{align*}
\pi_{T,Y} \colon \pcal_T(\xfr_\sfr) \ni w \mapsto w|_Y \in
\pcal_T(Y)
   \end{align*}
is a well defined $*$-algebra isomorphism (but not conversely, see
the next section).

It may happen that the whole dual $*$-semigroup $\xfr_\sfr$ does
not separate the points of $\sfr$ and consequently $\xfr_\sfr$ is
not determining (cf.\ \cite[Remarks 4.6.9\,(1)]{b-ch-r}). If $Y$
is not a subsemigroup of $\xfr_\sfr$, then implication
(iii)$\Rightarrow$(ii) of Proposition \ref{char} may fail to hold
(implication (ii)$\Rightarrow$(iii) is always true). To see this,
we shall discuss a $*$-semigroup introduced in 1955 by Sz.-Nagy
\cite{sz-n} for which the set $\pcal(\xfr_\sfr)$ can be
interpreted as the ring of all polynomials in $z$ and $\bar z$.
   \begin{exa} \label{nag}
Equip the Cartesian product $\zbb_\pluss \times
\zbb_\pluss$ with coordinatewise defined addition as
semigroup operation, i.e.\ $(i,j) + (k,l) = (i+k,j+l)$, and
involution $(m,n)^*=(n,m)$. The $*$-semigroup so obtained
will be denoted by $\nfr$. If $\cbb$ is thought of as a
$*$-semigroup with multiplication as semigroup operation
and complex conjugation as involution, then the mapping
   \begin{align*}
\xfr_\nfr \ni \chi \mapsto \chi(1,0) \in \cbb,
   \end{align*}
being a $*$-semigroup isomorphism, enables us to identify
algebraically $\xfr_\nfr$ with $\cbb$. Under this identification,
we have
   \begin{align} \label{1+}
\widehat{(m,n)} (z) = z^m \bar z^n, \quad m,n \in
\zbb_\pluss,\, z \in \cbb.
   \end{align}
Note that by \eqref{1+} and Lemma \ref{1} below the
system $\big\{\widehat{(m,n)}\big\}_ {(m,n) \in \nfr}$
is linearly independent in $\cbb^\cbb$. Hence, the set
$\pcal(\xfr_\nfr)$ can be thought of as the ring
$\cbb[z,\bar z]$ of all complex polynomials in $z$ and
$\bar z$.

It turns out that it is possible to give satisfactory
description of all subset of $\cbb$ separating the points
of $\nfr$.
   \begin{pro}\label{separuje}
A subset $Y$ of $\cbb$ separates the points of $\nfr$
if and only if the following two conditions hold{\em
:}
   \begin{enumerate}
   \item[(i)] $Y \not\subset \mathbb T \cup \nul$,
   \item[(ii)] $Y_1 \okr \big\{ \frac z{|z|}
\colon z \in Y\setminus \nul \big\} \not \subset \{w\in\cbb \colon
w^\varkappa=1\}$ for every integer $\varkappa \Ge 1$.
   \end{enumerate}
   \end{pro}
   \begin{proof}
Suppose that $Y$ does not satisfy the conjunction of conditions
(i) and (ii). Then either $Y \subset \mathbb T\cup\nul$, and hence
$z^2\bar z = z$ for all $z\in Y$, or $Y_1 \subset \{w\in\cbb
\colon w^\varkappa=1\}$ for some $\varkappa \Ge 1$, and hence
$z^{2\varkappa} = z^\varkappa \bar z^\varkappa$ for all $z\in Y$.
In both cases, $Y$ cannot separate the points of $\nfr$.

Assume now that $Y$ satisfies (i) and (ii). Take $m,n,k,l
\in \zbb_\pluss$ and suppose that $z^m \bar z^n = z^k \bar
z^l$ for all $z\in Y$. Taking absolute value of both sides
of the equality and employing (i) we get
   \begin{align}\label{m+n}
j\okr m+n = k+l.
   \end{align}
Dividing both sides of $z^m \bar z^n = z^k \bar z^l$ by
$|z|^j$ gives $w^\varkappa =1$ for all $w \in Y_1$ with
$\varkappa = m-n - (k-l)$. By (ii) this implies that
$\varkappa = 0$, which when combined with \eqref{m+n} leads
to $(m,n)=(k,l)$. The proof is complete.
   \end{proof}
We now indicate a class of subsets $Y$ of $\xfr_\nfr$ which
separate the points of $\nfr$ but which are not determining
for $\pcal(\xfr_\nfr) = \cbb[z,\bar z]$. Take a nonzero
polynomial $p \in \cbb[z,\bar z]$ such that the set
   \begin{align}\label{y-p}
   Y_p \okr \{z \in \cbb \colon p(z,\bar z) = 0\}
   \end{align}
is nonempty. Then evidently the functions
$\{\widehat{(m,n)}|_{Y_p} \colon m,n \in \zbb_\pluss\}$ are
linearly dependent in $\cbb^{Y_p}$, and hence $Y_p$ is not
determining for $\cbb[z,\bar z]$. Let us focus on the case
of circles and straight lines (which are always of the form
\eqref{y-p}). It follows from Proposition \ref{separuje}
that the ensuing sets do not separate the points of $\nfr$:
   \begin{enumerate}
   \item[--] the unit circle $\mathbb T$ centered at the origin,
   \item[--] a straight line $L$ such that $0\in L$ and the
points of $L\cap \mathbb T$ are complex $\varkappa$-roots
of $1$ for some (necessarily even) integer $\varkappa \Ge 2$.
   \end{enumerate}
All the other circles and straight lines do separate the
points of $\nfr$ (in many cases they embrace $1$).
Surprisingly, one point sets, which are still of the form
\eqref{y-p}, may separate the points of $\nfr$. Indeed, by
Proposition \ref{separuje}, if $z \in \cbb \setminus
(\mathbb T \cup \nul)$ and $\frac z{|z|}$ is not a complex
$\varkappa$-root of $1$ for any integer $\varkappa \Ge 1$,
then $\{z\}$ does separate the points of $\nfr$.
   \end{exa}
Contrary to the case of sets separating the points of
$\nfr$, one cannot expect any neat description of all
determining sets for $\cbb[z,\bar z]$. Nevertheless, we may
indicate some determining sets explicitly (see also Lemma
\ref{niepust}).
   \begin{lem} \label{1} Suppose that
$Y \subset \cbb$ is either a union of infinitely many
parallel straight lines or a union of infinitely many
concentric circles. If $\{a_{m,n}\}_{m,n \in \zbb}$
is a finite system of complex numbers such that
$\sum_{m,n \in \zbb} a_{m,n} z^m \bar z^n = 0$ for
all $z \in Y \setminus \nul$, then $a_{m,n} = 0$ for
all $m,n \in \zbb$.
   \end{lem}
   \begin{proof}
Since the set $\{(m,n) \in \zbb \times \zbb \colon
a_{m,n} \neq 0\}$ is finite, there exists an integer
$N \Ge 1$ such that $a_{m,n}=0$ for all integers $m,n$
such that $m < -N$ or $n < -N$. This implies that
$p(z) = \sum_{m,n \Ge -N} a_{m,n} z^{m+N} \bar
z^{n+N}$ is a complex polynomial in $z$ and $\bar z$
vanishing on $Y$. In the case of straight lines, we
can always find $\theta \in (-\pi/2,\pi/2]$ such that
the polynomial $p(\E^{\I \theta}z)$ vanishes on a
union of infinitely many straight lines parallel to
the real axis. Next, considering the complex
polynomial $p(\E^{\I \theta}(x+ \I y))$ in two real
variables $x$ and $y$, we deduce that $p(z)=0$ for all
$z\in \cbb$. In the other case, applying a suitable
translation of the argument, we can assume that the common
center of the circles is in the origin. Employing a
well known identity principle for complex polynomials
in $z$ and $\bar z$ completes the proof in both cases.
   \end{proof}
   \subsection{Complete positivity}
Let $\sfr$ be a unital commutative $*$-semigroup, $T$ be a
nonempty subset of $\sfr$ and $Y$ be a nonempty subset of
$\xfr_\sfr$. Denote by $\pcal_T(Y,\ell^2)$ the set of all
functions from $Y$ to $\ell^2$ of the form $\chi\mapsto\sum_{s\in
T} \chi(s) x_s$, where $\{x_s\}_{s\in T} \subset \ell^2$ and
$x_s=0$ for all but a finite number of $s$'s. Note that $\pcal_T
(Y,\ell^2)$ can be thought of as the algebraic tensor product
$\pcal_T(Y) \otimes \ell^2$. We abbreviate $\pcal_\sfr (Y,\ell^2)$
to $\pcal(Y,\ell^2)$. The standard notation $[a_{k,l}]_{k,l=1}^m
\Ge 0$ is used for {\em nonnegativity} of the scalar matrix
$[a_{k,l}]_{k,l=1}^m$.

For the linear space
   \begin{align*}
{\mathrm M}^m(\pcal_T(Y)) & \okr \big\{[w_{k,l}]_{k,l=1}^m \colon
w_{k,l} \in \pcal_T(Y) \text{ for all } k,l = 1, \ldots, m \big\},
\quad m\Ge1,
   \end{align*}
   its subsets
    \begin{align*}
{\mathrm M}_\pluss^m(\pcal_T(Y)) &\okr \big\{[w_{k,l}]_{k,l=1}^m
\in {\mathrm M}^m(\pcal_T(Y)) \colon [w_{k,l}(\chi)]_{k,l=1}^m \Ge
0 \text{ for all } \chi \in Y\big\}
\end{align*}
and
\begin{align*}
{\mathrm M}_{\mathrm f}^m(\pcal_T(Y)) & \okr {\mathrm
M}^m(\pcal_T(Y))\cap \big\{[\llangle p_k, p_l \rrangle]_{k,l=1}^m
\colon p_1, \ldots, p_m \in \pcal(Y,\ell^2) \big\},
   \end{align*}
where $\llangle p_k, p_l \rrangle$ is the function $\chi
\mapsto \is{p_k(\chi)} {p_l(\chi)}_{\ell^2}$, turn out to be
convex cones. Note that if $p_1, \ldots, p_m \in \pcal(Y,\ell^2)$,
then $[\llangle p_k, p_l \rrangle]_{k,l=1}^m \in {\mathrm
M}_\pluss^m(\pcal(Y))$. This implies that ${\mathrm
M}_{\mathrm f}^m(\pcal_T(Y)) \subset {\mathrm
M}_\pluss^m(\pcal_T(Y))$. One can also check that the square of each matrix
from ${\mathrm M}_\mathrm{f}^m(\pcal_T(Y))$ (respectively
${\mathrm M}_\pluss^m(\pcal_T(Y))$) belongs to ${\mathrm
M}_\mathrm{f}^m(\pcal_{T^{[2]}}(Y))$ (respectively
${\mathrm M}_\pluss^m(\pcal_{T^{[2]}}(Y))$). What is more, if $Y$ is a determining subset of $\xfr_\sfr$,
then the mappings
   \allowdisplaybreaks
   \begin{gather} \notag
\pcal(\xfr_\sfr,\ell^2) \ni p \mapsto p|_Y \in \pcal(Y,\ell^2),
\\
\pi_{T,Y}^{(m)} \colon {\mathrm
M}^m(\pcal_T(\xfr_\sfr)) \ni [w_{k,l}]_{k,l=1}^m
\longmapsto [w_{k,l}|_Y]_{k,l=1}^m \in {\mathrm M}
^m(\pcal_T(Y)), \label{pimty}
   \end{gather}
are linear isomorphisms and consequently
   \begin{align} \label{pitym}
\pi_{T,Y}^{(m)} ({\mathrm M}_{\mathrm f}^m(\pcal_T(\xfr_\sfr))) =
{\mathrm M}_{\mathrm f}^m(\pcal_T(Y)).
   \end{align}
   \begin{lem} \label{factor}
If $p_1, \ldots, p_m \in \pcal(Y,\ell^2)$ and $A = [\llangle p_k,
p_l \rrangle]_{k,l=1}^m$, then there exists a matrix
$P=[p_{k,j}]_{k=1}^m{}_{j=1}^n$ with entries $p_{k,j} \in \pcal(Y)$
such that $A = PP^*$, where $P^* = [\bar
p_{j,k}]_{k=1}^n{}_{j=1}^m$. Conversely, if $P$ is a matrix of size
$m \times n$ with entries in $\pcal(Y)$, then $PP^* = [\llangle p_k,
p_l \rrangle]_{k,l=1}^m $ with some $p_1, \ldots, p_m \in
\pcal(Y,\ell^2)$. In particular, if $p \in \pcal(Y,\ell^2)$, then
$\llangle p, p \rrangle = \sum_{j=1}^n |q_j|^2$ with some $q_j \in
\pcal(Y)$.
   \end{lem}
   \begin{proof} Take $p_1, \ldots, p_m \in
\pcal(Y,\ell^2)$. Then there exists a finite orthonormal
basis $\{e_j\}_{j=1}^n$ of the linear span of
$\bigcup_{k=1}^m p_k(Y)$. As a consequence, $p_k =
\sum_{j=1}^n p_{k,j} e_j$ with some $p_{k,j} \in \pcal(Y)$,
and hence $P=[p_{k,j}]_{k=1}^m{}_{j=1}^n$ is the required
matrix. Reversing the above reasoning concludes the proof.
   \end{proof}
   We now examine the behavior of the classes ${\mathrm
M}_{\mathrm f}^m(\pcal_T(Y))$ and ${\mathrm
M}_\pluss^m(\pcal_T(Y))$ under the operation of transposing
their members.
   \begin{lem}\label{niska}
{\em (i)} If a matrix $[w_{k,l}]_{k,l=1}^m$ is a member of
${\mathrm M}_\pluss^m(\pcal_T(Y))$, then the transposed matrix
$[w_{l,k}]_{k,l=1}^m$ belongs to ${\mathrm
M}_\pluss^m(\pcal_{T^*}(Y))$.

{\em (ii)} If $p_1, \ldots, p_m \in \pcal_T(Y,\ell^2)$, then there
exist $q_1, \ldots, q_m \in \pcal_{T^*}(Y,\ell^2)$ such that
   \begin{align} \label{plqk}
\llangle p_l, p_k \rrangle = \llangle q_k, q_l \rrangle, \quad k,l
= 1, \ldots, m.
   \end{align}
   \end{lem}
   \begin{proof} (i) By nonnegativity of $[w_{k,l}]_{k,l=1}^m$, we have
$w_{l,k}(\chi) = \overline{w_{k,l}(\chi)}$ for all $\chi
\in Y$ and $k,l = 1, \ldots, m$. Hence, the application of
\eqref{chi2} justifies (i).

   (ii) Arguing as in the proof of Lemma \ref{factor}, we
can write $p_k = \sum_{j=1}^n p_{k,j} e_j$ with some
$p_{k,j} \in \pcal_T(Y)$. Then
   \begin{align*}
\llangle p_l, p_k \rrangle (\chi) = \sum_{j=1}^n p_{l,j}
(\chi) \overline{p_{k,j} (\chi)} = \llangle q_k, q_l
\rrangle (\chi), \quad \chi \in Y,
   \end{align*}
where $q_k(\chi) = \sum_{j=1}^n \overline{p_{k,j}(\chi)}
e_j$ for $\chi \in Y$. By \eqref{chi2}, the functions $Y
\ni \chi \mapsto \overline{p_{k,j}(\chi)} \in \cbb$ belong
to $\pcal_{T^*}(Y)$. This completes the proof.
   \end{proof}
Let $F$ be a topological linear space and $\tau$ be
its topology. The closure of a subset $W$ of $F$ with
respect to $\tau$ is denoted by
$\overline{W}^{\,\tau}$. Given an integer $m\Ge 1$, we
write ${\mathrm M}^m(F)$ for the linear space (with
entrywise defined linear operations) of all $m$ by $m$
matrices with entries in $F$. Identifying ${\mathrm
M}^m(F)$ with the Cartesian product of $m^2$ copies of
$F$, we may regard ${\mathrm M}^m(F)$ as a topological
linear space with the product topology $\tau^{(m)}$.

Call a nonempty subset $Z$ of $\xfr_\sfr$ $T$-{\em
bounded}\/\footnote{\label{wysoka}\;Note that in view of Tychonoff's
theorem $\sfr$-bounded sets coincide with subsets of
$\xfr_\sfr$ which are relatively compact in the
topology of pointwise convergence on $\sfr$.} if
$\sup_{\chi \in Z} |\chi(s)| < \infty$ for every $s
\in T$. It is obvious that a nonempty subset $Z$ of
$\xfr_\sfr$ is $T$-bounded if and only if it is
$\langle T \rangle$-bounded, where $\langle T \rangle$
stands for the unital $*$-semigroup generated by $T$.
Note also that nonempty $Z \subset \xfr_\sfr$ is
$T$-bounded if and only if for every (equivalently:\
for some) integer $m \Ge 1$ and for every $w \in
{\mathrm M}^m(\pcal_T(Z))$, $\sup_{\chi \in Z}
\|w(\chi)\| < \infty$; here $\|w(\chi)\|$ stands for
the operator norm of the matrix $w(\chi)$. Denote by
$\tau_{T,Y}$ the locally convex topology on
$\pcal_T(Y)$ given by the system of seminorms $w
\mapsto \sup_{\chi \in Z} |w(\chi)|$ indexed by
$T$-bounded subsets $Z$ of $Y$. Observe that the
topology $\tau_{T,Y}^{(m)}$ on ${\mathrm
M}^m(\pcal_T(Y))$ is identical with the locally convex
topology given by the system of seminorms $w \mapsto
\sup_{\chi \in Z} \|w(\chi)\|$ with $Z$ ranging over
all $T$-bounded subsets of $Y$. It is clear that the
topology $\tau_{T,Y}^{(m)}$ is stronger than the
topology of pointwise convergence on $Y$. If
$\xfr_\sfr$ is equipped with the topology of pointwise
convergence on $\sfr$, then $\tau_{T,Y}^{(m)}$ is
still stronger than the topology of uniform
convergence on compact subsets of $Y$ (see Footnote \ref{wysoka}). In turn, if
there exist integers $i,j,k,l \Ge 0$ such that $i+j
\neq k+l$ and
   \begin{align} \label{i+j}
s^{*i}s^j=s^{*k}s^l,\quad s\in T,
   \end{align}
then $\sup_{\chi \in \xfr_\sfr } |\chi(s)| \Le 1$ for every $s \in
T$. Hence, in this specific situation, the topology
$\tau_{T,Y}^{(m)}$ describes exactly the uniform convergence on
$Y$. Property \eqref{i+j} holds for any $T$ which is a unital
commutative inverse semigroup (take $i=l=1$, $j=2$ and $k=0$, cf.\
\cite{cl-pr1}); in particular, this is the case for $T$ being an
abelian group with involution $s^*=s^{-1}$.

Below we stick to the notations declared at the beginning of this section.
   \begin{lem}\label{closure}
${\mathrm M}_{\mathrm f}^m(\pcal_T(Y)) \subset
{\mathrm M}_\pluss^m(\pcal_T(Y)) \subset
\overline{{\mathrm M}_{\mathrm f}^m(\pcal_{\llbra T
\rrbra}(Y))}^{\,\tau};$ $\tau = \tau_{\llbra T
\rrbra,Y}^{(m)}$.
   \end{lem}
   \begin{proof}
Only the second inclusion has to be justified.
Consider a $\llbra T \rrbra$-bounded subset $Z$ of $Y$ and take $w =
[w_{k,l}]_{k,l=1}^m\in {\mathrm
M}_\pluss^m(\pcal_T(Y))$. By Lemma \ref{niska},
the transpose $v$ of $w$ is a member of ${\mathrm
M}_\pluss^m(\pcal_{T^*}(Y))$. Note also that
$\|v(\chi)\| = \|w(\chi)\|$ for every $\chi \in Y$,
and consequently, $M \okr \sup_{\chi \in Z}
\|v(\chi)\| < \infty$. Without loss of generality we
can assume that $M > 0$. By the Weierstrass theorem,
there exists a sequence of real polynomials $\{
\rho_n\}_{n=1}^\infty$ vanishing at $0$ which tends to
the square root function uniformly on $[0,M]$. For
$\chi \in Y$, we denote by $\sqrt{v(\chi)}$ the square
root of $v(\chi)$. As for every $\chi \in Z$ the norm
of the nonnegative matrix $v(\chi)$ is less than or
equal to $M$, we get
   \begin{align}\label{spec}
\sup_{\chi \in Z} \| \sqrt{v(\chi)} - \rho_n(v(\chi))\| \Le
\sup_{x\in [0,M]} |\sqrt x - \rho_n(x)|, \quad n \Ge 1.
   \end{align}
Let $\{e_k\}_{k=1}^m$ be the standard `0\,-1' basis of
$\cbb^m$. Write $\rho_n(v)e_k$ for the function $Y \ni
\chi \to \rho_n(v(\chi))e_k \in \ell^2$ (after a
natural identification). Since $\rho_n(0)=0$, we see
that $\rho_n(v) e_k \in \pcal_{[T^*]}(Y,\ell^2)$ and
hence $[\llangle \rho_n(v) e_k, \rho_n(v)e_l
\rrangle]_{k,l=1}^m \in {\mathrm M}_{\mathrm
f}^m(\pcal_{\llbra T \rrbra}(Y))$. It follows from
\eqref{spec} that $[\llangle \rho_n(v) e_k,
\rho_n(v)e_l \rrangle]_{k,l=1}^m$ converges uniformly
on $Z$ to $\tilde v= [\tilde v_{k,l}]_{k,l=1}^m$ as $n
\to \infty$, where $\tilde v_{k,l}(\chi) \okr
\is{\sqrt{v(\chi)}e_k}{\sqrt{v(\chi)}{e_l}}$ for $\chi
\in Y$. Since
   \begin{align*}
\tilde v_{k,l}(\chi) = \is{v(\chi)e_k}{e_l} = v_{l,k}(\chi) =
w_{k,l}(\chi), \quad \chi \in Y, \, k,l = 1, \ldots, m,
   \end{align*}
and the class of $\llbra T \rrbra$-bounded subsets of $Y$ is directed upwards by
inclusion, the proof is complete.
   \end{proof}
   \begin{rem} \label{onie}
By the proof of Lemma \ref{closure}, every matrix in ${\mathrm
M}_\pluss^m(\pcal_T(Y))$ can be approximated in the
topology $\tau_{\llbra T \rrbra,Y}^{(m)}$ by means of
matrices of the form $[\llangle
p_k,p_l\rrangle]_{k,l=1}^m \in {\mathrm M}_{\mathrm
f}^m(\pcal_{\llbra T \rrbra}(Y))$, where
$p_1,\ldots,p_m \in \pcal_{[T^*]}(Y,\ell^2)$.
      \end{rem}
Let $\dcal$ be a complex linear space. Denote by $\scal(\dcal)$
the set of all sesquilinear forms on $\dcal$. For every integer $m
\Ge 1$, we define:
   \allowdisplaybreaks
   \begin{align*}
\mathrm M^m(\scal(\dcal)) & = \big\{[\phi_{k,l}]_{k,l=1}^m \colon
\phi_{k,l} \in \scal (\dcal) \text{ for all } k,l=1, \ldots, m
\big\},
   \\
\mathrm M_\pluss^m(\scal(\dcal)) & = \big\{[\phi_{k,l}]_{k,l=1}^m
\in \mathrm M^m(\scal(\dcal)) \colon [\phi_{k,l}]_{k,l=1}^m \gg
0\big\},
   \end{align*}
where the notation $[\phi_{k,l}]_{k,l=1}^m \gg 0$ means that
   \begin{align*}
\sum_{k,l=1}^m \phi_{k,l} (h_k,h_l) \Ge 0 \quad \text{for all }
h_1, \ldots, h_m \in \dcal.
   \end{align*}
We say that a mapping $\varPsi\colon \sfr \to \scal(\dcal)$ is
{\em positive definite} if
   \begin{align*}
\sum_{k,l=1}^m \varPsi(s_l^*s_k)(h_k,h_l) \Ge 0
   \end{align*}
for every integer $m \Ge 1$ and for all $s_1, \ldots, s_m \in
\sfr$ and $h_1, \ldots, h_m \in \dcal$.

Suppose that $Y$ is a determining subset of $\xfr_\sfr$. Then for
every mapping $\varPhi\colon T \to \scal(\dcal)$, there exists a
unique linear mapping $\varLambda_{\varPhi,Y} \colon \pcal_T(Y)
\to \scal(\dcal)$ such that
   \begin{align} \label{1maja}
\varLambda_{\varPhi,Y}(\hat s|_Y) = \varPhi(s), \quad s \in T.
   \end{align}
We say that the mapping $\varLambda_{\varPhi,Y}$ is {\em
completely positive} if for every integer $m \Ge 1$,
   \begin{align*}
\varLambda_{\varPhi,Y}^{(m)}(\mathrm{M}_\pluss^m(\pcal_T(Y)))
\subset \mathrm{M}_\pluss^m(\scal(\dcal)),
   \end{align*}
where $\varLambda_{\varPhi,Y}^{(m)} \colon
\mathrm{M}^m(\pcal_T(Y)) \to \mathrm M^m(\scal(\dcal))$ is a
linear mapping defined by
   \begin{align*}
\varLambda_{\varPhi,Y}^{(m)}([w_{k,l}]_{k,l=1}^m) =
[\varLambda_{\varPhi,Y}(w_{k,l})]_{k,l=1}^m, \quad
[w_{k,l}]_{k,l=1}^m \in \mathrm{M}^m(\pcal_T(Y)).
   \end{align*}
$\varLambda_{\varPhi,Y}$ is said to be {\em completely f-positive}
if for every integer $m \Ge 1$,
   \begin{align*}
\varLambda_{\varPhi,Y}^{(m)}(\mathrm{M}_{\mathrm f}^m(\pcal_T(Y)))
\subset \mathrm{M}_\pluss^m(\scal(\dcal)).
   \end{align*}
Apparently, complete positivity implies complete f-positivity. If
$Y=\xfr_\sfr$, we shall abbreviate $\varLambda_{\varPhi,Y}$
($\varLambda_{\varPhi,Y}^{(m)}$ respectively) to
$\varLambda_{\varPhi}$ ($\varLambda_{\varPhi}^{(m)}$
respectively).

We now show that the notion of complete f-positivity
of $\varLambda_{\varPhi,Y}$ does not depend on the
choice of determining set $Y$.
   \begin{pro} \label{row}
Suppose that $\sfr$ is a unital commutative
$*$-semigroup and $T$ is a nonempty subset of $\sfr$.
Let $\dcal$ be a complex linear space and
$\varPhi\colon T \to \scal(\dcal)$ be a mapping. If
$Y$ and $Z$ are determining subsets of $\xfr_\sfr$,
then $\varLambda_{\varPhi,Y}$ is completely f-positive
if and only if $\varLambda_{\varPhi,Z}$ is completely
f-positive.
   \end{pro}
   \begin{proof}
The proof reduces to the case $Z=\xfr_\sfr$. By \eqref{pimty} and
\eqref{1maja} we have
   \begin{align*}
\varLambda_{\varPhi,Y}^{(m)} \circ \pi_{T,Y}^{(m)} =
\varLambda_\varPhi^{(m)}, \quad m=1,2, \ldots,
   \end{align*}
 which, together with \eqref{pitym}, implies the desired equivalence.
   \end{proof}
   For a complex linear space $\dcal$, we denote by
$\varrho_\dcal$ the locally convex topology on
$\scal(\dcal)$ given by the system of seminorms $\phi
\mapsto |\phi(f,g)|$ indexed by all the pairs $(f,g)
\in \dcal \times \dcal$. Clearly, the topology
$\varrho_\dcal$ is nothing else than that of pointwise
convergence on $\dcal \times \dcal$, and therefore it
can be regarded as an analogue of the weak operator
topology on the set of bounded linear operators on a
Hilbert space.
   \begin{pro} \label{cont}
Assume that $\sfr$ is a unital commutative
$*$-semigroup, $T$ is a nonempty subset of $\sfr$ such
that $\llbra T \rrbra \subset T$, and $Y$ is a
determining subsets of $\xfr_\sfr$. Let $\dcal$ be a
complex linear space and $\varPhi\colon T \to
\scal(\dcal)$ be a mapping. Suppose that the mapping
$\varLambda_{\varPhi,Y} \colon (\pcal_T(Y),
\tau_{T,Y}) \to (\scal(\dcal), \varrho_\dcal)$ is
continuous. Then $\varLambda_{\varPhi,Y}$ is
completely f-positive if and only if
$\varLambda_{\varPhi,Y}$ is completely positive.
   \end{pro}
   \begin{proof}
One can check that the assumed continuity of
$\varLambda_{\varPhi,Y}$ implies that of
   \begin{align*}
\varLambda_{\varPhi,Y}^{(m)} \colon \big({\mathrm
M}^m(\pcal_T(Y)), \tau_{T,Y}^{(m)}\big) \to
\big({\mathrm M}^m(\scal(\dcal)),
\varrho_\dcal^{(m)}\big), \quad m =1, 2, \ldots
   \end{align*}
Hence, by the $\varrho_\dcal^{(m)}$-closedness of $\mathrm M_\pluss^m(\scal(\dcal))$
and Lemma \ref{closure}, we arrive at the desired conclusion.
   \end{proof}
Regarding Proposition \ref{cont}, note that if $T$ is
a $*$-subsemigroup of a unital commutative
$*$-semigroup $\sfr$ ($T$ need not be unital), then
$\llbra T \rrbra \subset T$. The reverse implication
is not true in general. In fact, it may happen that
$\llbra T \rrbra \subset T$ although $T$ is not a
subsemigroup of $\sfr$ and $T \neq T^*$ (thus neither
$T \subset \llbra T \rrbra$ nor $T^* \subset \llbra T
\rrbra$ holds). We leave it to the reader to verify
that this is the case for the subset
   \begin{align*}
T \okr \{(k,l)\in\nfr \colon k\Ge 1,\, l\Ge 1 \} \cup
T_0
   \end{align*}
of the $*$-semigroup $\nfr$ considered in Example
\ref{nag}, where $T_0$ is a proper subset of
$\{(k,0)\in\nfr \colon k\Ge 1\}$ containing $(1,0)$.
   \subsection{\label{cfe}Criteria for extendibility}
As mentioned in Introduction, not every positive
definite function on $\nfr$ extends to a positive
definite function on the $*$-semigroup $\nfr_\pluss$ (see Example \ref{nag} and Section \ref{s4} for the definitions). However, if we impose a stronger positivity-like condition on a
function defined on $\nfr$ (in the language of \cite{bis-01} this is
positive definiteness with respect to $\nfr_\pluss$),
then it is extendable to a positive definite function
on $\nfr_\pluss$, and reversely (cf.\ \cite{st-sz1}). The property of
extendibility of positive definite functions was
characterized likewise in \cite{st-sz1} also in the
case of $*$-subsemigroups of abstract (unital commutative) $*$-semigroups.
This led the authors of \cite{st-sz1} to find a new
solution of the complex moment problem (not to mention
other extension results). The key feature that made this approach successful was the semiperfectness of
$\nfr_\pluss$, the property guaranteeing that every
positive definite function on $\nfr_\pluss$ is a
moment function, i.e.\ it has the Laplace-type integral
representation on the dual $*$-semigroup of
$\nfr_\pluss$. Inspired by this, Bisgaard attached to
any $*$-semigroup $\sfr$ an enveloping
perfect\footnote{\;i.e.\ semiperfect with the
uniqueness of integral representation} $*$-semigroup
$\mathfrak Q$ such that the set of all moment
functions on $\sfr$ coincides with the set of all
functions which are positive definite with respect to
$\mathfrak Q$ (cf.\ \cite{bis-01}). The instance of
semiperfect (but not perfect) $*$-semigroup
$\nfr_\pluss$ as an extending $*$-semigroup for $\nfr$
shows that semiperfectness is sufficient as far as
moment problems are concerned. Though there is a
limited freedom of choice of an extending semiperfect
$*$-semigroup for a fixed $*$-semigroup, it can by no
means be chosen arbitrarily, as indicated in the
discussion concerning the inclusions (22) in
\cite{st-sz1}. In this section we are looking for criteria that guarantee positive definite extendibility of mappings defined on symmetric subsets of (operator) semiperfect $*$-semigroups. As a result, we obtain the characterizations of ``truncated '' moment functions enriching those in \cite{st-sz1, st-sz2} with the new positivity conditions of the Riesz-Haviland type.

Suppose that ${\sf M}$ is a $\sigma$-algebra of
subsets of a set $X\neq \varnothing$ and $\dcal$ is a
complex linear space. A mapping $\mu\colon{\sf M} \to
\scal(\dcal)$ is called a {\it semispectral
measure}\/\footnote{\;With the natural identification
of bounded linear operators with sesquilinear forms,
our definition subsumes the classical semispectral
operator-valued measures (cf.\ \cite{Ml}).} on ${\sf
M}$ if $\mu(\,\cdot\,)(f,f)$ is a finite positive
measure for every $f\in \dcal$.

Let $\sfr$ be
a unital commutative $*$-semigroup. Denote by ${\sf
M}_\sfr$ the smallest $\sigma$-algebra of subsets of
$\xfr_\sfr$ with respect to which all the transforms
$\hat s$, $s\in\sfr$, are measurable. Following
\cite{bis}, we say that $\sfr$ is {\it operator
semiperfect} if for any complex linear space $\dcal$
and for any positive definite mapping
$\varPsi\colon\sfr\to\scal(\dcal)$ there exists a
semispectral measure $\mu\colon {\sf M}_\sfr \to
\scal(\dcal)$ such that
   \begin{equation} \label{sigma}
   \varPsi(s)=\int_{\xfr_\sfr}\hat s(\chi)\, \mu(\D \chi),\quad
s\in\sfr.
   \end{equation}
This equality is to be understood in the following sense
   \begin{equation*}
   \varPsi(s)(f,g)=\int_{\xfr_\sfr}\hat s(\chi)\, \mu(\D \chi)(f,g),\quad
f,g \in \dcal, \; s\in\sfr;
   \end{equation*}
here and forth all the integrands are tacitly assumed to be
summable. An equivalent definition of operator semiperfectness may
be stated in a matrix-type form, as shown in \cite{bis}.

Note that if $\xfr_\sfr$ is equipped with the topology
of pointwise convergence on $\sfr$, then the
transforms $\hat s$, $s\in\sfr$, are continuous, and
consequently the $\sigma$-algebra ${\sf M}_\sfr$ is
contained in the $\sigma$-algebra of all Borel subsets
of $\xfr_\sfr$ (the equality does not hold in
general). It is also clear that if ${\sf M}$ is a
$\sigma$-algebra of subsets of $\xfr_\sfr$ such that
${\sf M}_\sfr \subset {\sf M}$ and $\mu$ is a
semispectral measure on ${\sf M}$ satisfying
\eqref{sigma}, then the restriction of $\mu$ to ${\sf
M}_\sfr$ satisfies \eqref{sigma} as well.
   \begin{lem} \label{wn1}
Let $\sfr$ be a unital commutative $*$-semigroup whose dual
$*$-semi\-group $\xfr_\sfr$ is determining and let $T$ be a
nonempty subset of $\sfr$. Suppose that a complex linear space
$\dcal$ and a mapping $\varPhi\colon T \to \scal(\dcal)$ are
given.
   \begin{enumerate}
   \item[(i)] If $\mu\colon{\sf M}\to\scal(\dcal)$ is a
semispectral measure on a $\sigma$-algebra ${\sf M}$ of subsets of
$\xfr_\sfr$, ${\sf M}_\sfr \subset {\sf M}$ and $Y$ is a
determining subset of $\xfr_\sfr$ such that $Y \in {\sf M}$ and
   \begin{align*}
\varPhi(s) = \int_Y \hat s \, \D \mu,\quad s \in T,
   \end{align*}
then $\varLambda_{\varPhi,Y}$ is completely positive.
   \item[(ii)] If $T=\sfr$ and $\varLambda_\varPhi$ is completely
f-positive, then $\varPhi$ is positive definite.
   \end{enumerate}
In particular, if $\sfr$ is operator semiperfect and $T=\sfr$,
then $\varPhi$ is positive definite if and only if
$\varLambda_\varPhi$ is completely f-positive or equivalently if
$\varLambda_\varPhi$ is completely positive.
   \end{lem}
   \begin{proof} (i) Take $[w_{k,l}]_{k,l=1}^m \in {\mathrm
M}_\pluss^m(\pcal_T(Y))$ and $e_1, \ldots, e_m \in \dcal$. Notice
that the complex measures $\mu(\,\cdot\,)(e_i,e_j)$, $i,j \in \{1,
\ldots, m\}$, are absolutely continuous with respect to some
finite positive measure $\nu$ on ${\sf M}$ (e.g.\
$\nu(\,\cdot\,)=\sum_{i=1}^m \mu(\,\cdot\,)(e_i,e_i)$). By the
Radon-Nikodym theorem, there exists a system
$\{h_{i,j}\}_{i,j=1}^m$ of ${\sf M}$-measurable complex functions
on $\xfr_\sfr$ such that $\mu(\sigma)(e_i,e_j) = \int_{\sigma}
h_{i,j} \, \D \nu$ for all $\sigma \in {\sf M}$. Since
$[\mu(\sigma)(e_i,e_j)]_{i,j=1}^m \Ge 0$ for every $\sigma \in
{\sf M}$, one can show (see the proof of \cite[Theorem 6.4]{Ml})
that there exists $Z \in {\sf M}$ such that $\nu(\xfr_\sfr
\setminus Z)=0$ and $[h_{k,l}(\chi)]_{k,l=1}^m \Ge 0$ for every
$\chi \in Z$. Let for $\chi \in Z$, $[a_{k,l}(\chi)]_{k,l=1}^m$ be
the square root of $[h_{k,l}(\chi)]_{k,l=1}^m$. Since
$[w_{k,l}(\chi)]_{k,l=1}^m \Ge 0$ for every $\chi \in Y$ and
$\nu(Y \setminus (Y \cap Z)) = 0$, we get
   \begin{align*}
\sum_{k,l=1}^m \varLambda_{\varPhi,Y}(w_{k,l}) (e_k,e_l) =
\int\limits_{Y\cap Z} \sum_{k,l=1}^m w_{k,l} h_{k,l} \, \D \nu =
\int\limits_{Y \cap Z} \sum_{j=1}^m \sum_{k,l=1}^m w_{k,l} a_{k,j}
\overline{a_{l,j}} \, \D \nu \Ge 0.
   \end{align*}

   (ii) Take finite sequences $s_1, \ldots, s_m \in
\sfr$ and $e_1, \ldots, e_m \in \dcal$. It is easily
seen that $[\widehat {s_k s_l^*}]_{k,l=1}^m \in
{\mathrm M}_{\mathrm f}^m(\pcal(\xfr_\sfr))$. By
complete f-positivity of $\varLambda_\varPhi$, we get
$\sum_{k,l=1}^m \varPhi(s_l^*s_k)(e_k,e_l) \Ge 0$,
which finishes the proof.
   \end{proof}
   We now deal with the question of when a
$\scal(\dcal)$-valued mapping extends from a subset $T$ of
$\sfr$ to a positive definite mapping on the whole of
$\sfr$. The following result is the main tool in our
considerations. It is proved as a consequence of Theorem 1
of \cite{st-sz2} (in fact, a prototype of this theorem
appeared in \cite{st-sz1}). Below, we interpret the
algebraic tensor product $\fcal_\sfr \otimes \ell^2$ as the
collection of all $\ell^2$-valued functions on $\sfr$
vanishing off finite sets. Revoking the definition of
$\sfr_{\mathfrak h}$ from page \pageref{hermi} it may be
convenient for further references to detach the following
condition
   \begin{equation} \label{sym_diag}
   \text{$T$ is a symmetric subset of $\sfr$ containing
$\sfr_{\mathfrak h}$.}
   \end{equation}
   \begin{lem} \label{main}
Suppose that $\sfr$ is a unital commutative $*$-semigroup,
$T$ satisfies \eqref{sym_diag} and $Y$ is a determining
subset of $\xfr_\sfr$. If $\dcal$ is a complex linear
space, then for every mapping $\varPhi\colon T \to
\scal(\dcal)$ the following three conditions are
equivalent{\em :}
   \begin{enumerate}
   \item[(i)] $\varPhi$ extends to a positive definite mapping
$\varPsi \colon \sfr \to \scal(\dcal)$,
   \item[(ii)] $\sum_{k,l = 1}^m \sum_{\substack {s,t \in
\sfr \\ t^* s \in T}} \varPhi (t^*s)(e_k, e_l) \, \is{\lambda_k
(s)} {\lambda_l (t)}_{\ell^2} \Ge 0$ for every integer $m \Ge 1$
and for all finite systems $\{e_n\}_{n=1}^m \subset \dcal$ and
$\{\lambda_n\}_{n=1}^m \subset \fcal_\sfr \otimes \ell^2$ such
that
   \begin{align} \label{21.11.2006+}
\sum\nolimits_{\substack {s,t \in \sfr \\ t^*s = u}} \is{\lambda_i
(s)} {\lambda_j (t)}_{\ell^2} = 0, \quad u \in \sfr \setminus T,
\, i,j = 1, \dots, m,
   \end{align}
   \item[(iii)] $\varLambda_{\varPhi,Y}$ is completely f-positive.
   \end{enumerate}
   \end{lem}
   \begin{proof}
(i)$\Leftrightarrow$(ii) The reader can convince himself that
condition (ii) of \cite[Theorem 1]{st-sz2} is equivalent to our
condition (ii); note that this can be shown directly, without
recourse to the proof given therein. Hence, our equivalence
(i)$\Leftrightarrow$(ii) is a consequence of \cite[Theorem
1]{st-sz2}.

(ii)$\Rightarrow$(iii) Fix an integer $m \Ge 1$. Take $w =
[\llangle p_k, p_l \rrangle]_{k,l=1}^m \in {\mathrm M}_{\mathrm
f}^m(\pcal_T(Y))$ with $p_1, \ldots, p_m \in \pcal(Y,\ell^2)$.
Then there exist $\lambda_1, \dots, \lambda_m \in \fcal_\sfr
\otimes \ell^2$ such that
   \begin{align} \label{lambdas}
p_k(\chi) = \sum_{s \in \sfr} \chi(s) \lambda_k(s), \quad \chi \in
Y, \, k=1, \ldots, m.
   \end{align}
It is clear that
   \begin{align} \label{21.11.2006}
\is{p_k(\chi)}{p_l(\chi)}_{\ell^2} = \sum_{u \in \sfr} \chi(u) \,
\sum_{\substack{s,t \in \sfr \\ t^*s = u}}
\is{\lambda_k(s)}{\lambda_l(t)}_{\ell^2}, \quad \chi \in Y, \, k,l
=1, \ldots, m.
   \end{align}
Since $w \in {\mathrm M}^m(\pcal_T(Y))$ and $Y$ is determining, we
see that \eqref{21.11.2006} implies \eqref{21.11.2006+}. It
follows from \eqref{1maja}, \eqref{21.11.2006+},
\eqref{21.11.2006} and (ii) that
   \allowdisplaybreaks
   \begin{align*}
\sum_{k,l=1}^m \varLambda_{\varPhi,Y}(\llangle p_k, p_l
\rrangle)(e_k,e_l) & = \sum_{k,l=1}^m \; \sum_{u \in T} \;
\sum_{\substack{s,t \in \sfr \\ t^*s=u}} \varPhi(u) (e_k,e_l)
\is{\lambda_k(s)}{\lambda_l(t)}_{\ell^2}
   \\
& = \sum_{k,l=1}^m \; \sum_{\substack{s,t \in \sfr \\
t^*s \in T}} \varPhi(t^*s) (e_k,e_l)
\is{\lambda_k(s)}{\lambda_l(t)}_{\ell^2} \Ge 0
   \end{align*}
for all vectors $e_1, \ldots, e_m \in \dcal$. This shows that
$\varLambda_{\varPhi,Y}$ is completely f-positive.

Reversing the above reasoning, we infer (ii) from (iii); to see
this, for fixed $\lambda_1, \ldots, \lambda_m \in \fcal_\sfr
\otimes \ell^2$ consider $p_1, \ldots, p_m \in \pcal(Y,\ell^2)$
defined by \eqref{lambdas}. This completes the proof.
   \end{proof}
   With the above discussions, we are in a position to state
the main result of the paper which supplies criteria for
extendibility to a positive definite function.
   \begin{thm} \label{main2}
Suppose that $\sfr$ is an operator semiperfect
$*$-semigroup, $T$ satisfies \eqref{sym_diag} and $Y$ is a
determining subset of $\xfr_\sfr$. If $\dcal$ is a complex
linear space, then for every mapping $\varPhi\colon T \to
\scal(\dcal)$ the following four conditions are
equivalent{\em :}
   \begin{enumerate}
   \item[(i)] $\varPhi$ extends to a positive definite mapping
$\varPsi \colon \sfr \to \scal(\dcal)$,
   \item[(ii)] $\varPhi(s) = \int_{\xfr_\sfr} \hat s \, \D \mu$
for all $s \in T$ with some semispectral measure $\mu$ on ${\sf
M}_\sfr$,
   \item[(iii)] $\varLambda_\varPhi$ is completely
positive\,\footnote{\;It follows from our assumptions that
$\xfr_\sfr$ is determining, which makes it legitimate to
consider $\varLambda_\varPhi$.},
   \item[(iv)] $\varLambda_{\varPhi,Y}$ is completely f-positive.
   \end{enumerate}
   \end{thm}
   \begin{proof}
(i)$\Rightarrow$(ii) Use operator semiperfectness of $\sfr$.

(ii)$\Rightarrow$(i) Since $|\hat s|^2 = \widehat{s^*s}$ and $s^*s
\in T$ for every $s\in \sfr$, we see that the function $\hat s$ is
square summable with respect to $\mu$ for every $s\in \sfr$. This
enables us to define the mapping $\varPsi \colon \sfr \to
\scal(\dcal)$ by
   \begin{align*}
\varPsi(s) = \int_{\xfr_\sfr} \hat s \, \D \mu, \quad s \in \sfr,
   \end{align*}
It follows from Lemma \ref{wn1} that $\varPsi$ is a positive
definite extension of $\varPhi$.

(i)$\Leftrightarrow$(iv) Apply Lemma \ref{main}.

(ii)$\Rightarrow$(iii) This is a consequence of Lemma \ref{wn1}.

(iii)$\Rightarrow$(iv) Since (iii) implies (iv) with $Y =
\xfr_\sfr$, we see that $\varLambda_\varPhi$ is completely
f-positive. An application of Proposition \ref{row} guarantees
that $\varLambda_{\varPhi,Y}$ is completely f-positive as well.
This completes the proof.
   \end{proof}
We now turn to the the case of scalar functions. A
unital commutative $*$-semigroup $\sfr$ is called {\it
semiperfect} (cf.\ \cite{bis}) if for any positive
definite function $\varPsi\colon\sfr\to\cbb$ there
exists a finite positive measure $\mu$ on ${\sf M}_\sfr$ such
that
   \begin{equation*}
   \varPsi(s)=\int_{\xfr_\sfr}\hat s(\chi)\, \mu(\D \chi),\quad
s\in\sfr.
   \end{equation*}
Evidently, operator
semiperfectness implies semiperfectness but not conversely as indicated by Bisgaard in
\cite{bis2}. The following is a scalar counterpart of Theorem
\ref{main2}.
   \begin{thm} \label{main2+}
Suppose that $\sfr$ is a semiperfect $*$-semigroup, $T$
satisfies \eqref{sym_diag} and $Y$ is a determining subset
of $\xfr_\sfr$. Then for every function $\varphi\colon T
\to \cbb$ the following four conditions are equivalent{\em
:}
   \begin{enumerate}
   \item[(i)] $\varphi$ extends to a positive definite function
$\psi \colon \sfr \to \cbb$,
   \item[(ii)] $\varphi(s) = \int_{\xfr_\sfr} \hat s \, \D \mu$
for all $s \in T$ with some positive measure $\mu$ on ${\sf
M}_\sfr$,
   \item[(iii)] $\varLambda_{\varphi}(p) \Ge 0$ for every $p \in
\pcal_T(\xfr_\sfr)$ such that $p(\chi) \Ge 0$ for all $\chi \in
\xfr_\sfr$,
   \item[(iv)] $\varLambda_{\varphi,Y}(p) \Ge 0$ for every
$p \in \pcal_T(Y)$ for which there exist finitely many functions $q_1,
\ldots, q_n \in \pcal(Y)$ such that $p(\chi) = \sum_{j=1}^n
|q_j(\chi)|^2$, $\chi \in Y$.
   \end{enumerate}
   \end{thm}
   \begin{proof}
Note first that the functional $\varLambda_\varphi$ is completely
positive if and only if (iii) holds. Indeed, if (iii) holds and
$w=[w_{i,j}]_{i,j=1}^m \in {\mathrm M}_\pluss^m(\pcal_T(\xfr_\sfr))$, then
for every $\lambda = (\lambda_1, \ldots, \lambda_m) \in \cbb^m$,
the function $p_\lambda \okr \sum_{i,j=1}^m \lambda_i
\overline{\lambda_j} w_{i,j} \in \pcal_T(\xfr_\sfr)$ is nonnegative on
$\xfr_\sfr$, and consequently
   \begin{align*}
\sum_{i,j=1}^m \lambda_i \overline{\lambda_j} \varLambda_\varphi
(w_{i,j}) = \varLambda_\varphi(p_\lambda) \Ge 0.
   \end{align*}
This means that $[\varLambda_\varphi(w_{i,j})]_{i,j=1}^m
\Ge 0$. The reverse implication is obvious.

The next observation is that the functional
$\varLambda_{\varphi,Y}$ is completely f-positive if and
only if (iv) holds. Indeed, if (iv) holds and $w =
[\llangle p_k, p_l \rrangle]_{k,l=1}^m \in {\mathrm
M}_{\mathrm f}^m(\pcal_T(Y))$ with some $p_1, \ldots, p_m
\in \pcal(Y,\ell^2)$, then for every $\lambda = (\lambda_1,
\ldots, \lambda_m) \in \cbb^m$, the function $q_\lambda
\okr \sum_{i=1}^m \lambda_i p_i$ belongs to
$\pcal(Y,\ell^2)$ whereas $\llangle q_\lambda, q_\lambda
\rrangle $ is in $\pcal_T(Y)$. Hence, by Lemma
\ref{factor}, we have
   \begin{align*}
   \sum_{i,j=1}^m \lambda_i \overline{\lambda_j}
\varLambda_{\varphi,Y} (\llangle p_i, p_j \rrangle) =
\varLambda_{\varphi, Y} (\llangle q_\lambda, q_\lambda \rrangle) \Ge 0,
   \end{align*}
which means that $[\varLambda_{\varphi,Y} (\llangle p_i, p_j
\rrangle)]_{i,j=1}^m \Ge 0$. The reverse implication is
plain.

The above enables us to adapt the proof of Theorem \ref{main2} to the present context.
   \end{proof}
   \section*{\sc Applications}
   \subsection{\label{s4}The truncated complex moment problem}
The $*$-semigroup $\nfr_\pluss$ we intend to
investigate comes from \cite{st-sz1}. It plays a
crucial role in the complex moment problem. The
initial part of this section is devoted to a
description of $\pcal_T(Y)$ in this particular case.

Denote by $\nfr_\pluss$ the $*$-semigroup $(\{(m,n) \in
\zbb \times \zbb \colon m+n \Ge 0\}, +, *)$ with
coordinatewise defined addition as semigroup operation and
involution $(m,n)^* = (n,m)$. Owing to \cite[Remark
7]{st-sz1}, the dual $*$-semigroup $\xfr_{\nfr_+}$ can be
identified algebraically with the $*$-subsemigroup
$(\varOmega \cup (\nul \times \mathbb T), \cdot,*)$ of the
product $*$-semigroup $(\cbb \times \mathbb T, \cdot, *)$
(with coordinatewise defined multiplication as semigroup
operation and involution $(z,w)^* = (\bar z, \bar w)$),
where
   \begin{align*}
\varOmega \okr \{(z, z \, {\bar z}^{\,-1}) \colon z \in
\cbb_*\}.
   \end{align*}
Under this identification, $\widehat{(m,n)}$ is given by
   \begin{align} \label{crucial}
\widehat{(m,n)}(z,w) = \chi_{z,w}(m,n) =
   \begin{cases}
   z^m \bar z^n & \text{if } z\neq0, \\
   0 & \text{if $z=0$ and $m+n>0$,} \\
   w^m & \text{if $z=0$ and $m+n=0,$}
   \end{cases}
   \end{align}
for all $(z,w) \in \varOmega \cup (\nul \times \mathbb T)$
and $(m,n) \in \nfr_\pluss$.
Given a subset $Z$ of $\cbb_*$, we write
   \begin{align} \label{yz}
Y_Z = \{(z,z \, {\bar z}^{\,-1}) \colon z \in Z\} \subset \varOmega.
   \end{align}
By the above identification, we see that if $T$ is a
nonempty subset of $\nfr_\pluss$, then $\pcal_T (Y_Z)$ may
be regarded as the set of all rational functions $p$ on $Z$
of the form
   \begin{align} \label{ratf}
p(z,\bar z)=\sum_{(m,n) \in T} a_{m,n} z^m \bar z^n, \quad
z \in Z,
   \end{align}
where $\{a_{m,n}\}_{(m,n) \in T} \subset \cbb$ is a finite
system.

For a subset $T$ of $\zbb \times \zbb$, we denote by
$\cbb_T(z,\bar z)$ the set of all rational functions $p$ of
the form \eqref{ratf} with $Z=\cbb_*$ $(\{a_{m,n}\}_{(m,n)
\in T} \subset \cbb$ is a finite system$)$. If $T \subset
\nfr$, where $\nfr$ is as in Example \ref{nag}, then
$\cbb_T(z,\bar z)$ coincides with the set $\cbb_T[z,\bar
z]$ of all complex polynomials in $z$ and $\bar z$ whose
coefficients vanish for all indices in $\nfr \setminus T$.
If $T = \nfr$, then we abbreviate $\cbb_T[z,\bar z]$ to
$\cbb[z,\bar z]$. It is worth noticing that members of
$\cbb_{\nfr_+}(z,\bar z)$ are functions, which are bounded
on every punctured disc $\{ z\in \cbb \colon 0 < |z| \Le
R\}$, $R > 0$. This is due to the fact that all terms
$z^m \bar z^n$ with $m+n \Ge 0$ share this property. As a
consequence, the set $\cbb_{\nfr_+}(z,\bar z)$ is
essentially smaller than $\cbb_{\zbb \times \zbb}(z,\bar
z)$. In fact, members of $\cbb_{\nfr_+}(z,\bar z)$ can be
characterized by their boundedness on the fixed punctured
disc $\{ z\in \cbb \colon 0 < |z| \Le 1\}$.
   \begin{pro}\label{ogrwym}
If $p \in \cbb_{\zbb \times \zbb}(z, \bar z)$, then the
following conditions are equivalent
   \begin{enumerate}
   \item[(i)] $p \in \cbb_{\nfr_+}(z,\bar z)$,
   \item[(ii)] $\sup \{|p(z,\bar z)|\colon z \in \cbb, \,
0 < |z| \Le 1\} < \infty$.
   \end{enumerate}
   \end{pro}
   \begin{proof}
Since the implication (i)$\Rightarrow$(ii) has been
clarified above, we can focus on the implication
(ii)$\Rightarrow$(i). Assume that (ii) holds and $p$ is as in
\eqref{ratf} with $T=\zbb \times \zbb$ and $Z =
\cbb_*$. We may write the rational function $p$ in the form
$p = p_0 + \sum_{j=1}^k p_j$, where $k$ is a positive
integer, $p_0 \in \cbb_{\nfr_+}(z,\bar z)$ and
   \begin{align} \label{ladroz}
p_j(z, \bar z) = \sum_{\substack{m,n \in \zbb\\ m+n=-j}}
a_{m,n} z^m\bar z^n = |z|^{-j}\sum_{\substack{m,n \in \zbb\\
m+n=-j}} a_{m,n} \left[\frac z{|z|}\right]^m \left[\frac
{\bar z}{|z|}\right]^n
   \end{align}
for $z\in \cbb_*$ and $j=1, \ldots,k$. Then
   \begin{align*}
|z|^{k-1} p_k(z,\bar z) = |z|^{k-1}\Big( p(z,\bar z) -
p_0(z,\bar z) - \sum_{j=1}^{k-1} p_j(z, \bar z) \Big),
\quad z\in \cbb_*,
   \end{align*}
which, together with (ii) and \eqref{ladroz}, implies that
the function $z \mapsto |z|^{k-1} p_k(z,\bar z)$ is bounded
on $\{z \in \cbb \colon 0 < |z| \Le 1\}$. Substituting
$z=r\E^{\I\theta}$ with real $\theta$ and $r>0$, and using
\eqref{ladroz}, we deduce that $p_k(z,\bar z)=0$ for all $z
\in \cbb_*$. Repeating this argument, we see that
$p_1(z,\bar z)=\ldots=p_k(z,\bar z)=0$ for all $z \in
\cbb_*$. By Lemma~ \ref{1}, this completes the proof.
   \end{proof}
Given $T \subset \zbb \times \zbb$, we say that a subset
$Z$ of $\cbb_*$ is {\em determining} for $\cbb_T(z,\bar z)$
(or shorter:\ {\em $\cbb_T(z,\bar z)$-determining}) if
every rational function $p \in \cbb_T(z,\bar z)$ vanishing
on $Z$ (i.e.\ $p(z,\bar z) = 0$ for all $z\in Z$) vanishes
on $\cbb_*$. By Lemma \ref{1}, the set $Z$ is
$\cbb_T(z,\bar z)$-determining if and only if the system
of functions $z \mapsto z^m \bar z^n$, $(m,n) \in T$, is linearly independent in $\cbb^Z$.
If $T \subset \zbb_\pluss \times \zbb_\pluss$,
then we allow $0$ to be a member of a determining set.
Clearly, if $T \subset \zbb \times \zbb$ and $Z \subset
\cbb_*$ is a determining set for $\cbb_{\zbb\times\zbb}
(z,\bar z)$, then $Z$ is determining for $\cbb_T (z,\bar
z)$. The reverse is not true, e.g.\ for $T=\{(n,n) \colon n
\in \zbb_\pluss\}$ and $Z = \{z \in \cbb_*\colon z= \bar
z\}$. Lemma \ref{1} provides examples of
$\cbb_{\zbb\times\zbb} (z,\bar z)$-determining sets. A particular class of them is indicated below.
   \begin{lem} \label{niepust}
If the interior of a set $Z \subset \cbb_*$ is not empty,
then $Z$ is $\cbb_{\zbb\times\zbb} (z,\bar z)$-determining.
   \end{lem}
The following simple fact is crucial for further
investigations. It enables us to deal with the complex
plane instead of a larger and less handy dual $*$-semigroup
$\xfr_{\nfr_+}$.
    \begin{lem} \label{fdesr}
A subset $Z$ of $\cbb_*$ is determining for
$\cbb_{\nfr_+}(z,\bar z)$ if and only if the set $Y_Z$
defined in \eqref{yz} is determining for
$\pcal(\xfr_{\nfr_+})$.
    \end{lem}
    \begin{proof}
Apply Lemma \ref{1} and the description \eqref{crucial} of $\widehat
{\nfr_+}$.
    \end{proof}
The next lemma guarantees that a subset $Z$ of $\cbb$ is
determining for $\cbb[z,\bar z]$ if and only if $Z
\setminus \nul$ is determining for $\cbb_{\nfr_+}(z,\bar
z)$.
   \begin{lem} \label{determ}
If $Z$ is a subset of $\cbb $, then the following
conditions are equivalent
   \begin{enumerate}
   \item[(i)] $Z$ is determining for $\cbb[z,\bar z]$,
   \item[(ii)] $Z\setminus Z_0$ is determining
for $\cbb[z,\bar z]$ whenever $Z_0 \subset \cbb$ is
finite,
   \item[(iii)] $Z\setminus (Z_0 \cup \nul)$ is determining
for $\cbb_{\zbb\times \zbb}(z,\bar z)$ whenever $Z_0
\subset \cbb$ is finite.
   \end{enumerate}
   \end{lem}
   \begin{proof}
(i)$\Rightarrow$(iii) Suppose that $Z_0 = \{\lambda_1,
\ldots, \lambda_k\}$ and take $p \in \cbb_{\zbb \times
\zbb}(z,\bar z)$ such that $p(z,\bar z)= 0$ for all $z\in
Z\setminus (Z_0 \cup \nul)$. Then for a sufficiently large
positive integer $N$ the rational function $q(z, \bar z)
\okr z^N \bar z^N (z-\lambda_1) \ldots
(z-\lambda_k)p(z,\bar z)$ is a polynomial in $z$ and $\bar
z$ such that $q(z, \bar z) =0$ for all $z\in Z$. By (i),
$q(z,\bar z) = 0$ for all $z\in \cbb$, and so $p(z, \bar z)
= 0$ for all $z \in \cbb \setminus (Z_0 \cup \nul)$. As a
consequence, $p(z, \bar z) = 0$ for all $z \in \cbb_*$.

Implications (iii)$\Rightarrow$(ii) and
(ii)$\Rightarrow$(i) are evident.
   \end{proof}
   A sequence $\{c_{m,n}\}_{m,n=0}^\infty \subset \cbb$ is
called a {\em complex moment sequence} if there exists a
positive Borel measure $\mu$ on $\cbb$ such that
   \begin{align*}
c_{m,n} = \int_\cbb z^m \bar z^n \mu(\D z), \quad m,n \Ge
0.
   \end{align*}
Such a measure $\mu$ is called a {\em representing
measure} of $\{c_{m,n}\}_{m,n=0}^\infty$; it is by no
means unique. If a representing measure $\mu$ is
unique, then $\{c_{m,n}\}_{m,n=0}^\infty$ is called a
{\em determinate} complex moment sequence. For this
and related questions we refer the reader to
\cite{sh-t}, \cite{fug} and \cite{st-sz1}.

   The following result is an extension of Theorem 1 of
\cite{st-sz1} as well as of the complex version of the
Riesz-Haviland theorem (see Theorem B). The first of
the two theorems can be seen as the equivalence
(i)$\Leftrightarrow$(iii) below with $T = \zbb_\pluss \times
\zbb_\pluss$, while the other as the equivalence
(i)$\Leftrightarrow$(iv) with the same $T$.
   \vspace{2ex}
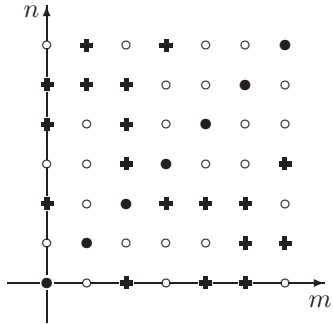
   \begin{center}
\unitlength=1.5pt
\begin{picture}(80,80)
\put(0,10){\line(1,0){8}}
\put(12,10){\line(1,0){7}}
\put(21,10){\line(1,0){7}}
\put(32,10){\line(1,0){7}}
\put(41,10){\line(1,0){7}}
\put(52,10){\line(1,0){6}}
\put(62,10){\line(1,0){7}}
\put(71,10){\vector(1,0){9}}

\put(10,0){\line(0,1){8}}
\put(10,12){\line(0,1){7}}
\put(10,21){\line(0,1){7}}
\put(10,32){\line(0,1){7}}
\put(10,41){\line(0,1){7}}
\put(10,52){\line(0,1){6}}
\put(10,62){\line(0,1){7}}
\put(10,71){\vector(0,1){9}}

\put(76,4){\text{$m$}} \put(4,77){\text{$n$}}

\put(10,70){\circle{2}}
\put(20,70){\makebox(0,0)[cc]{\pplus}}
\put(30,70){\circle{2}}
\put(40,70){\makebox(0,0)[cc]{\pplus}}
\put(50,70){\circle{2}} \put(60,70){\circle{2}}
\put(70,70){\circle*{2.5}}

\put(10,60){\makebox(0,0)[cc]{\pplus}}
\put(20,60){\makebox(0,0)[cc]{\pplus}}
\put(30,60){\makebox(0,0)[cc]{\pplus}}
\put(40,60){\circle{2}} \put(60,60){\circle*{2.5}}
\put(50,60){\circle{2}} \put(70,60){\circle{2}}

\put(10,50){\makebox(0,0)[cc]{\pplus}}
\put(20,50){\circle{2}}
\put(30,50){\makebox(0,0)[cc]{\pplus}}
\put(40,50){\circle{2}} \put(50,50){\circle*{2.5}}
\put(60,50){\circle{2}} \put(70,50){\circle{2}}

\put(10,40){\circle{2}} \put(20,40){\circle{2}}
\put(30,40){\makebox(0,0)[cc]{\pplus}}
\put(40,40){\circle*{2.5}} \put(50,40){\circle{2}}
\put(60,40){\circle{2}}
\put(70,40){\makebox(0,0)[cc]{\pplus}}

\put(10,30){\makebox(0,0)[cc]{\pplus}}
\put(20,30){\circle{2}} \put(30,30){\circle*{2.5}}
\put(40,30){\makebox(0,0)[cc]{\pplus}}
\put(50,30){\makebox(0,0)[cc]{\pplus}}
\put(60,30){\makebox(0,0)[cc]{\pplus}}
\put(70,30){\circle{2}}

\put(10,20){\circle{2}} \put(20,20){\circle*{2.5}}
\put(30,20){\circle{2}} \put(40,20){\circle{2}}
\put(50,20){\circle{2}}
\put(60,20){\makebox(0,0)[cc]{\pplus}}
\put(70,20){\makebox(0,0)[cc]{\pplus}}

\put(10,10){\circle*{2.5}} \put(20,10){\circle{2}}
\put(30,10){\makebox(0,0)[cc]{\pplus}}
\put(40,10){\circle{2}}
\put(50,10){\makebox(0,0)[cc]{\pplus}}
\put(60,10){\makebox(0,0)[cc]{\pplus}}
\put(70,10){\circle{2}}
\end{picture}
   \\[1.5ex]
   {\small Figure 1.\ $T$ appearing in Theorem \ref{zespmom} -- an
example:\
    \\
{\large $\bullet$} - obligatory data, $\pplus$ - additional
data, {\footnotesize $\boldsymbol{\circ}$} - missing data.}
   \end{center}
   \vspace{1ex}
   \begin{thm} \label{zespmom}
Let $T$ be a symmetric subset of $\zbb_\pluss \times
\zbb_\pluss$ $($i.e.\ $(n,m)\in T$ for all $(m,n) \in T$$)$
containing the diagonal $\{(n,n)\colon n \in
\zbb_\pluss\}$, and let $Z\subset \cbb$ be a determining
set for\/ $\cbb[z,\bar z]$. Then for any system of complex
numbers $\{c_{m,n}\}_{(m,n)\in T}$, the following
conditions are equivalent{\em :}
   \begin{enumerate}
   \item[(i)] there exists a positive Borel measure $\mu$ on
$\cbb$ such that
   \begin{align*}
   c_{m,n} = \int_\cbb z^m \bar z^n \mu(\D z), \quad (m,n)
\in T,
   \end{align*}
   \item[(ii)] there exists a complex moment sequence
$\{\tilde c_{m,n}\}_{m,n = 0}^\infty$ such that
$\tilde c_{m,n} = c_{m,n}$ for all $(m,n) \in T$,
   \item[(iii)] there exists\,\footnote{\;Notation $m+n\Ge 0$
has to be understood as $(m,n) \in \zbb \times \zbb$ and $m+n \Ge
0$.} $\{{\tilde c}_{m,n}\}_{m+n \Ge 0} \subset \cbb$ such that
${\tilde c}_{m,n}=c_{m,n}$ for all $(m,n) \in T$, and
$\sum_{\substack {m+n\Ge 0\\p+q\Ge 0}} \tilde c_{m+q,n+p}
\lambda_{m,n} \bar \lambda_{p,q} \Ge 0$ for all finite systems
$\{\lambda_{m,n}\}_{m+n\Ge 0} \subset \cbb$,
   \item[(iv)] $\sum_{(m,n)\in T} p_{m,n} c_{m,n}
\Ge 0$ for every finite system $\{ p_{m,n} \}_{(m,n)
\in T} \subset \cbb$ such that $\sum_{(m,n)\in T}
p_{m,n} z^m \bar z^n \Ge 0$ for all $z \in \cbb$,
   \item[(v)] $\sum_{(m,n)\in T} p_{m,n} c_{m,n}
\Ge 0$ for every finite system $\{ p_{m,n} \}_{(m,n) \in T} \subset
\cbb$ for which there exist finitely many rational functions $q_1,
\ldots, q_k \in \cbb_{\nfr_+}(z,\bar z)$ such that $\sum_{(m,n)\in
T} p_{m,n} z^m \bar z^n = \sum_{j=1}^k |q_j(z, \bar z)|^2$ for all
$z \in \cbb$, $z\neq 0$,
   \item[(vi)] $\sum_{(m,n)\in T} p_{m,n} c_{m,n}
\Ge 0$ for every finite system $\{ p_{m,n} \}_{(m,n) \in T} \subset
\cbb$ for which there exist finitely many rational functions $q_1,
\ldots, q_k \in \cbb_{\nfr_+}(z,\bar z)$ such that $\sum_{(m,n)\in
T} p_{m,n} z^m \bar z^n = \sum_{j=1}^k |q_j(z, \bar z)|^2$ for all
$z \in Z$, $z\neq 0$.
   \end{enumerate}
   \end{thm}
   \begin{proof}
   (i)$\Rightarrow$(ii) Since $\{(n,n)\colon n \in
\zbb_\pluss\} \subset T$, we see that complex
polynomials in $z$ and $\bar z$ are summable with
respect to $\mu$. Hence, we can define ${\tilde
c}_{m,n} = \int _\cbb z^m \bar z^n \mu(\D z)$ for
$m,n \in \zbb_\pluss$, which is the desired complex
moment sequence.

   (ii)$\Leftrightarrow$(iii) This can be deduced from
\cite[Theorem 1]{st-sz1}.

   Implications (ii)$\Rightarrow$(i), (i)$\Rightarrow$(iv) and
(iv)$\Rightarrow$(v) are evident.

   (v)$\Rightarrow$(vi) By Lemma \ref{determ}, $Z \setminus
\nul$ is a determining set for $\cbb_{\nfr_+}(z,\bar z)$.
Since both sides of the equality in (vi) are members of
$\cbb_{\nfr_+}(z,\bar z)$, and they coincide on $Z
\setminus \nul$, we deduce that they are equal on $\cbb_*$.

   (vi)$\Rightarrow$(iii) In view of Lemmata \ref{fdesr}
and \ref{determ}, $Y_{Z\setminus \nul}$ is a determining
set for $\pcal(\xfr_{\nfr_+})$. Since $\nfr_\pluss$ is a
semiperfect $*$-semigroup (this can be deduced
from\footnote{\;\label{fut}Indeed, in view of the
discussion preceding Lemma \ref{wn1}, we see that our
meaning of semiperfectness is wider than that of
\cite[Remark 7]{st-sz1}. The interested reader can verify
that for finitely generated unital commutative $*$-semigroups,
like $\nfr_\pluss$, both these notions are equivalent.}
\cite[Remark 7]{st-sz1}), we can apply implication
(iv)$\Rightarrow$(i) of Theorem \ref{main2+} to
$\sfr=\nfr_\pluss$, $Y=Y_{Z\setminus \nul}$ and
$\varphi(m,n)=c_{m,n}$. This completes the proof.
   \end{proof}

Let us point out the difference between the usual meaning of the truncated (complex) moment problem (where a finite system of complex numbers is given) and that appearing in Theorem \ref{zespmom}. One of our
assumptions requires for a system of complex numbers
$\{c_{m,n}\}_{(m,n)\in T}$, which is the given data,
to include all the diagonal entries
$c_{m,m}$, $m=0,1,2, \ldots$ This enabled us to show
that the most direct analogue of the complex version
of the Riesz-Haviland theorem for
the truncated moment problem in our meaning is true. In
a recent paper \cite[Example 2.1]{cu-fi}, Curto and
Fialkow have shown that this is not the case for the
truncated moment problem in the usual sense. Instead,
they have found an analogue of the Riesz-Haviland
theorem, which requires extending appropriate linear
functionals on polynomials of degree limited by $2n$ to positive linear functionals on polynomials of degree
limited by $2n+2$ (cf.\ \cite[Theorem
2.2.]{cu-fi}).
   \vspace{1ex}
   \subsubsection{\bf The case of $(\nfr_\pluss)_{\mathfrak h} \not\subset T$}
In this subsection we show that the assumption that
$T$ contains the diagonal $\{(n,n)\colon n \in
\zbb_\pluss\}$ cannot be dismissed without destroying
the equivalences (i)$\Leftrightarrow$(ii) as well as
(i)$\Leftrightarrow$(iv) of Theorem \ref{zespmom}.
   \begin{exa}\label{dziury}
Clearly, the implication (ii)$\Rightarrow$(i) of Theorem
\ref{zespmom} holds for an arbitrary subset $T$ of
$\zbb_\pluss \times \zbb_\pluss$. The reverse implication
does not hold in general even if $T$ contains all but a finite number of the diagonal elements
(hence all but a finite number of the moments exist). This can be shown for any subset $T$ of
$\zbb_\pluss \times \zbb_\pluss$ such that $(0,0) \notin T$
and
   \begin{align*}
\{(m,n) \in \zbb_\pluss \times \zbb_\pluss \colon m,n\Ge
k\} \subset T
   \end{align*}
with some integer $k \Ge 1$. For this, define the positive
Borel measure $\mu$ on $\cbb$ via $\D\mu(z) = |z|^{-2}
\eta(z) \D V(z)$, where $V$ stands for the planar Lebesgue
measure and $\eta$ is the characteristic function of the
disc $\varDelta \okr {\{z\in \cbb \colon |z|\Le 1\}}$.
Since $(0,0) \not \in T$, the system $c_{m,n} \okr
\int_\cbb z^m \bar z^n \D \mu(z)$, $(m,n) \in T$, is well
defined and fulfils the condition (i) of Theorem
\ref{zespmom}. Suppose that, contrary to our claim, it
satisfies the condition (ii) of Theorem \ref{zespmom}. Then
there exists a finite positive Borel measure $\nu$ on
$\cbb$ such that $c_{m,n} = \int_\cbb z^m \bar z^n
\D \nu(z)$ for all $(m,n) \in T$. Hence, we have
   \begin{align*}
\int_\varDelta z^m \bar z^n |z|^{2k}\D V(z) = \int_\cbb z^m
\bar z^n |z|^{2k+2} \D \nu(z), \quad m,n \in \zbb_\pluss.
   \end{align*}
Since the left-hand side represents a determinate complex
moment sequence indexed by $(m,n)$ (as it has a compactly
supported representing measure), we see that
   \begin{align*}
|z|^{2k} \eta(z) \D V(z) = |z|^{2k+2} \D \nu(z).
   \end{align*}
It follows that $\mu (\cbb_*) = \nu (\cbb_*) < \infty$,
which is a contradiction because $\mu$ is not a finite
measure. This proves our claim.
   \end{exa}
In view of \cite[Theorem 1]{st-sz1}, the extension
procedure required in the condition (ii) of Theorem
\ref{zespmom} can be realized in two steps: first we have
to extend the system $\{c_{m,n}\}_{(m,n)\in T}$ to a
positive definite system over the $*$-semigroup $\nfr$, and
then to a positive definite system over $\nfr_\pluss$. The
first step can be done in a more explicit way for
$*$-ideals $T$ of $\nfr$ (i.e.\ $T=T^*$ and $T + \nfr
\subset T$) as in \cite{fhsz}; see also
\cite{li-ma,fhsz2,js} for earlier attempts in this
direction. However, the $*$-ideal technique is not
applicable in the other step when extending positive
definite functions from $\nfr$ to $\nfr_\pluss$. This
situation requires methods invented in \cite{st-sz1}.

We next consider the case of the equivalence
(i)$\Leftrightarrow$(iv) of Theorem \ref{zespmom}.
   \begin{exa}
It is evident that the implication (i)$\Rightarrow$(iv) of
Theorem \ref{zespmom} holds for sets $T$ not necessarily
containing the diagonal $\{(m,m) \colon m \in
\zbb_\pluss\}$. The reverse implication does not hold in
general, which can be shown for all sets $T$ such that
   \begin{align}\label{kael}
   \{(k,k),(l,l)\} \subset T \subset \{(m,n) \in
   \zbb_\pluss \times \zbb_\pluss \colon m,n \Ge k\},
   \end{align}
where $k, l$ are integers such that $1 \Le k < l$. Indeed,
consider the system $\boldsymbol c = \{ \delta_{0,
m+n-2k}\}_{(m,n) \in T}$, where $\delta_{i,j}$ is the
Kronecker delta function. It can be readily checked that
$\boldsymbol c$ satisfies the condition (iv) of Theorem
\ref{zespmom} (see the limit formula in the proof of (b)$\Rightarrow$(a) of
Proposition \ref{havi}). Suppose that, contrary to our
claim, it satisfies the condition (i) of Theorem
\ref{zespmom} with some positive Borel measure $\mu$ on
$\cbb$. Then $1 = \delta_{0,0} = \int_\cbb |z|^{2k}
\D\mu(z)$ and $0 = \delta_{0,2(l-k)} = \int_\cbb |z|^{2l}
\D \mu(z)$. The latter equality implies that $\mu$ is
supported in $\nul$, which contradicts the former (because
$k \Ge 1$). In the extremal case $T$ may consist of only
two elements, which is the smallest possible number
required for the above argument because for one point sets
$T$ of the form $\{(k,k)\}$ with $k \in \zbb_\pluss$ the
equivalence (i)$\Leftrightarrow$(iv) of Theorem
\ref{zespmom} is valid. Note also that if $T=
\{(0,0),(l,l)\}$ with $l \Ge 1$, then the system
$\{c_{m,n}\}_{(m,n) \in T} \subset \cbb$ given by $c_{0,0}
= 0$ and $c_{l,l} = 1$ satisfies the condition (iv) of
Theorem \ref{zespmom} but not (i).
   \end{exa}

The ensuing proposition shows that for a subclass of sets
$T$ satisfying \eqref{kael} the condition (iv) of Theorem
\ref{zespmom} leads to a representation similar (but not
equivalent if $k\Ge 1$) to that in (i) of Theorem \ref{zespmom}.
Proposition \ref{havi} is somehow in the flavour of
\cite{fhsz1} where backward extensions of moment sequences
are considered.
   \begin{pro}\label{havi}
Let $T$ be a symmetric subset of $\zbb_\pluss \times
\zbb_\pluss$ such that
   \begin{align}\label{transT}
   \{(m,m) \in \zbb_\pluss \times \zbb_\pluss \colon m \Ge
k\} \subset T \subset \{(m,n) \in \zbb_\pluss \times
\zbb_\pluss \colon m,n \Ge k\}
   \end{align}
with some integer $k \Ge 0$, and let $\{c_{m,n}\}_{(m,n)
\in T}$ be a system of complex numbers. Then the following
conditions are equivalent{\em :}
   \begin{enumerate}
   \item[(a)] $\sum_{(m,n)\in T} p_{m,n} c_{m,n}
\Ge 0$ for every finite system $\{ p_{m,n} \}_{(m,n) \in T}
\subset \cbb$ such that $\sum_{(m,n)\in T} p_{m,n} z^m \bar
z^n \Ge 0$ for all $z \in \cbb$,
   \item[(b)] there exist a positive Borel measure $\mu$ on $\cbb$ and a real
$a\Ge 0$ such that $\mu(\nul) = 0$ and
   \begin{align}\label{delta0}
c_{m,n} = \int_\cbb z^m \bar z^n \D\mu(z) + a
\delta_{0,m+n-2k}, \quad (m,n) \in T.
   \end{align}
   \end{enumerate}
In particular, $\int_\cbb |z|^{2k} \D\mu(z) < \infty$ for
the measure $\mu$ appearing in {\em (b)}.
   \end{pro}
   \begin{proof} Let $\varLambda \colon \cbb_T[z, \bar z]
\to \cbb$ be a linear functional determined by $\varLambda(
z^m \bar z^n) = c_{m,n}$ for all $(m,n) \in T$. A
polynomial $p \in \cbb[z,\bar z]$ is called {\em
nonnegative} if $p(z,\bar z) \Ge 0$ for all $z\in \cbb$.

   (a)$\Rightarrow$(b) Note that the set $T_k \okr
\{(m-k,n-k) \colon (m,n) \in T\}$ satisfies the assumptions
of Theorem \ref{zespmom}. Since $\varLambda (z^k \bar z^k
q(z,\bar z)) \Ge 0$ for all nonnegative polynomials $q \in
\cbb_{T_k}[z,\bar z]$, we infer from the implication
(iv)$\Rightarrow$(i) of Theorem \ref{zespmom} that there
exists a finite positive Borel measure $\nu$ on $\cbb$ such
that $\varLambda(z^k \bar z^k q(z, \bar z)) = \int_\cbb
q(z, \bar z) \D \nu(z)$ for all $q \in \cbb_{T_k}[z,\bar
z]$. Hence
   \begin{align*}
c_{m,n} = \varLambda( z^k \bar z^k (z^{m-k} \bar z^{n-k}))
& =
\int_\cbb z^{m-k} \bar z^{n-k} \D \nu(z) \\
& = \int_{\cbb_*} z^m \bar z^n \D \mu(z) + \nu(\nul)
\delta_{0,m+n-2k}, \quad (m,n) \in T,
   \end{align*}
where $\mu$ is the positive Borel measure on $\cbb$ given
by
   \begin{align*}
\mu(\sigma) = \int_{\sigma \setminus \nul} \frac 1
{|z|^{2k}} \D \nu(z), \quad \sigma \text{ -- a Borel subset
of } \cbb.
   \end{align*}
This $\mu$ and $a \okr \nu(\nul)$ satisfy (b) as well as
the ``in particular'' part of the conclusion.

(b)$\Rightarrow$(a) Pick a nonnegative polynomial $p \in
\cbb_T [z,\bar z]$ and denote its $(k,k)$-coefficient by
$p_{k,k}$. Then $\varLambda (p) = \int_\cbb p(z, \bar z) \D
\mu(z) + a p_{k,k}$ which is nonnegative because $p$ is
nonnegative and $p_{k,k} = \lim_{z\to 0} |z|^{-2k} p(z,\bar
z)$.
   \end{proof}
Clearly, the implication (b)$\Rightarrow$(a) of Proposition
\ref{havi} holds with the same proof if it is only assumed
that $T$ is a (not necessarily symmetric) subset of
the set $\{(m,n) \in \zbb_\pluss \times \zbb_\pluss \colon m,n \Ge k\}$ with some integer $k
\Ge 0$. Note also that for every integer $k \Ge 1$ and for
every symmetric subset $T$ of $\zbb_\pluss \times
\zbb_\pluss$ satisfying \eqref{transT}, we may find a
system $\{c_{m,n}\}_{(m,n)\in T}$ fulfilling the condition
(b) of Proposition \ref{havi} with $\mu$ such that $\int_\cbb |z|^{2k} \D\mu(z) <
\infty$ and $\int_\cbb |z|^{2l} \D\mu(z) =\infty$ for all
$l=0,\ldots, k-1$. Indeed, such a system can be produced
with the help of the formula \eqref{delta0} with $a=0$ and
$\mu$ given by $\D\mu(z) = |z|^{-2k} \eta(z) \D V(z)$,
where $\eta$ and $V$ are as in Example \ref{dziury}.
   \vspace{1ex}
   \subsubsection{\bf The lack of symmetry\label{bsym}}
The phenomenon described below is of general nature
and as such occurs in other instances, like the
truncated multidimensional trigonometric moment
problem (cf.\ Theorem \ref{trmpr}) and the truncated
two-sided complex moment problem (cf.\ Theorem
\ref{zespmom+}). Let us discuss it in the case of the
truncated complex moment problem, which is related to
the $*$-semigroup $\nfr_\pluss$, leaving the other
cases for the reader.

Consider a not necessarily symmetric set $T$ such that
   \begin{equation} \label{primaaprilis}
   (\nfr_\pluss)_{\mathfrak h} \subset T \subset \nfr
   \end{equation}
and look at what happens to the equivalence of the
conditions (i)--(vi) of Theorem \ref{zespmom} (the other
assumptions of Theorem \ref{zespmom} being still in force).
First of all, we see the two natural candidates for
replacing $T$ by a symmetric set: $T\cup T^*$ and $T\cap
T^*$, both satisfying \eqref{primaaprilis}. As shown below, the set $T\cup
T^*$ plays an essential role in conditions (i)--(iii) while
$T\cap T^*$ does so in (iv)--(vi); because these two sets
for very nonsymmetric $T$'s may differ in the extreme the
aforesaid feature seems to be worthy of taking a closer
look at.

Call a system $\{c_{m,n}\}_{(m,n) \in T} \subset \cbb$ {\em
symmetrizable} if
   \begin{equation*}
   c_{m,n}= \overline{c_{n,m}}, \quad (m,n) \in T\cap T^*.
   \end{equation*}
If $\{c_{m,n}\}_{(m,n) \in T}$ is symmetrizable, then its
{\em symmetrization} $\{c_{m,n}^\natural\}_{(m,n) \in T
\cup T^*}$:
   \begin{align*}
c_{m,n}^\natural = \begin{cases}c_{m,n}, & (m,n) \in T,
   \\[1ex]
\overline{c_{n,m}}, & (m,n) \in T^*,
   \end{cases}
   \end{align*}
is well defined. One can verify that if a system
$\{c_{m,n}\}_{(m,n) \in T} \subset \cbb$ satisfies (i) (respectively:\ (ii), (iii))
 on $T$, then it is symmetrizable
and its symmetrization satisfies (i) (respectively:\ (ii), (iii))
on $T\cup T^*$, and vice versa. Theorem
\ref{zespmom} implies that, via the symmetrization
procedure, for {\em any} set $T$ obeying
\eqref{primaaprilis} the conditions (i)--(iii) are
equivalent on $T$.

Regarding the conditions (iv)--(vi), their prospective
equivalence needs to be justified in a different way.
Namely, a system $\{c_{m,n}\}_{(m,n) \in T} \subset \cbb$
fulfils (iv) (respectively:\ (v), (vi)) on $T$ if and only
if the restricted system $\{c_{m,n}\}_{(m,n) \in T\cap
T^*}$ fulfils (iv) (respectively:\ (v), (vi)) on $T\cap
T^*$ (hence $c_{m,n}$'s over $T \setminus T^*$ are
irrelevant). Indeed, this can be deduced from the fact that
a real-valued polynomial $p \in \cbb[z,\bar z]$ belongs to
$\cbb_T[z,\bar z]$ if and only if it belongs to
$\cbb_{T\cap T^*} [z, \bar z]$ (hint:\ conjugate the
polynomial $p$ and deduce that $p_{m,n} =
\overline{p_{n,m}}$, where $p_{m,n}$ are the coefficients
of $p$). Hence, by Theorem \ref{zespmom}, for {\em any} set
$T$ satisfying \eqref{primaaprilis} the conditions
(iv)--(vi) are equivalent on $T$. If this happens, then the
system $\{c_{m,n}\}_{(m,n) \in T}$ is symmetrizable.

Since, in fact, the conditions (i)--(iii) concern the
extension of the system $\{c_{m,n}\}_{(m,n) \in T}$ to
$T\cup T^*$, while (iv)--(vi) deal with its restriction to
$T\cap T^*$, it is to be expected that they cannot be
altogether equivalent for arbitrary $T$. Indeed, consider
any nonsymmetric set $T$ satisfying \eqref{primaaprilis}
and take $(k,l) \in T \setminus T^*$. Suppose that a system
$\{c_{m,n}\}_{(m,n) \in T} \subset \cbb$ fulfils the
conditions (iv)--(vi) on $T$ (e.g.\ any restriction to $T$
of a complex moment sequence does). Define the new system
   \begin{align*}
\tilde c_{m,n} = \begin{cases}c_{m,n}, & (m,n)\in T
\setminus \{(k,l)\},
   \\[1ex]
\sqrt{c_{k,k}c_{l,l} + 1}, & (m,n)=(k,l).
   \end{cases}
   \end{align*}
Owing to the above discussion the so defined system
satisfies the conditions (iv)--(vi), but fails to satisfy
any of the conditions (i)--(iii), because otherwise by the
Cauchy-Schwarz inequality we would have $|\tilde c_{k,l}|^2
\Le \tilde c_{k,k} \tilde c_{l,l} = c_{k,k} c_{l,l}$, a
contradiction. This means that Theorem \ref{zespmom} is not
true as long as $T$ satisfying \eqref{primaaprilis} is not
symmetric. In other words, symmetricity of $T$ is a
necessary condition for Theorem \ref{zespmom} to hold.

However, it becomes now clear that for an {\em arbitrary}
set $T$ satisfying \eqref{primaaprilis} a seemingly more
general version of Theorem \ref{zespmom} can be considered.
Putting it precisely, for a symmetrizable system
$\{c_{m,n}\}_{(m,n)\in T}$ all the conditions (i)--(vi)
remain equivalent if in the conditions (iv)--(vi) the
system $\{c_{m,n}\}_{(m,n)\in T}$ is replaced by its
symmetrization $\{c_{m,n}^\natural\}_{(m,n)\in T \cup
T^*}$.
   \vspace{1ex}
   \subsubsection{\label{sumsquare} \bf Sum-square representation}
We say that a locally convex topology $\tau$ on the linear
space $\cbb[z, \bar z]$ is {\em evaluable} if the set
   \begin{align*}
\{\lambda \in \cbb\colon \text{the evaluation } E_\lambda
\colon \cbb[z,\bar z] \ni p \mapsto p(\lambda, \bar
\lambda) \in \cbb \text{ is $\tau$-continuous}\}
   \end{align*}
is dense in $\cbb$. The class of such topologies is rich.
In particular, it contains every locally convex topology on
$\cbb[z, \bar z]$ generated by the family
$\{E_\lambda\colon \lambda \in Z\}$, where $Z$ is a dense
subset of $\cbb$. This fact, the linear independence of
$\{E_\lambda \colon \lambda \in \cbb\}$ and \cite[Theorem
3.10]{rud} imply that there exist two evaluable topologies
such that the only linear functional on $\cbb[z, \bar z]$
continuous with respect to each of them is the zero
functional.

One can deduce from Artin's solution of the 17th Hilbert problem
(cf.\ \cite{art} or \cite[Theorem 6.1.1]{bcr}; see also
\cite{rez,put-vas} for the case of positive homogeneous polynomials)
that for every nonnegative polynomial $p \in \cbb[z,\bar z]$ (i.e.\
$p(z,\bar z) \Ge 0$ for all $z \in \cbb$), there exist finitely many
rational functions $q_1, \ldots, q_n$ in two complex variables such
that $p(z, \bar z) = \sum_{j=1}^n |q_j(z,\bar z)|^2$ for all $z \in
\cbb$ except singularities of the right-hand side of the equality.
The question arises whether general rational functions in the above
representation of $p$ can be replaced by specific ones from
$\cbb_{\nfr_+}(z, \bar z)$. Though we are unable to answer this
question in full generality, we can do it in the case in which the
nonnegativity of $p$ and its sum-square representation is considered
on a closed proper subset $Z$ of $\cbb$ which is determining for
$\cbb[z,\bar z]$. For convenience we denote by $\varSigma^2(Z)$ the
set of all polynomials $q \in \cbb[z, \bar z]$ for which there exist
finitely many rational functions $q_1, \ldots, q_n \in
\cbb_{\nfr_+}(z,\bar z)$ such that
   \begin{align} \label{sum2}
q(z, \bar z) = \sum_{j=1}^n |q_j(z, \bar z)|^2, \quad
z \in Z \setminus \nul.
   \end{align}
Perhaps this is the right place to mention that if $q \in \cbb[z,
\bar z]$ is of the form \eqref{sum2} with $Z=\cbb$ and $q_1,
\ldots, q_n \in \cbb_{\zbb \times \zbb}(z, \bar z)$, then by
Proposition \ref{ogrwym} the rational functions $q_1, \ldots, q_n$
must belong to $\cbb_{\nfr_+}(z, \bar z)$.
   \begin{pro} \label{sumk}
Let $Z$ be a closed proper subset of $\cbb$ which is
determining for $\cbb[z,\bar z]$ and let $\tau$ be an
evaluable locally convex topology on $\cbb[z, \bar z]$.
Then there exists a polynomial $p \in \cbb[z,\bar z]$ which
is nonnegative on $Z$ and which does not belong to the
$\tau$-closure of $\varSigma^2(Z)$. In particular, $p$ is
not in $\varSigma^2(Z)$.
   \end{pro}
   \begin{proof} Denote by $\mathscr P^+(Z)$ the set
of all polynomials $q \in \cbb[z,\bar z]$ which are
nonnegative on $Z$ (i.e.\ $q(z,\bar z) \Ge 0$ for all $z
\in Z$) and write $\overline{\varSigma^2(Z)}^{\,\tau}$ for
the $\tau$-closure of $\varSigma^2(Z)$. It follows from our
assumptions that there exists $\lambda \in \cbb \setminus
Z$ such that the evaluation $E_\lambda$ is
$\tau$-continuous. By the determining property of
$Z\setminus \nul$ (cf.\ Lemma \ref{determ}), we have
$\varSigma^2(Z) = \varSigma^2(\cbb)$. Hence
   \begin{align} \label{qua}
E_\lambda(q) \Ge 0, \quad q \in
\overline{\varSigma^2(Z)}^{\,\tau}.
   \end{align}
Since $\lambda \notin Z$, there exists real $\varepsilon >
0$ such that $\{z \in \cbb \colon |z - \lambda| \Le
\varepsilon\} \subset \cbb \setminus Z$. It is then clear
that the polynomial $p_\varepsilon (z,\bar z) \okr |z -
\lambda|^2 -\varepsilon^2$ belongs to $\mathscr P^+(Z)$.
Note that $p_\varepsilon \notin
\overline{\varSigma^2(Z)}^{\,\tau}$. Indeed, otherwise
\eqref{qua} implies that $E_\lambda (p_\varepsilon) \Ge 0$,
which contradicts $E_\lambda (p_\varepsilon) = -
\varepsilon^2$.
   \end{proof}
   The proof of Proposition \ref{sumk} remains unchanged if
we assume only that $\tau$ is a locally convex topology on
$\cbb[z, \bar z]$ for which there exists $\lambda \in \cbb
\setminus Z$ such that the evaluation $E_\lambda$ is
$\tau$-continuous.
   \vspace{1ex}
   \subsubsection{\bf Determining sets versus supports of
representing measures} Notice that no determining set is
mentioned in the condition (v) of Theorem \ref{zespmom}. On
the other hand, this condition remains equivalent to the
variety of conditions obtained from (vi) by taking all $\cbb[z,\bar z]$-determining subsets $Z$ of
$\cbb$. The same observation refers to the mutual
relationship between (iv) and (vi). Our intension now is to
see what happens if in (iv) the phrase ``$z \in\cbb$'' is
replaced by ``$z\in Z$''; denote such a modified condition
by (iv)$_Z$ (the same operation applied to (v) leads to
(vi)). Evidently, if a system $\{c_{m,n}\}_{(m,n) \in T}$
satisfies (iv)$_Z$, then it also satisfies (iv).
 We will discuss the following two questions:
   \begin{enumerate}
   \item[$1^\circ$] given a symmetric set $T$ obeying
\eqref{primaaprilis} and a nonzero system
$\{c_{m,n}\}_{(m,n)\in T}$ of complex numbers, is (iv)
equivalent to (iv)$_Z$ for any $\cbb[z,\bar z]$-determining
set $Z\subset \cbb$?
   \item[$2^\circ$] given a symmetric set $T$ obeying
\eqref{primaaprilis} and a
closed\,\footnote{\label{4z}\;Note that (iv)$_Z$ is
equivalent to (iv)$_{\overline Z}$ for any $Z\subset \cbb$.}
proper subset $Z$ of $\cbb$, is (iv) equivalent to (iv)$_Z$
for any system $\{c_{m,n}\}_{(m,n)\in T} \subset \cbb$?
   \end{enumerate}

We will use the notation
   \begin{align*}
Y_{\mathrm r} = \{|z|^2 \colon z\in Y\} \text{ and }
\sqrt{Y_{\mathrm r}} = \{|z| \colon z\in Y\} \ \text{for
any } Y\subset \cbb.
   \end{align*}
Observe that if $Y$ is closed, then so are $Y_{\mathrm r}$
and $\sqrt{Y_{\mathrm r}}$. In order to handle the
questions just posed, we need the following lemma.
   \begin{lem} \label{zet er}
If $T$ obeys {\em
\eqref{primaaprilis}}, $Z$ is a subset of $\cbb$ and
$\{c_{m,n}\}_{(m,n) \in T} \subset \cbb$ is a system
satisfying {\em (iv)}$_Z$, then $\{c_{m,m}\}_{m=0}^\infty$
is a Stieltjes moment sequence which has a representing
measure supported in $\overline{Z_{\mathrm r}}$.
   \end{lem}
   \begin{proof}
If $p(x) = \sum_{j=0}^k p_j x^j \in \cbb[x]$ is nonnegative on
$Z_{\mathrm r}$, then $p(z\bar z) \in \cbb_T[z,\bar z]$ is
nonnegative on $Z$, and consequently, by (iv)$_Z$,
$\sum_{j=0}^k p_j c_{j,j} \Ge 0$. Applying Theorem A
with $\varkappa = 1$, we see that the
Stieltjes moment sequence $\{c_{m,m}\}_{m=0}^\infty$ has a
desired representing measure.
   \end{proof}
The answer to the question $1^\circ$ is in the negative.
Indeed, take a nonzero system $\{c_{m,n}\}_{(m,n)\in T}
\subset \cbb$ satisfying (iv) and suppose that, contrary to
our claim, the system satisfies (iv)$_Z$ for all
$\cbb[z,\bar z]$-determining sets $Z \subset \cbb$. In
particular, this is the case for $Z_1=\{z \in \cbb\colon
|z| \Le 1\}$ and $Z_2 = \{z \in \cbb\colon 2\Le |z| \Le
3\}$ (for their determining property see Lemma \ref{1}). By
Lemma \ref{zet er}, $\{c_{m,m}\}_{m=0}^\infty$ is a
Stieltjes moment sequence having two representing measures supported in $[0,1]$ and
$[4,9]$, respectively. Since each Hamburger moment sequence
with a compactly supported representing measure is
determinate (cf.\ \cite{fug}), we deduce that the support
of the unique representing measure of
$\{c_{m,m}\}_{m=0}^\infty$ is empty. Therefore $c_{m,m}=0$
for all $m \in \zbb_\pluss$. By Theorem \ref{zespmom}\,(i)
and the Cauchy-Schwarz inequality, we have $|c_{m,n}|^2 \Le
c_{m,m} c_{n,n} = 0$ for all $(m,n) \in T$, a
contradiction.

Regarding the question $1^\circ$ with $T=\nfr$, it is
possible to find a complex moment sequence
$\{c_{m,n}\}_{m,n=0}^\infty$ which fulfils (iv)$_Z$ with
uncountably many pairwise disjoint sets $Z$. To see this
consider any indeterminate Hamburger moment sequence
$\{a_n\}_{n=0}^\infty \subset \rbb$ and set $c_{m,n} =
a_{m+n}$ for $m,n\in \zbb_\pluss$. It follows from
\cite[Theorem 4.11]{sim} that $\{a_n\}_{n=0}^\infty$ has a
family $\mathscr W$ (necessarily of cardinality continuum)
of representing measures $\mu$ whose closed supports $\supp
\mu$ are infinite and pure point (i.e.\ with no cluster points), and form
a partition of the real line. It is now easily seen that
the closed sets $Z_\mu \okr \{z \in \cbb \colon \mathfrak
{Re}\hspace{.1ex} z \in \supp \mu\}$, $\mu\in \mathscr W$,
are determining for $\cbb[z,\bar z]$ (see Lemma \ref{1})
and the sequence $\{c_{m,n}\}_{m,n=0}^\infty$ satisfies
(iv)$_{Z_\mu}$ for every $\mu \in \mathscr W$. Note that
the family $\{Z_\mu\}_{\mu\in \mathscr W}$ is a partition
of $\cbb$. An example of an indeterminate Hamburger moment
sequence $\{a_n\}_{n=0}^\infty$ with explicitly computed
pure point supports of representing measures forming a
partition of $\rbb$ may be found in \cite{c-s-sz} (see also \cite{ted} for an explicit example of an indeterminate Stieltjes moment sequence with continuum of representing measures).

The answer to the question $2^\circ$ depends essentially on
the interplay between the sets $T$ and $Z$. We do not
demand that $Z$ be $\cbb[z,\bar z]$-determining, however,
this can be guaranteed in all the examples presented below.
We will first take a closer look at the extremal case
$T=\nfr$ (the other extremality $T =
(\nfr_\pluss)_\mathfrak h$ is discussed below). Then any
determinate nonzero complex moment sequence with the
representing measure supported in $\cbb \setminus Z$
satisfies (iv), but not (iv)$_Z$, the latter being a
consequence of Theorem B. Such a moment
sequence always exists; e.g.\ it can be produced from any
nonzero finite positive Borel measure on $\cbb$ compactly
supported in $\cbb \setminus Z$; for the determinacy of the
so obtained moment sequence see \cite{fug}. Hence, in this
particular case, the answer to the question $2^\circ$ is in
the negative. An alternative way to achieve this conclusion
is by applying Proposition \ref{k-l=1} below.

Another instance of the negative answer to $2^\circ$ is when
$T$ and $Z$ are as in $2^\circ$ and
$Z_\mathrm r \varsubsetneq [0,\infty)$. For this we may
consider a nonzero complex moment sequence
$\{c_{m,n}\}_{m,n=0}^\infty$ with a representing measure
compactly supported in the open set $\{\lambda \in \cbb
\colon |\lambda|^2 \notin Z_\mathrm r\}$. By the measure
transport theorem (or Lemma \ref{zet er}) the Stieltjes
moment sequence $\{c_{m,m}\}_{m=0}^\infty$ has a
representing measure compactly supported in $[0,\infty)
\setminus Z_\mathrm r$ and as such is determinate. It turns
out that the system $\{c_{m,n}\}_{(m,n) \in T}$ satisfies
(iv), but not (iv)$_Z$. Indeed, if it satisfied (iv)$_Z$,
then by Lemma \ref{zet er} the moment sequence
$\{c_{m,m}\}_{m=0}^\infty$ would have a representing
measure supported in $Z_\mathrm r$. Again we would deduce
that the representing measure of $\{c_{m,m}\}_{m=0}^\infty$
is the zero measure and hence $c_{m,n} = 0$ for all $m,n
\in \zbb_\pluss$, a contradiction.

However, the answer to the question $2^\circ$ is in the
affirmative when $T = (\nfr_\pluss)_\mathfrak h$ and $Z_\mathrm r
= [0,\infty)$. To see this it suffices to notice that for every
$Y\subset \cbb$ the system $\{c_{m,n}\}_{(m,n) \in
(\nfr_+)_\mathfrak h}$ satisfies (iv)$_Y$ if and only if the
sequence $\{c_{m,m}\}_{m=0}^\infty$ satisfies the {\em
Riesz-Haviland positivity condition on} $Y_\mathrm r$, i.e.\
$\sum_{j=0}^k p_j c_{j,j} \Ge 0$ for every polynomial $p (x)=
\sum_{j=0}^k p_j x^j \in \cbb[x]$ which is nonnegative on
$Y_{\mathrm r}$. Since $Z_\mathrm r = \cbb_\mathrm r =
[0,\infty)$, we get the desired conclusion.

We now provide more elaborate examples of $T$ and $Z$ for
which the answer to the question $2^\circ$ remains
affirmative. Fix integers $l> k \Ge 0$ and set
   \begin{align*}
\mathscr T_{k,l} = (\nfr_\pluss)_\mathfrak h \cup \{(k,l),
(l,k)\}.
   \end{align*}
Clearly, $\mathscr T_{k,l}$ is symmetric and fulfils
\eqref{primaaprilis}. By Proposition \ref{lambda^2} below,
the answer to the question $2^\circ$ is in the affirmative
whenever $l-k$ is even, $T= \mathscr T_{k,l}$ and $Z = \mathscr Z
\big( \frac{2\pi}{l-k} \big)$, where
   \begin{align*} 
\mathscr Z(\alpha)\okr \{\varrho\,\E^{\I t} \colon t\in [
0, \alpha],\ \varrho \Ge 0 \}, \quad \alpha \in [0,2\pi];
   \end{align*}
note that due to Lemma \ref{niepust} the set $\mathscr
Z(\alpha)$ is $\cbb[z,\bar z]$-determining for $\alpha >
0$. The case of $l-k$ being an arbitrary integer greater
than or equal to $2$ will be settled affirmatively in
Proposition \ref{lambda^2new} below, however its proof
making use of Theorem \ref{zespmom} is no longer
elementary. What is more, while Proposition \ref{lambda^2}
is stated purely in terms of the system $\{c_{m,n}\}_{(m,n)
\in T}$, this seems to be impossible in the case of
Proposition \ref{lambda^2new} (apart from some restricted
cases in which the square root can be approximated by
polynomials in $L^2$-norm with respect to a representing
measure of $\{c_{m,m}\}_{m=0}^\infty$, e.g.\ when the
representing measure is N-extremal, cf.\ \cite{sim}).
According to footnote \ref{4z} and the equality $(\bar
Z)_\mathrm r = \overline{Z_\mathrm r}$, there is no loss of
generality in assuming that $Z$ is closed.
   \begin{pro} \label{lambda^2}
Let $T = \mathscr T_{k,l}$ with $\varkappa \okr (l-k)/2$ being a
positive integer and let $Z$ be a closed $\cbb[z,\bar
z]$-determining subset of $\cbb$ such that
   \begin{align} \label{kappaout}
\Big\{ \varrho\,\E^{\I t} \colon t\in \Big[0,
\frac{2\pi}{l-k} \Big),\ \varrho \in \sqrt{Z_\mathrm r}
\Big\} \subset Z.
   \end{align}
Then for any system $\{c_{m,n}\}_{(m,n) \in T} \subset
\cbb$ the following conditions are equivalent{\em :}
   \begin{enumerate}
   \item[(a)] $\{c_{m,n}\}_{(m,n) \in T}$ satisfies {\em (iv)$_Z$},
   \item[(b)] the sequence $\{c_{m,m}\}_{m=0}^\infty$
satisfies the Riesz-Haviland positivity condition on $Z_\mathrm
r$, $c_{l,k} = \overline {c_{k,l}}$ and $|c_{k,l}| \Le
c_{k+\varkappa, k+ \varkappa}$.
   \end{enumerate}
In particular, if $\mathscr Z\big( \frac{2\pi}{l-k} \big)
\subset Z \varsubsetneq \cbb$, then the answer to
the question $2^\circ$ is in the affirmative.
   \end{pro}
   \begin{proof}
(a)$\Rightarrow$(b) The Riesz-Haviland positivity condition for
$\{c_{m,m}\}_{m=0}^\infty$ has been already discussed in the proof
of Lemma \ref{zet er}. Notice that for every $\theta \in \cbb$
such that $|\theta|\Le 1$, the polynomial
   \begin{align*}
2(z\bar z)^{k+\varkappa} + \theta z^l \bar z^k + \bar\theta
z^k \bar z^l = 2(z\bar z)^{k+\varkappa} + 2 \RE (\theta z^l
\bar z^k)
   \end{align*}
is nonnegative on $\cbb$. Hence, by (a), $2c_{k+\varkappa,
k+\varkappa} + \theta c_{l,k} + \bar\theta c_{k,l} \Ge 0$
for all $\theta\in\cbb$ with $|\theta|\Le 1$. Substituting successively
$\theta = 0$, $\theta =1$ and $\theta = \I$, we deduce that $c_{k+\varkappa,
k+\varkappa} \Ge 0$ and $c_{l,k} =
\overline {c_{k,l}}$. In turn, taking $\theta$ such that
$|\theta|=1$ and $\theta c_{l,k} = - |c_{l,k}|$, we obtain
the remaining inequality in (b).

(b)$\Rightarrow$(a) Assume that $p(z,\bar z) =
\sum_{j=0}^N p_j z^j \bar z^j + \theta z^l \bar z^k +
\tilde\theta z^k \bar z^l \Ge 0$ for all $z\in Z$
($p_0,\ldots, p_N, \theta, \tilde\theta \in \cbb$).
Since $Z$ is a determining set for\/ $\cbb[z,\bar z]$,
we see that $\tilde \theta = \bar \theta$ and $p_j \in
\rbb$ for all $j$. As $\varrho\, \E^{\I t} \in Z$ for
all $t \in [0, \frac \pi\varkappa)$ and $\varrho \in
\sqrt{Z_\mathrm r}$, we get
   \begin{align*}
\sum_{j=0}^N p_j \varrho^{2j} + 2 \varrho^{2(k+ \varkappa)}
\RE (\theta \E^{2\I \varkappa t}) = p(\varrho\, \E^{\I t},
\varrho\, \E^{-\I t}) \Ge 0, \quad \varrho \in \sqrt{Z_\mathrm r}, \, t \in \Big[0, \frac \pi\varkappa\Big).
   \end{align*}
Since the numbers $2\varkappa t$, $t \in [0, \frac \pi\varkappa)$,
exhaust the whole interval $[0,2\pi)$, we deduce that
$\sum_{j=0}^N p_j \varrho^{2j} - 2 |\theta| \varrho^{2(k+
\varkappa)} \Ge 0$ for all $\varrho \in \sqrt{Z_\mathrm r}$. By
the Riesz-Haviland positivity condition, we see that $\sum_{j=0}^N
p_j c_{j,j} - 2 |\theta| c_{k+\varkappa, k+\varkappa} \Ge 0$.
Owing to this inequality and (b), we conclude that
   \begin{align*}
-\theta c_{l,k} - \tilde \theta c_{k,l} = - 2 \RE (\theta
c_{l,k}) \Le 2 |\theta| |c_{l,k}| \Le 2 |\theta|
c_{k+\varkappa, k+ \varkappa} \Le \sum_{j=0}^N p_j c_{j,j},
   \end{align*}
which shows that $\{c_{m,n}\}_{(m,n) \in T}$ satisfies
(iv)$_Z$.

The ``in particular'' part of the conclusion follows from
Lemma \ref{niepust} and the equivalence
(a)$\Leftrightarrow$(b) (because $Z_\mathrm r =\cbb_\mathrm r = [0,\infty)$).
   \end{proof}
Note that if $l-k$ is even, then by Theorem A
and the measure transport theorem the condition (b) of
Proposition \ref{lambda^2} is equivalent to the condition
(b) below. The key observation is that the integral
$\int_{[0,\infty)} \varrho^{k+l} \D \nu(\varrho)$ is equal
to $c_{k+\varkappa,k+\varkappa}$.
   \begin{pro} \label{lambda^2new}
Let $T = \mathscr T_{k,l}$ with $l > k$ and let $Z$ be a closed
subset of $\cbb$ satisfying \eqref{kappaout}. Given a
system $\{c_{m,n}\}_{(m,n) \in T} \subset \cbb$, consider
the following two conditions:
   \begin{enumerate}
   \item[(a)] $\{c_{m,n}\}_{(m,n) \in T}$ satisfies {\em (iv)$_Z$},
   \item[(b)] there exists a finite positive Borel measure $\nu$ on
$[0,\infty)$ supported in $\sqrt{Z_\mathrm r}$ such
that $c_{m,m} = \int_{[0,\infty)} \varrho^{2m} \D
\nu(\varrho)$ for all $m\in \zbb_\pluss$, $c_{l,k} =
\overline {c_{k,l}}$ and $|c_{k,l}| \Le
\int_{[0,\infty)} \varrho^{k+l} \D \nu(\varrho)$.
   \end{enumerate}
Then {\em (b)} implies {\em (a)}. If additionally $Z_\mathrm r =
[0,\infty)$ or $\{c_{m,m}\}_{m=0}^\infty$ is a determinate
Stieltjes moment sequence, then {\em (a)} implies {\em
(b)}. In particular, if $\mathscr Z\big( \frac{2\pi}{l-k}
\big) \subset Z \varsubsetneq \cbb$ with $l-k \Ge
2$, then the answer to the question $2^\circ$ is in the
affirmative.
   \end{pro}
   \begin{proof}
(a)$\Rightarrow$(b) Since $\{c_{m,n}\}_{(m,n) \in T}$
evidently satisfies (iv), we deduce from Theorem
\ref{zespmom} that there exists a positive Borel measure
$\mu$ on $\cbb$ such that $c_{m,n} = \int_\cbb z^m \bar z^n
\D \mu(z)$ for all $(m,n) \in T$. Clearly, $c_{l,k} =
\overline {c_{k,l}}$. Applying the measure transport
theorem, we see that the finite positive Borel measure
$\nu$ on $[0,\infty)$ defined via $\nu(\sigma) =
\mu(\{z\in \cbb \colon |z| \in \sigma\})$ for Borel subsets
$\sigma$ of $[0,\infty)$ satisfies the first equality in
(b). The inequality in (b) can be justified as follows:
   \begin{align*}
|c_{k,l}| \Le \int_\cbb |z|^{k+l} \D \mu(z) =
\int_{[0,\infty)} \varrho^{k+l} \D \nu (\varrho).
   \end{align*}
Thus the case of $Z_\mathrm r = [0,\infty)$ is settled. If
$Z_\mathrm r \varsubsetneq [0,\infty)$ and
$\{c_{m,m}\}_{m=0}^\infty$ is a determinate Stieltjes
moment sequence, then by Lemma \ref{zet er} and the measure
transport theorem we deduce that the measure $\nu$ is
supported in $\sqrt{Z_\mathrm r}$.

(b)$\Rightarrow$(a) By the inequality in (b), there exists
$\theta \in \cbb$ such that $|\theta|\Le 1$ and $c_{k,l} =
\theta \int_{[0,\infty)} \varrho^{k+l} \D \nu(\varrho)$.
It is easily seen that there exist (not necessarily
distinct) numbers $t_1,t_2 \in [0,2\pi)$ such that
$\bar\theta = \frac 12( \E^{\I t_1} + \E^{\I t_2})$. Let
$\zeta$ be a positive Borel measure on $[0,2\pi)$ supported
in $\big\{\frac{t_1}{j}, \frac{t_2}{j}\big\}$ with
$\zeta\big(\big\{\frac{t_1}{j}\big\}\big) =
\zeta\big(\big\{\frac{t_2}{j}\big\}\big) = \frac 12$, where
$j = l-k$. Define the Borel measure $\mu$ on $\cbb$ via
   \begin{align*}
\mu(\sigma) = \int_{[0,2\pi)} \int_{[0,\infty)}
\chi_\sigma(\varrho \, \E^{\I t}) \D\nu(\varrho)
\D\zeta(t), \quad \sigma \text{ -- Borel subset of } \cbb.
   \end{align*}
It is a matter of routine to verify that $c_{m,n} =
\int_\cbb z^m \bar z^n \D\mu(z)$ for all $(m, n)\in T$
(hint:\ $\bar \theta = \int_{[0,2\pi)} \E^{\I j t}
\D\zeta(t)$). One can show that the closed support of the
measure $\mu$ is contained in the set
   \begin{align}\label{listek}
\{\varrho\, \E^{\I t_1/j} \colon \varrho \in
\sqrt{Z_\mathrm r}\} \cup \{\varrho\, \E^{\I t_2/j} \colon
\varrho \in \sqrt{Z_\mathrm r}\},
   \end{align}
which in view of \eqref{kappaout} is a subset of
$Z$. This implies that the system
$\{c_{m,n}\}_{(m,n)\in T}$ satisfies (iv)$_Z$. Observe
that the construction of the measure $\mu$ is based on
the possibility of representing $\theta$ as an
arithmetic mean of two complex numbers of absolute
value $1$. In fact, the same proof works if $\theta$
is represented as a finite convex combination of
complex numbers of absolute value $1$, in which case
the closed support of $\mu$ consists of a finite
number of sets of the type appearing in
\eqref{listek}.

The ``in particular'' part of the conclusion follows from
the equivalence (a)$\Leftrightarrow$(b) which is valid
because $Z_\mathrm r = \cbb_\mathrm r = [0,\infty)$.
   \end{proof}
The following proposition shows that the angle
$\frac{2\pi}{l-k}$ appearing in the assumption
\eqref{kappaout} of Propositions \ref{lambda^2} and
\ref{lambda^2new} is optimal, i.e.\ it cannot be made
smaller. It is worth pointing out that if $l=k+1$, then the
assumption \eqref{pierw} below is satisfied by any proper
subset $Z$ of $\cbb$.
   \begin{pro}\label{k-l=1}
Let $T$ be a symmetric subset of $\nfr$ such that $\mathscr T_{k,l}
\subset T$ for some integers $l> k \Ge 0$ and let $Z$ be a
closed subset of $\cbb$ for which there exists $\lambda \in
\cbb$ such that
   \begin{align}\label{pierw}
\{z\in \cbb \colon z^{l-k} = \lambda\} \subset \cbb
\setminus Z.
   \end{align}
Then there exists a system $\{c_{m,n}\}_{(m,n)\in T}
\subset \cbb$ which satisfies {\em (iv)}, but not {\em
(iv)$_Z$}. In particular, if $Z \subset \mathscr Z(\alpha)$
with $\alpha \in \big[0,\frac{2\pi}{l-k} \big)$, then the
answer to the question $2^\circ$ is in the negative.
   \end{pro}
   \begin{proof}
Put $j=l-k$. Without loss of generality we may assume that
$\lambda \neq 0$. Then there exists $\varepsilon >0$ such
that
   \begin{align}\label{deltaj}
\varnothing \neq \varDelta_\varepsilon \okr \{z \in \cbb
\colon |z^j-\lambda| < \varepsilon \} \subset \cbb
\setminus (Z \cup \nul).
   \end{align}
To see this suppose that $\lambda_1, \ldots, \lambda_j$ are
all the complex $j$-roots of $\lambda$. By \eqref{pierw}
and the closedness of $Z$, there exists $\delta>0$ such
that $\min_{n \in \{1,\ldots, j\}} |z - \lambda_n| \Ge
\delta$ for all $z\in Z\cup \nul$. Since $z^j-\lambda =
\prod_{n=1}^j (z - \lambda_n)$ and consequently
   \begin{align*}
\Big( \min_{n \in \{1,\ldots, j\}} |z - \lambda_n| \Big)^j
\Le \prod_{n=1}^{j} |z - \lambda_n| = |z^j - \lambda|,\quad
z\in \cbb,
   \end{align*}
we deduce that $\varDelta_\varepsilon$ is contained in
$\cbb \setminus(Z \cup \nul)$ whenever $\varepsilon \Le
\delta^j$.

Consider a nonzero finite positive Borel measure $\mu$ on
$\cbb$ compactly supported in the open set $\varDelta_\varepsilon$. Set
$c_{m,n} = \int_\cbb z^m \bar z^n \D \mu(z)$ for $(m,n) \in
T$ and
   \begin{multline*}
p(z,\bar z) = |z|^{2k}(|z^j - \lambda|^2 - \varepsilon^2) =
z^l \bar z^l - \bar \lambda z^l \bar z^k - \lambda z^k \bar
z^l + (|\lambda|^2 - \varepsilon^2) z^k \bar z^k, \quad
z\in \cbb.
   \end{multline*}
Plainly, $p \in \cbb_T[z,\bar z]$. By \eqref{deltaj}, we
see that $p(z,\bar z) \Ge 0$ for all $z\in Z$ and $p(z,\bar
z) < 0$ for all $z\in \varDelta_\varepsilon$ (because $0
\notin \varDelta_\varepsilon$). Evidently,
$\{c_{m,n}\}_{(m,n)\in T}$ satisfies (iv), but not
(iv)$_Z$, because $\mu \neq 0$ and
   \begin{align*}
c_{l,l} - \bar \lambda c_{l,k} - \lambda c_{k,l} +
(|\lambda|^2 - \varepsilon^2) c_{k,k} =
\int_{\varDelta_\varepsilon} p(z,\bar z) \D\mu (z) < 0.
   \end{align*}

To prove the ``in particular'' part of the conclusion note
that $\lambda = \E^{\I j \theta}$ satisfies \eqref{pierw}
for any $\theta \in (\alpha, 2\pi/j)$. The proof is
complete.
   \end{proof}
Summing up, the case of sets $\mathscr T_{k,l}$ serves as a good
elucidation of the interplay between $T$ and $Z$ which is
crucial when dealing with the question $2^\circ$. In the
table below we gather information concerning this question;
we keep the assumptions on $T$ and $Z$ made therein.
   \vspace{1ex}
   \begin{center}
   \renewcommand{\arraystretch}{1.3}
   \begin{tabular}{|c||c|c|}
\hline
{\sf Answer} & $T$ & $Z$ \\
\hline\hline \multirow{4}{1.3cm}{\centering\sf NO}
& arbitrary & $Z_\mathrm r \varsubsetneq [0,\infty)$ \\
   \cline{2-3} & $T \supset \mathscr T_{k,l}$, $l-k\Ge 1$
& $Z$ satisfies \eqref{pierw} \\
   \cline{2-3} & $T \supset \mathscr T_{k,l}$, $l-k\Ge 1$ & $Z
\subset \mathscr Z(\alpha)$, $\alpha \in \big[0,
\frac{2\pi}{l-k}\big)$ \\
   \cline{2-3}
& $T \supset \mathscr T_{k,k+1}$, $k\Ge 0$ & arbitrary \\
\hline\hline \multirow{2}{1.3cm}{\centering\sf YES} & $T =
\mathscr T_{k,l}$, $l-k \Ge 2$
& $Z \supset \mathscr Z\big(\frac{2\pi}{l-k} \big)$ \\
   \cline{2-3} & $T = (\nfr_\pluss)_\mathfrak h$ &
   $Z_\mathrm r =
[0,\infty)$\\
\hline
   \end{tabular}
   \end{center}
   \vspace{1ex}
   To justify the `{\sf YES}' part of the
table one should notice that if $T' \subset T$, $Z \subset
Z'$ and the answer to the question $2^\circ$ is in the
affirmative for $T$ and $Z$, then it is so for $T'$ and
$Z'$. In turn, the `{\sf NO}' part requires
contraposition, i.e.\ if
the answer to $2^\circ$ is in the negative for $T'$ and
$Z'$, then it is so for $T$ and $Z$. In view
of these properties and the table above, if $T=\mathscr T_{k,l}$
with $l-k \Ge 1$ and $Z = \mathscr Z(\alpha)$, then the
answer to the question $2^\circ$ is in the negative for
$\alpha \in [0, \frac{2\pi}{l-k})$ and in the affirmative
for $\alpha \in [\frac{2\pi}{l-k}, 2\pi)$.
   \subsection{Subnormality}
Let $S$ be a densely defined linear operator in a complex Hilbert
space $\hh$ with domain $\dz S$. We say that $S$ is {\em subnormal}
if there exist a complex Hilbert space $\kk$ and a normal operator
$N$ in $\kk$ such that $\hh \subset \kk$ (isometric embedding), $\dz
S \subset \dz N$ and $Sh = Nh$ for all $h \in \dz S$. For
fundamentals of the theory of unbounded subnormal operators we refer
the reader to \cite{stoszaf0,stoszaf1,stoszaf2}.

   The following characterization of subnormality
simplifies substantially that of \cite[Theorem 3]{stoszaf1}
(one double sum turns out to be redundant). As in
\cite{stoszaf1}, it is intrinsic in a sense that no
extension is involved. Theorem \ref{chsub} is also related
to part (iv) of \cite[Theorem 37]{st-sz1} which in the case
of $\mathscr F = \mathscr D$ is equivalent to condition
(iii) below.
   \begin{thm} \label{chsub}
Let $S$ be a densely defined linear operator in a complex
Hilbert space $\hh$ such that $S(\dz S) \subset \dz S$.
If $Z \subset \cbb_*$ is a determining set for $\cbb[z,\bar
z]$, then the following conditions are equivalent{\em :}
   \begin{enumerate}
   \item[(i)] $S$ is subnormal,
   \item[(ii)] for every system $\{a_{p,q}^{i,j}\}_{p,q = 0,
\ldots, n}^{i,j=1, \ldots, m} \subset \cbb$, if
   \begin{align}\label{ogol}
\sum_{i,j=1}^m \sum_{p,q=0}^n a_{p,q}^{i,j} \lambda^p
\bar \lambda^q z_i \bar z_j \Ge 0, \quad \lambda, z_1,
\ldots, z_m \in \cbb,
   \end{align}
then
   \begin{align} \label{opineq}
\sum_{i,j=1}^m \sum_{p,q=0}^n a_{p,q}^{i,j} \is{S^p f_i} {S^q f_j}
\Ge 0, \quad f_1, \ldots, f_m \in \dz S,
   \end{align}
   \item[(iii)] for every system $\{a_{p,q}^{i,j}\}_{p,q = 0,
\ldots, n}^{i,j=1, \ldots, m} \subset \cbb$, if there is a finite
matrix $[q_{i,l}]_{i=1}^m{}_{l=1}^k$ with entries in
$\cbb_{\nfr_+}(z,\bar z)$ such that
   \begin{align}\label{spe}
\sum_{p,q=0}^n a_{p,q}^{i,j} \lambda^p \bar \lambda^q = \sum_{l=1}^k
q_{i,l}(\lambda, \bar \lambda) \overline{q_{j,l}(\lambda, \bar
\lambda)}, \quad \lambda \in Z, \, i,j = 1, \ldots, m,
   \end{align}
then \eqref{opineq} holds.
   \end{enumerate}
   \end{thm}
   \begin{proof}
    The proof of implication (i)$\Rightarrow$(ii) proceeds along the
same lines as the proof of the ``only if'' part of \cite[Theorem
3]{stoszaf1}. The other possibility is to argue as in the proof of
Lemma \ref{wn1}\,(i).

    (ii)$\Rightarrow$(iii) First note that if \eqref{spe}
holds, then, by the determining property of $Z$ (cf.\ Lemma
\ref{determ}), the equality in \eqref{spe} is valid for all
$\lambda \in \cbb_*$. It is then easily seen that
\eqref{ogol} holds, which by (ii) yields \eqref{opineq}.

    (iii)$\Rightarrow$(i) Consider the $*$-semigroup
$\sfr=\nfr_\pluss$ (which is operator semiperfect due to Remark 7 and Proposition 23 in \cite{st-sz1}), the linear space $\dcal = \dz S$ and the sets $T=\zbb_\pluss \times \zbb_\pluss$ and $Y=Y_Z$ (cf.\ \eqref{yz}). Define the mapping $\varPhi\colon T \to
\scal(\dcal)$ by
    \begin{align*}
\varPhi(m,n) (f,g) = \is{S^mf}{S^n g}, \quad f,g \in \dcal,\, (m,n)
\in T,
    \end{align*}
    and attach to it the linear mapping $\varLambda_{\varPhi,Y} \colon \pcal_T(Y)
\to \scal(\dcal)$ via formula \eqref{1maja}. Then a simple
calculation based on Lemma \ref{factor} shows that (iii) is
equivalent to the complete f-positivity of $\varLambda_{\varPhi,Y}$.
Applying implication (iv)$\Rightarrow$(i) of Theorem \ref{main2} and
implication (iii)$\Rightarrow$(i) of \cite[Theorem 37]{st-sz1} with
$\mathscr F = \mathscr D$ completes the proof.
   \end{proof}
   \subsection{\label{sekl}Unitary dilation of several contractions}
In what follows $\varkappa$ stands for a positive integer.
Set $\zbb_\minuss=-\zbb_\pluss$. Denote by
$\zbb^\varkappa$, $\zbb_\pluss^\varkappa$,
$\zbb_\minuss^\varkappa$ and $\mathbb T^\varkappa$ the
cartesian product of $\varkappa$ copies of $\zbb$,
$\zbb_\pluss$, $\zbb_\minuss$ and $\mathbb T$,
respectively. For simplicity, we write $0$ for the zero
element of the group $\zbb^\varkappa$. Observe that
$\zbb_\pluss^\varkappa \cup \zbb_\minuss^\varkappa =
\zbb^\varkappa$ only for $\varkappa=1$. Let $\pcal(\mathbb
T^\varkappa)$ stand for the linear space of all functions
$p \colon \mathbb T^\varkappa \to \cbb$ of the form
    \begin{align} \label{for1}
p(z) = \sum_{\alpha \in \zbb^\varkappa} a_\alpha z^\alpha, \quad z
\in \mathbb T^\varkappa,
    \end{align}
where $\{a_\alpha\}_{\alpha \in \zbb^\varkappa}$ is a
finite system of complex numbers, and $z^\alpha =
z_1^{\alpha_1} \cdots z_\varkappa^{\alpha_\varkappa}$ for
$z = (z_1, \dots, z_\varkappa) \in \mathbb T^\varkappa$ and
$\alpha=(\alpha_1, \ldots, \alpha_\varkappa)
\in\zbb^\varkappa$. The members of $\pcal(\mathbb
T^\varkappa)$ are called {\em trigonometric polynomials} in
$\varkappa$ variables. A trigonometric polynomial $p$ vanishes on $\mathbb T^\varkappa$ if and only if all its coefficients $a_\alpha$ vanish. Given $T \subset \zbb^\varkappa$, we
denote by $\pcal_T(\mathbb T^\varkappa)$ the linear space
of all trigonometric polynomials $p \in \pcal(\mathbb
T^\varkappa)$ of the form \eqref{for1}, where $a_\alpha =
0$ for all $\alpha \in \zbb^\varkappa \setminus T$. We
abbreviate $\pcal_{\zbb_+^\varkappa}(\mathbb T^\varkappa)$
to $\pcal_\pluss(\mathbb T^\varkappa)$ and
$\pcal_{\zbb_+^\varkappa \cup \zbb_-^\varkappa}(\mathbb
T^\varkappa)$ to $\pcal_\pm(\mathbb T^\varkappa)$. One can
think of members of $\pcal_\pluss(\mathbb T^\varkappa)$ as
{\em analytic} trigonometric polynomials. A nonempty subset
$Y$ of $\mathbb T^\varkappa$ is said to be {\em
determining} for $\pcal(\mathbb T^\varkappa)$
(respectively:\ $\pcal_\pluss(\mathbb T^\varkappa)$) if
each trigonometric polynomial $p \in \pcal(\mathbb
T^\varkappa)$ (respectively:\ $p \in \pcal_\pluss(\mathbb
T^\varkappa)$) vanishing on $Y$ vanishes on the whole set
$\mathbb T^\varkappa$. Note that any infinite subset of
$\mathbb T$ is determining for $\pcal(\mathbb T)$.
    \begin{lem} \label{ppl}
A subset $Y$ of $\mathbb T^\varkappa$ is determining for
$\pcal(\mathbb T^\varkappa)$ if and only if it is determining for
$\pcal_\pluss(\mathbb T^\varkappa)$.
    \end{lem}
    \begin{proof}
This is clear, because for every $p \in \pcal(\mathbb
T^\varkappa)$, there exists $n \in \zbb_\pluss$ such that
the trigonometric polynomial $z_1^n\dots z_\varkappa^n \,
p(z_1, \dots, z_\varkappa)$ is analytic.
    \end{proof}
Let $\boldsymbol A = (A_1, \ldots, A_\varkappa)$ be a
$\varkappa$-tuple of bounded linear operators on a complex Hilbert
space $\hh$. Define the family $\{\boldsymbol A^{[\alpha]}\colon
\alpha \in \zbb_\pluss^\varkappa \cup \zbb_\minuss^\varkappa\}$ by
    \begin{align*}
\boldsymbol A^{[\alpha]} =
   \begin{cases} A_1^{\alpha_1} \cdots
A_\varkappa^{\alpha_\varkappa}, & \alpha \in \zbb_\pluss^\varkappa,
   \\[1ex]
A_1^{*|\alpha_1|} \cdots A_\varkappa^{*|\alpha_\varkappa|}, & \alpha
\in \zbb_\minuss^\varkappa.
   \end{cases}
   \end{align*}
For $\alpha \in \zbb_\pluss^\varkappa$ we replace $\boldsymbol
A^{[\alpha]}$ by the standard multi-index notation $\boldsymbol
A^\alpha$. Following \cite[page 32]{nag3}, we say that a
$\varkappa$-tuple $\boldsymbol A$ has a {\it unitary power dilation}
if there exists a complex Hilbert space $\kk \supset \hh$ (isometric
embedding) and a $\varkappa$-tuple $\boldsymbol U = (U_1, \ldots,
U_\varkappa)$ of commuting unitary operators on $\kk$ such that
    \begin{align*} 
\boldsymbol A^\alpha = P \boldsymbol U^\alpha \bigr|_{\hh}, \quad
\alpha \in \zbb_\pluss^\varkappa,
    \end{align*}
   where $P$ stands for the orthogonal projection of $\kk$ onto
$\hh$. Such $\boldsymbol U$ is called a {\em unitary power dilation}
of $\boldsymbol A$. The proof of the following fact is left to the
reader.
    \begin{lem} \label{*dyl}
If $\boldsymbol U$ is a unitary power dilation of $\boldsymbol A$,
then the operators $A_1, \ldots, A_\varkappa$ commute if and only if
$\boldsymbol A^{[\alpha]} = P \boldsymbol U^\alpha \bigr|_{\hh}$ for
all $\alpha \in \zbb_\pluss^\varkappa \cup \zbb_\minuss^\varkappa$.
    \end{lem}
    We are now in a position to formulate necessary and
sufficient conditions for a $\varkappa$-tuple of bounded
operators to have a unitary power dilation. Theorem
\ref{nagy} below is related to characterizations of
families of operators having unitary power dilations given
in \cite[Corollary 6]{st-sz2} (see also \cite[Lemma
1]{j-s-sz} for another formulation which does not appeal to
the boundedness of $\boldsymbol A$; in fact, one can easily
write a version of Theorem \ref{nagy} for operators which
are not a priori assumed to be bounded).
Different approaches to the problem of the existence of
unitary power dilation have recently appeared in \cite{fhv} and \cite{ar}; however, the solutions proposed therein are not written in terms of operators in question.
    \begin{thm} \label{nagy}
If $\boldsymbol A = (A_1, \ldots, A_\varkappa)$ is a
$\varkappa$-tuple of bounded linear operators on a complex Hilbert
space $\hh$ and $Y$ is a determining set for $\pcal_\pluss(\mathbb
T^\varkappa)$, then the following conditions are equivalent{\em :}
    \begin{enumerate}
    \item[(i)] $\boldsymbol A$ has a unitary power dilation and
the operators $A_1, \ldots, A_\varkappa$ commute,
    \item[(ii)] for every finite system $\{a_\alpha^{i,j}\colon
i,j=1, \ldots, m, \alpha \in \zbb_\pluss^\varkappa \cup
\zbb_\minuss^\varkappa\} \subset \cbb$, if
   \begin{align} \label{drit}
\sum_{i,j=1}^m \sum_{\alpha \in \zbb_+^\varkappa \cup
\zbb_-^\varkappa} a_\alpha^{i,j} \lambda^\alpha z_i \bar z_j \Ge 0,
\quad \lambda \in \mathbb T^\varkappa, \, z_1, \ldots, z_m \in \cbb,
   \end{align}
then
   \begin{align} \label{nagye}
\sum_{i,j=1}^m \sum_{\alpha \in \zbb_+^\varkappa \cup
\zbb_-^\varkappa} a_\alpha^{i,j} \is{\boldsymbol A^{[\alpha]} f_i}
{f_j} \Ge 0, \quad f_1, \ldots, f_m \in \hh,
   \end{align}
    \item[(iii)] for every finite system $\{a_\alpha^{i,j}\colon
i,j=1, \ldots, m, \alpha \in \zbb_\pluss^\varkappa \cup
\zbb_\minuss^\varkappa\} \subset \cbb$, if there is a finite matrix
$[q_{i,l}]_{i=1}^m{}_{l=1}^k$ with entries in $\pcal_\pluss(\mathbb
T^\varkappa)$ such that
   \begin{align} \label{spe+}
\sum_{\alpha \in \zbb_+^\varkappa \cup \zbb_-^\varkappa}
a_\alpha^{i,j} \lambda^\alpha = \sum_{l=1}^k q_{i,l}(\lambda)
\overline{q_{j,l}(\lambda)}, \quad \lambda \in Y, \, i,j = 1,
\ldots, m,
   \end{align}
then \eqref{nagye} holds.
    \end{enumerate}
    \end{thm}
    \begin{proof}
Without loss of generality we can assume that $Y$ is a determining
set for $\pcal(\mathbb T^\varkappa)$ (cf.\ Lemma \ref{ppl}). In what
follows, we regard $\zbb^\varkappa$ as the $*$-semigroup equipped
with coordinatewise defined addition as semigroup operation and
involution $\alpha^*=-\alpha$, $\alpha \in \zbb^\varkappa$. In turn,
$\mathbb T^\varkappa$ is regarded as the multiplicative
$*$-semigroup with coordinatewise defined multiplication as
semigroup operation and involution
    $$
(z_1, \ldots, z_\varkappa)^* = (\bar z_1, \ldots, \bar z_\varkappa),
\quad (z_1, \ldots, z_\varkappa) \in \mathbb T^\varkappa.
    $$
It is easily checked that the dual $*$-semigroup
$\xfr_{\zbb^\varkappa}$ of $\zbb^\varkappa$ can be identified
algebraically with the $*$-semigroup $\mathbb T^\varkappa$ via the
mapping
    \begin{align}
\xfr_{\zbb^\varkappa} \ni \chi \mapsto (\chi(e_1), \ldots,
\chi(e_\varkappa)) \in \mathbb T^\varkappa,
    \end{align}
where $e_j = (\delta_{j,1}, \ldots, \delta_{j,\varkappa}) \in
\zbb^\varkappa$ ($\delta_{i,j}$ is the Kronecker delta function).
Under this identification, $\widehat{\alpha}$ is given by
    \begin{align*}
\widehat{\alpha} (z) = z^\alpha, \quad z = (z_1, \ldots,
z_\varkappa) \in \mathbb T^\varkappa, \, \alpha \in \zbb^\varkappa,
    \end{align*}
which means that the notation $\pcal(\mathbb T^\varkappa)$
introduced at the beginning of this section is consistent with that
for $*$-semigroups in Section \ref{s1}, and that $Y$ is a
determining subset of $\xfr_{\zbb^\varkappa}$. It is well known that
the $*$-semigroup $\zbb^\varkappa$ is operator semiperfect (e.g.\
see \cite{mas} and footnote \ref{fut}). Put $T =
\zbb_\pluss^\varkappa \cup \zbb_\minuss^\varkappa$ and define
$\varPhi\colon T \to \scal(\hh)$ via
    \begin{align*}
\varPhi(\alpha) (f,g) = \is{\boldsymbol A^{[\alpha]} f} g, \quad f,g
\in \hh,\, \alpha\in T.
    \end{align*}
In view of Lemma \ref{*dyl} and \cite[Theorem 3]{st-sz2}, the
current condition (i) is equivalent to condition (i) of Theorem
\ref{main2} with $\sfr=\zbb^\varkappa$ and $\dcal=\hh$. In turn, the
current condition (ii) is a counterpart of condition (iii) of
Theorem \ref{main2}. Finally, by Lemma \ref{factor}, the current
condition (iii) is a counterpart of condition (iv) of Theorem
\ref{main2}, because if \eqref{spe+} holds for some matrix
$[q_{i,l}]_{i=1}^m{}_{l=1}^k$ with entries in $\pcal(\mathbb
T^\varkappa)$, then there exists $n \in \zbb_\pluss$ such that all
the trigonometric polynomials
    \begin{align}
\tilde q_{i,l}(z_1, \dots, z_\varkappa) \okr z_1^n\dots
z_\varkappa^n \, q_{i,l}(z_1, \dots, z_\varkappa), \quad z_1,
\ldots, z_\varkappa \in \mathbb T,
    \end{align}
are analytic and \eqref{spe+} is valid with $[\tilde
q_{i,l}]_{i=1}^m{}_{l=1}^k$ in place of
$[q_{i,l}]_{i=1}^m{}_{l=1}^k$. Hence, applying Theorem \ref{main2}
completes the proof.
    \end{proof}
   \begin{rem} The implication (iii)$\Rightarrow$(ii) of
Theorem \ref{nagy} can also be deduced from \cite[Corollary
5.2]{dr}. Indeed, if \eqref{drit} is valid, then for every real
$\varepsilon > 0$, the square-matrix-valued trigonometric polynomial
$Q^{(\varepsilon)} (\lambda) \okr \sum_{\alpha \in \zbb_+^\varkappa
\cup \zbb_-^\varkappa} Q_\alpha \lambda^\alpha + \varepsilon I_m$,
where $Q_\alpha = [a_\alpha^{i,j}]_{i,j=1}^m$ and $I_m =
[\delta_{i,j}]_{i,j=1}^m$, is strictly positive on $\mathbb
T^\varkappa$. By \cite[Corollary 5.2]{dr}, the polynomial
$Q^{(\varepsilon)}$ has a factorization by an analytic (in general
non-square) matrix-valued trigonometric polynomial. This implies
that the polynomial $Q^{(\varepsilon)}$ takes the form which is
required in \eqref{spe+} (use the trick\footnote{\;This additional effort
comes from the fact that Dritschel's factorization
$F(\lambda)^*F(\lambda)$ differs from the factorization $P(\lambda)
P(\lambda)^*$ required in \eqref{spe+} by the location of the
asterisk.} from the proof of Lemma \ref{ppl}). Hence, by (iii), we have
    \begin{align*}
\sum_{i,j=1}^m \sum_{\alpha \in \zbb_+^\varkappa \cup
\zbb_-^\varkappa} a_\alpha^{i,j} \is{\boldsymbol A^{[\alpha]} f_i}
{f_j} + \varepsilon \sum_{i=1}^m \|f_i\|^2 \Ge 0, \quad f_1, \ldots,
f_m \in \hh,\, \varepsilon > 0.
   \end{align*}
Passing with $\varepsilon$ to $0$ completes the proof.
    \end{rem}
   \subsection{\label{conf}The truncated multidimensional trigonometric moment
problem}
   The problem in the title goes back to Krein's theorem
\cite{krein} on extending positive definite functions
from an interval to $\rbb$. A several dimensional
version of this is not true according to examples of
Calder\'on and Pepinsky \cite{cal-pep} (the additive
group $\zbb^\varkappa$) and Rudin \cite{rud1} (the
additive groups $\zbb^\varkappa$ and
$\rbb^\varkappa$). In this section we concentrate on
the discrete case. Our result, which is related to
\cite[Corollary 4]{st-sz2}, deals with the truncated
multidimensional trigonometric moment problem on
symmetric subsets of $\zbb^\varkappa$.

It was proved in \cite{cal-pep}, and independently in
\cite{rud1}, that a finite nonempty subset $\varLambda$ of
$\zbb^\varkappa$ has the extension property (i.e.\ each
complex function on $\varLambda - \varLambda$ which is
positive definite on $\varLambda$ extends to a positive
definite complex function on $\zbb^\varkappa$) if and only
if each nonnegative trigonometric polynomial in
$\pcal_{\varLambda - \varLambda}(\mathbb T^\varkappa)$ is
equal to a finite sum of squares of moduli of trigonometric
polynomials in $\pcal_\varLambda(\mathbb T^\varkappa)$.
Recently, Gabardo found new conditions under which
$\varLambda$ has or fails to have the extension property
(cf.\ \cite{ga3}; see also \cite{ga4} for
$\varLambda$-determinacy and \cite{ga1,ga2} for related
questions). Clearly, if $\varLambda$ has the extension
property, then $\mathscr P_{\varLambda -
\varLambda}^+(\mathbb T^\varkappa) = \varSigma_{\varLambda
- \varLambda}^2(\mathbb T^\varkappa)$ (see Section
\ref{appr} for notation). Note that each {\em difference
set} $T \subset \zbb^\varkappa$, i.e.\ a set of the form
$\varLambda - \varLambda$ with some nonempty $\varLambda
\subset \rbb^\varkappa$, has the property $0 \in T = -T$,
which is required in Theorem \ref{trmpr} below. However,
not every set $T \subset \zbb^\varkappa$ satisfying $0 \in
T = -T$ is a difference set, which can be seen even for
$\varkappa=1$. Namely, one can show that for all integers
$k,n$ such that $1 \Le k < n-k <n$ (necessarily $n \Ge 3$
and $k < n/2$) any subset $T$ of $\zbb$ fulfilling the
following conditions
   \begin{enumerate}
   \item[(i)] $0 \in T = -T$,
   \item[(ii)] $T \subset \{j \in \zbb\colon |j| \Le
   n\}$,
   \item[(iii)] $n-k, n \in T$,
   \item[(iv)] $k \notin T$ and $j \notin T$ for every integer $j$
such that $n-k < j < n$,
   \end{enumerate}
is not a difference set (hint:\ replace $\varLambda$ by
$\varLambda - \min \varLambda$). The cardinality of such
sets $T$ may (and does) vary between $5$ and $2(n-k) + 1$.
There is a simple way of producing multidimensional
variants of non-difference sets from one-dimensional ones.
Indeed, if $T_1$ is a non-difference subset of $\zbb$ and
$T^\prime$ is any subset of $\zbb^\varkappa$, then $T_1
\times T^\prime$ is a non-difference subset of
$\zbb^{\varkappa+1}$ (hint:\ $P(\varLambda - \varLambda) =
P(\varLambda) - P(\varLambda)$, where $P(n,\alpha) = n$ for
$n \in \zbb$ and $\alpha \in \zbb^\varkappa$). Of course,
there are non-difference sets which cannot be obtained this
way, e.g. $\varkappa = 2$ and $T= \{(0,0),
(1,0),(-1,0),(1,1),(-1,-1)\}$. Summing up, our solutions of
the truncated trigonometric moment problem given in Theorem
\ref{trmpr} below allow for much more general truncations,
even in the case of finite data $T$. Surprisingly, the
infinite set $\zbb_\pluss^\varkappa \cup
\zbb_\minuss^\varkappa$, playing a pivotal role in Section
\ref{sekl}, is a difference set. Indeed, if $\zbb_\pluss
\ni n \mapsto \alpha_n \in \zbb_\pluss^\varkappa$ is any
surjection with $\alpha_0 = 0$, then $\zbb_\pluss^\varkappa \cup
\zbb_\minuss^\varkappa = \varLambda - \varLambda$ with
   \begin{align*}
\varLambda=\{\alpha_0 + \ldots + \alpha_n\colon n \in
\zbb_\pluss\}.
   \end{align*}
Another (more explicit) choice of $\varLambda$ for $\varkappa=2$ is
illustrated in Figure 2.

\vspace{1,5ex}
   \begin{center}
\unitlength=1.5pt
\begin{picture}(80,80)
\put(0,10){\vector(1,0){80}}

\put(10,0){\vector(0,1){80}}

\put(76,4){\text{$m$}} \put(4,77){\text{$n$}}

\put(10,10){\circle*{2}} \put(20,10){\circle*{2}}
\put(14,24){\text{\footnotesize $(1,1)$}}
\put(20,20){\circle*{2}} \put(30,20){\circle*{2}}
\put(40,20){\circle*{2}} \put(40,30){\circle*{2}}
\put(34,44){\text{\footnotesize $(3,3)$}}
\put(40,40){\circle*{2}} \put(50,40){\circle*{2}}
\put(60,40){\circle*{2}} \put(70,40){\circle*{2}}
\put(70,50){\circle*{2}} \put(70,60){\circle*{2}}
\put(64,74){\text{\footnotesize $(6,6)$}}
\put(70,70){\circle*{2}} \put(80,70){\circle*{2}}
\end{picture}\\[1.5ex]
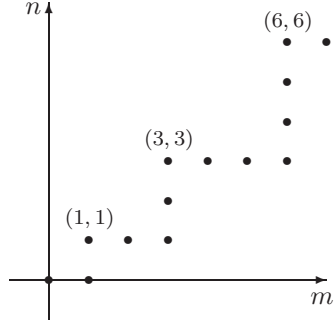
{\small Figure 2. An example of $\varLambda$ such that
$\zbb_\pluss^2\cup\zbb_\minuss^2 = \varLambda-\varLambda$.}
\end{center}

\vspace{1,5ex}
   We now go back to solving the truncated trigonometric
moment problem.
   \begin{thm} \label{trmpr}
Assume that $T$ is a subset of $\zbb^\varkappa$ such
that\,\footnote{\;\label{expl}This is the explicit form of
condition \eqref{sym_diag} under the circumstances of the
$*$-semigroup $\zbb^\varkappa$.} $0 \in T = -T$, and $Y$ is
a determining set for $\pcal_\pluss(\mathbb T^\varkappa)$.
If $\{c_\alpha\}_{\alpha \in T}$ is a sequence of complex
numbers, then the following conditions are equivalent{\em
:}
    \begin{enumerate}
    \item[(i)] there exists a finite positive Borel measure
$\mu$ on $\mathbb T^\varkappa$ such that
    \begin{align}
c_\alpha = \int_{\mathbb T^\varkappa} z^\alpha \D \mu(z), \quad
\alpha \in T,
    \end{align}
   \item[(ii)] $\sum_{\alpha \in T} a_\alpha c_\alpha
\Ge 0$ for every finite system $\{a_\alpha\}_{\alpha \in T}$ of
complex numbers such that $\sum_{\alpha \in T} a_\alpha z^\alpha \Ge
0$ for all $z \in \mathbb T^\varkappa$,
   \item[(iii)] $\sum_{\alpha \in T} a_\alpha c_\alpha
\Ge 0$ for every finite system $\{a_\alpha\}_{\alpha \in T}$ of
complex numbers for which there exist finitely many analytic
trigonometric polynomials $q_1, \ldots, q_k \in \pcal_\pluss(\mathbb
T^\varkappa)$ such that $\sum_{\alpha \in T} a_\alpha z^\alpha =
\sum_{j=1}^k |q_j(z)|^2$ for all $z \in Y$.
    \end{enumerate}
    \end{thm}
    \begin{proof}
We can argue essentially as in the proof of Theorem \ref{nagy} using
Theorem \ref{main2+} instead of Theorem \ref{main2}.
    \end{proof}
It is worth mentioning that Proposition \ref{sumk} can be easily
adapted to the present context. We say that a locally convex
topology $\tau$ on the linear space $\pcal_\pm(\mathbb T^\varkappa)$
(cf.\ Section \ref{sekl}) is {\em evaluable} if the set of all
points $\lambda \in \mathbb T^\varkappa$ for which the evaluation
    $$
\pcal_\pm(\mathbb T^\varkappa) \ni p \mapsto p(\lambda) \in \cbb
    $$
is $\tau$-continuous is dense in $\mathbb
T^\varkappa$. Given $Y \subset \mathbb T^\varkappa$,
we denote by $\varSigma^2(Y)$ the set of all
trigonometric polynomials $q \in \pcal_\pm(\mathbb
T^\varkappa)$ for which there exist finitely many
analytic trigonometric polynomials $q_1, \ldots, q_n
\in \pcal_\pluss(\mathbb T^\varkappa)$ such that
   \begin{align*}
q(z) = \sum_{j=1}^n |q_j(z)|^2, \quad z \in Y.
   \end{align*}
Arguing as in the proof of Proposition \ref{sumk} with
$p_\varepsilon \in \pcal_\pm(\mathbb T^\varkappa)$ given by
    $$
p_\varepsilon (z) =\sum_{j=1}^\varkappa |z_j - \lambda_j|^2
-\varepsilon^2 = 2\varkappa -\varepsilon^2 -\sum_{j=1}^\varkappa
\bar \lambda_j z_j - \sum_{j=1}^\varkappa \lambda_j z_j^{-1}, \quad
z=(z_1, \ldots, z_\varkappa) \in \mathbb T^\varkappa,
    $$
where $(\lambda_1, \ldots, \lambda_\varkappa) \in \mathbb
T^\varkappa \setminus Y$, we are led to the ensuing result.
   \begin{pro} \label{sumk+}
Let $Y$ be a closed proper subset of $\mathbb T^\varkappa$
which is determining for $\pcal_\pluss(\mathbb
T^\varkappa)$ and let $\tau$ be an evaluable locally convex
topology on $\pcal_\pm(\mathbb T^\varkappa)$. Then there
exists a trigonometric polynomial $p \in \pcal_\pm(\mathbb
T^\varkappa)$ which is nonnegative on $Y$ and which does
not belong to the $\tau$-closure of $\varSigma^2(Y)$. In
particular, $p$ is not in $\varSigma^2(Y)$.
   \end{pro}
Note that if a trigonometric polynomial $p \in \pcal(\mathbb T) =
\pcal_\pm(\mathbb T)$ is nonnegative on $\mathbb T$, then by the
F\'ejer-Riesz theorem there exists an analytic trigonometric
polynomial $q \in \pcal_\pluss(\mathbb T)$ such that $p(z)=|q(z)|^2$
for all $z \in \mathbb T$.
   \subsection{\label{appr}Approximation}
In this section we intend to apply our approach to
approximating nonnegative polynomials by sums of squares of
moduli of rational functions. The method presented here
could be a fertile source of other approximation results of
this kind.

Let $T$ be a subset of $\zbb_\pluss \times \zbb_\pluss$.
Denote by $\varSigma_T^2(\cbb)$ the set of all polynomials
$q \in \cbb_T[z, \bar z]$ for which there exist finitely
many rational functions $q_1, \ldots, q_k \in
\cbb_{\nfr_+}(z,\bar z)$ such that $q(z, \bar z) =
\sum_{j=1}^k |q_j(z, \bar z)|^2$ for all $z \in \cbb_*$.
Let $\mathscr P_T^+(\cbb)$ stand for the set of all
polynomials $q \in \cbb_T[z,\bar z]$ such that $q(z,\bar z)
\Ge 0$ for all $z\in \cbb$. We shall regard
$\varSigma_T^2(\cbb)$, $\mathscr P_T^+(\cbb)$ and
$\cbb_T[z, \bar z]$ as sets of complex functions on $\cbb$.
Given $p \in \cbb[z, \bar z]$ of the form $p(z,\bar z) =
\sum_{m,n \Ge 0} a_{m,n} z^m \bar z^n$, we set
$\|p\|_{\mathrm{co}} = \max\{|a_{m,n}| \colon m,n \Ge 0\}$.
The function $\|\cdot\|_{\mathrm{co}}$ is a norm on
$\cbb[z, \bar z]$. Recall that the finest locally convex
topology on a linear space $V$ is given by the family of
all seminorms on $V$.
    \begin{pro} \label{dens1}
Let $T$ be a subset of $\zbb_\pluss \times \zbb_\pluss$
such that $(n,m) \in T$ for all $(m,n) \in T$, and $(n,n)
\in T$ for all $n \in \zbb_\pluss$. Then the set
$\varSigma_T^2(\cbb)$ is dense in $\mathscr P_T^+(\cbb)$
with respect to the finest locally convex topology $\tau$
on $\cbb_T[z, \bar z]$. In particular, this is the case for
the topology of uniform convergence on compact subsets of
$\cbb$ and the topology given by the norm
$\|\cdot\|_{\mathrm{co}}$.
    \end{pro}
    \begin{proof}
It only suffices to discuss the case of the finest locally convex topology.
We regard $\cbb[z,\bar z]$ as a $*$-algebra of complex
functions on $\cbb$ with involution $q^*(z,\bar z) =
\overline{q(z,\bar z)}$ for $z \in \cbb$. Since the set $T$
is assumed to have a symmetry property, we see that
$\cbb_T[z,\bar z]$ is a vector subspace of $\cbb[z,\bar z]$
such that $q^* \in \cbb_T[z,\bar z]$ for every $q \in
\cbb_T[z,\bar z]$. It is then clear that
    $$\varSigma_T^2(\cbb) \subset \mathscr P_T^+(\cbb) \subset
\cbb_T[z,\bar z]_{\mathfrak h} \okr \{q \in \cbb_T[z,\bar z] \colon
q=q^*\}.
    $$
Suppose that, contrary to our claim, there exists $q_0 \in \mathscr
P_T^+(\cbb)$ which does not belong to the $\tau$-closure of
$\varSigma_T^2(\cbb)$. By the separation theorem (cf.\ \cite[Theorem
3.4\,(b)]{rud}), there exist a real-linear functional $\tilde
\varLambda\colon \cbb_T[z,\bar z] \to \rbb$ and $\gamma \in \rbb$
such that
    \begin{align} \label{rud}
\tilde \varLambda(q_0) < \gamma \Le \tilde \varLambda(q), \quad q
\in \varSigma_T^2(\cbb).
    \end{align}
Since $\varSigma_T^2(\cbb)$ has the property that $tq \in
\varSigma_T^2(\cbb)$ for all $q\in \varSigma_T^2(\cbb)$ and
$t \in [0,\infty)$, we deduce from \eqref{rud} that
    \begin{align*} 
\tilde\varLambda(q_0) < 0 \Le \tilde\varLambda(q), \quad q
\in \varSigma_T^2(\cbb),
    \end{align*}
Define $\varLambda\colon \cbb_T[z, \bar z] \to \cbb$ by
    \begin{align*}
\varLambda(q) = \tilde\varLambda(\rea\, q) + \I\, \tilde
\varLambda(\ima\, q), \quad q\in \cbb_T[z,\bar z],
    \end{align*}
with $\rea\, q \okr (q+q^*)/2$ and $\ima\, q \okr
(q-q^*)/(2\,\I)$. It is easily seen that $\varLambda$ is a
complex-linear functional such that
$\varLambda(q)=\tilde\varLambda(q)$ for all $q\in
\cbb_T[z,\bar z]_{\mathfrak h}$. Hence, in view of the
inclusion $\{q_0\} \cup \varSigma_T^2(\cbb) \subset
\cbb_T[z,\bar z]_{\mathfrak h}$, we have
    \begin{align} \label{contr}
\varLambda(q_0) < 0 \Le \varLambda(q), \quad q \in
\varSigma_T^2(\cbb).
    \end{align}
Since $\varLambda$ is of the form
    $$ \varLambda(q) = \sum_{(m,n) \in T}
p_{m,n} c_{m,n} \quad \text{for} \quad q \in \cbb_T[z,\bar z]\colon
q(z,\bar z) = \sum_{(m,n)\in T} p_{m,n} z^m \bar z^n,
    $$ where
$\{c_{m,n}\}_{(m,n) \in T}$ is a system of complex numbers uniquely
determined by $\varLambda$, we deduce from the weak inequality in
\eqref{contr} and implication (v)$\Rightarrow$(iv) of Theorem
\ref{zespmom} that $\varLambda(q_0) \Ge 0$, which contradicts the
strict inequality in \eqref{contr}.
    \end{proof}
Consider now the case of $T= \nfr=\zbb_\pluss \times
\zbb_\pluss$. Observe that $\widehat\varSigma^2_\nfr(\cbb)$, the convex cone of all polynomials $q \in \cbb[z,\bar z]$ for which there exist finitely many polynomials $p_1, \ldots, p_n \in \cbb[z,\bar z]$ such that $q(z, \bar z) = \sum_{j=1}^n |p_j(z, \bar z)|^2$ for all $z\in \cbb$, is closed with respect to the finest locally convex topology of $\cbb[z,\bar z]$ (see the proof of \cite[Theorem 6.3.5]{b-ch-r}). Evidently, the convex cone $\mathscr
P_\nfr^+(\cbb)$ is closed with respect to the same topology. It is clear that
   \begin{align}\label{kapelusz}
\widehat\varSigma^2_\nfr(\cbb) \subset \varSigma_\nfr^2(\cbb) \subset \mathscr P_\nfr^+(\cbb).
   \end{align}
Owing to the proof of \cite[Theorem 6.3.5]{b-ch-r}, $\widehat\varSigma^2_\nfr(\cbb)$ is a proper subset of $ \mathscr P_\nfr^+(\cbb)$. Hence, by Proposition \ref{dens1}, the first inclusion in \eqref{kapelusz} is proper as well. The open question whether the second inclusion in \eqref{kapelusz} is proper has already been posed in the paragraph preceding Proposition \ref{sumk}.
    \begin{cor} \label{our}
Let
    \begin{align*}
q(z,\bar z) = \sum_{m,n \Ge 0} a_{m,n} z^m \bar z^n, \quad z \in
\cbb,
    \end{align*}
be a nonnegative polynomial with coefficients in
$\cbb$. Then for every $\varepsilon > 0$, there exists
a nonnegative polynomial $q_\varepsilon \in
\varSigma_\nfr^2(\cbb)$ such that
    \begin{align*}
& q_\varepsilon(z, \bar z) = \sum_{m,n \Ge 0}
a_{m,n}^{(\varepsilon)} z^m \bar z^n, \quad z \in \cbb \quad
\big(a_{m,n}^{(\varepsilon)} \in \cbb\big),
    \\
& |a_{m,n} - a_{m,n}^{(\varepsilon)}| \Le \varepsilon |a_{m,n}|,
\quad m,n \in \zbb_\pluss, m \neq n,
    \\
& |a_{m,n} - a_{m,n}^{(\varepsilon)}| \Le \varepsilon, \quad m,n \in
\zbb_\pluss.
    \end{align*}
    \end{cor}
    \begin{proof}
It is enough to consider the case of $q\neq 0$. Since the polynomial
$q$ is nonnegative, we deduce that the set
    $$
T = \{(m,n) \in \zbb_\pluss \times \zbb_\pluss \colon a_{m,n} \neq
0\} \cup \{(n,n) \colon n \in \zbb_\pluss\}
    $$
satisfies the assumptions of Proposition \ref{dens1}, and $q \in
\mathscr P_T^+(\cbb)$. By this proposition, there exists
$q_\varepsilon \in \varSigma_T^2(\cbb)$ such that $\|q -
q_\varepsilon\|_{\mathrm{co}} \Le \tilde\varepsilon$, where
    $$
\tilde\varepsilon = \varepsilon \cdot \min\Big\{1,
\min\big\{|a_{m,n}|\colon m,n \Ge 0, \, a_{m,n} \neq 0\big\}\Big\}
>0.
    $$
This completes the proof.
    \end{proof}
We now discuss the case of trigonometric polynomials (see
Section \ref{sekl} for notation). Let $T$ be a subset of
$\zbb^\varkappa$ and let $\varSigma_T^2(\mathbb
T^\varkappa)$ stand for the set of all trigonometric
polynomials $q \in \pcal_T(\mathbb T^\varkappa)$ for which
there exist finitely many analytic trigonometric
polynomials $q_1, \ldots, q_k \in \pcal_\pluss(\mathbb
T^\varkappa)$ such that $q(z) = \sum_{j=1}^k |q_j(z)|^2$
for all $z \in \mathbb T^\varkappa$. Denote by $\mathscr
P_T^+(\mathbb T^\varkappa)$ the set of all trigonometric
polynomials $q \in \pcal_T(\mathbb T^\varkappa)$ such that
$q(z) \Ge 0$ for all $z\in \mathbb T^\varkappa$. Given $p
\in \pcal(\mathbb T^\varkappa)$ of the form \eqref{for1},
we set $\|p\|_{\mathrm{co}}^\prime = \max\{|a_\alpha|\colon
\alpha \in \zbb^\varkappa\}$. It is clear that
$\|\cdot\|_{\mathrm{co}}^\prime$ is a norm on
$\pcal(\mathbb T^\varkappa)$.

    Arguing as in the proof of Proposition \ref{dens1} (using
Theorem \ref{trmpr} instead of Theorem \ref{zespmom}), we get
another approximation result.
    \begin{pro} \label{dens2}
Let $T$ be a subset of $\zbb^\varkappa$ such that $0\in T=-T$. Then
the set $\varSigma_T^2(\mathbb T^\varkappa)$ is dense in $\mathscr
P_T^+(\mathbb T^\varkappa)$ with respect to the finest locally
convex topology on $\pcal_T(\mathbb T^\varkappa)$. In particular,
this is the case for the topology of uniform convergence on $\mathbb
T^\varkappa$ and the topology given by the norm
$\|\cdot\|_{\mathrm{co}}^\prime$.
    \end{pro}
    \begin{cor} \label{Drit}
Let
    \begin{align*}
q(z) = \sum_{\alpha \in \zbb^\varkappa} a_\alpha z^\alpha, \quad z
\in \mathbb T^\varkappa,
    \end{align*}
be a nonnegative trigonometric polynomial with
coefficients in $\cbb$. Then for every $\varepsilon >
0$, there exists a nonnegative trigonometric
polynomial $q_\varepsilon \in
\varSigma_{\zbb^\varkappa}^2(\mathbb T^\varkappa)$
such that
    \begin{align*}
& q_\varepsilon(z) = \sum_{\alpha \in \zbb^\varkappa}
a_\alpha^{(\varepsilon)} z^\alpha, \quad z \in \mathbb T^\varkappa
\quad \big(a_\alpha^{(\varepsilon)} \in \cbb\big),
    \\
&|a_\alpha - a_\alpha^{(\varepsilon)}| \Le \varepsilon |a_\alpha|,
\quad \alpha \in \zbb^\varkappa \setminus \{(0, \ldots,0)\},
   \end{align*} and
   \begin{align} |a_\alpha - a_\alpha^{(\varepsilon)}| \Le \varepsilon ,
\quad \alpha \in \zbb^\varkappa. \label{excep}
    \end{align}
    \end{cor}
    \begin{proof}
It is enough to consider the case of $q\neq 0$. Since the
trigonometric polynomial $q$ is nonnegative, we see that the set $T
= \{\alpha \in \zbb^\varkappa \colon a_\alpha \neq 0\} \cup \{(0,
\ldots,0)\}$ satisfies the assumptions of Proposition \ref{dens2},
and $q \in \mathscr P_T^+(\mathbb T^\varkappa)$. By this
proposition, there exists $q_\varepsilon \in \varSigma_T^2(\mathbb
T^\varkappa)$ such that $\|q - q_\varepsilon\|_{\mathrm{co}}^\prime
\Le \tilde\varepsilon$, where
    $$
\tilde\varepsilon = \varepsilon \cdot \min\Big\{1,
\min\big\{|a_\alpha|\colon \alpha \in \zbb^\varkappa, \, a_\alpha
\neq 0\big\}\Big\} >0.
    $$
This completes the proof.
    \end{proof}
For the case of $T=\zbb^\varkappa$ it is well known that $\mathscr
P_{\zbb^\varkappa}^+(\mathbb T^\varkappa) \setminus
\varSigma_{\zbb^\varkappa}^2(\mathbb T^\varkappa) \neq
\varnothing$ if $\varkappa \Ge 2$ (cf.\ \cite[page 51]{sa}).
Recently Dritschel proved that each (strictly) positive
trigonometric polynomial in $\pcal(\mathbb T^\varkappa)$ belongs
to $\varSigma_{\zbb^\varkappa}^2(\mathbb T^\varkappa)$ (see
\cite[Corollary 5.2]{dr} where operator valued trigonometric
polynomials are considered; see also \cite{g-w} for related
questions concerning scalar trigonometric polynomials in two
variables). The proof of Dritschel's result is based on his
Theorem 5.1 which in the scalar case coincides with the major part
of our Corollary \ref{Drit} (except for \eqref{excep}).
   \subsection{\label{twosided} The truncated two-sided complex moment problem}
Let $\zfr$ stand for the $*$-semigroup $(\zbb \times \zbb,
+, *)$ with coordinatewise defined addition as semigroup
operation, and involution $(m,n)^*=(n,m)$. Given a subset
$T$ of $\zfr$, we denote by $\varSigma_T^2(\cbb_*)$ the set
of all rational functions $q \in \cbb_T(z, \bar z)$ for
which there exist finitely many rational functions $q_1,
\ldots, q_k \in \cbb_\zfr (z,\bar z)$ such that $q(z, \bar
z) = \sum_{j=1}^k |q_j(z, \bar z)|^2$ for all $z \in
\cbb_*$. Let $\mathscr P_T^+(\cbb_*)$ stand for the set of
all rational functions $q \in \cbb_T(z,\bar z)$ such that
$q(z,\bar z) \Ge 0$ for all $z\in \cbb_*$. We regard
$\varSigma_T^2(\cbb_*)$, $\mathscr P_T^+(\cbb_*)$ and
$\cbb_T(z, \bar z)$ as sets of complex functions on
$\cbb_*$. Given $p \in \cbb_\zfr(z, \bar z)$ of the form
$p(z,\bar z) = \sum_{m,n \in \zbb} a_{m,n} z^m \bar z^n$,
we set $\|p\|_{\mathrm{co}} = \max\{|a_{m,n}| \colon m,n
\in \zbb\}$. The function $\|\cdot\|_{\mathrm{co}}$ is a
norm on $\cbb_\zfr(z, \bar z)$.

It was proved by Bisgaard (cf.\ \cite{bis1}) that the
$*$-semigroup $\zfr$ is semiperfect (in fact, as noted in
\cite[Theorem 24]{st-sz1}, $\zfr$ is operator semiperfect).
This result enables us to apply our machinery. Below, we
formulate counterparts of Theorem \ref{zespmom} (a shorter
version), Proposition \ref{dens1} and Corollary \ref{our}.
Their proofs are analogical.
   \begin{thm} \label{zespmom+}
Let $T$ be a symmetric subset of $\zfr$ $($i.e.\ $(n,m)\in
T$ for all $(m,n) \in T$$)$ containing the
diagonal\,\footnote{\;Again, as noticed in footnote
\ref{expl}, this is the form condition \eqref{sym_diag}
takes now.} $\{(n,n)\colon n \in \zbb\}$, and let $Z\subset
\cbb_*$ be a determining set for\/ $\cbb[z,\bar z]$. Then
for any system of complex numbers $\{c_{m,n}\}_{(m,n)\in
T}$, the following conditions are equivalent{\em :}
   \begin{enumerate}
   \item[(i)] there exists a positive Borel measure $\mu$ on
$\cbb_*$ such that
   \begin{align*}
   c_{m,n} = \int_{\cbb_*} z^m \bar z^n \mu(\D z), \quad
(m,n) \in T,
   \end{align*}
   \item[(ii)] $\sum_{(m,n)\in T} p_{m,n} c_{m,n}
\Ge 0$ for every finite system $\{ p_{m,n} \}_{(m,n) \in T}
\subset \cbb$ such that $\sum_{(m,n)\in T} p_{m,n} z^m \bar
z^n \Ge 0$ for all $z \in \cbb_*$,
   \item[(iii)] $\sum_{(m,n)\in T} p_{m,n} c_{m,n}
\Ge 0$ for every finite system $\{ p_{m,n} \}_{(m,n) \in T}
\subset \cbb$ for which there exist finitely many rational
functions $q_1, \ldots, q_k \in \cbb_{\zfr}(z,\bar z)$ such
that $\sum_{(m,n)\in T} p_{m,n} z^m \bar z^n = \sum_{j=1}^k
|q_j(z, \bar z)|^2$ for all $z \in Z$.
   \end{enumerate}
   \end{thm}
    \begin{pro} \label{dens1+}
Let $T$ be as in Theorem {\em \ref{zespmom+}}. Then the set
$\varSigma_T^2(\cbb_*)$ is dense in $\mathscr
P_T^+(\cbb_*)$ with respect to the finest locally convex
topology on $\cbb_T(z, \bar z)$. In particular, this
is the case for the topology of uniform convergence on
compact subsets of $\cbb_*$ and the topology given by the
norm $\|\cdot\|_{\mathrm{co}}$.
    \end{pro}
    \begin{cor} \label{our+}
Let
    \begin{align*}
q(z,\bar z) = \sum_{m,n \in \zbb} a_{m,n} z^m \bar z^n,
\quad z \in \cbb_*,
    \end{align*}
be a nonnegative rational function with coefficients in
$\cbb$. Then for every $\varepsilon
> 0$, there exists a nonnegative rational function
$q_\varepsilon \in \varSigma_\zfr^2(\cbb_*)$ such that
    \begin{align*}
& q_\varepsilon(z, \bar z) = \sum_{m,n \in \zbb}
a_{m,n}^{(\varepsilon)} z^m \bar z^n, \quad z \in \cbb_*
\quad \big(a_{m,n}^{(\varepsilon)} \in \cbb\big),
    \\
& |a_{m,n} - a_{m,n}^{(\varepsilon)}| \Le \varepsilon
|a_{m,n}|, \quad m,n \in \zbb, m \neq n,
    \\
& |a_{m,n} - a_{m,n}^{(\varepsilon)}| \Le \varepsilon,
\quad m,n \in \zbb.
    \end{align*}
    \end{cor}
Yet another proof of Proposition \ref{dens1+} consists in
deriving it from Proposition \ref{dens1}. In turn,
Corollary \ref{our+} can be deduced directly from
Proposition \ref{dens1+}. However, Theorem \ref{zespmom+}
does not seem to be an easy consequence of Theorem
\ref{zespmom}.
   \subsection{Concluding remarks}
A result of Bisgaard (cf.\ \cite{bis1,bis5}) states
that the Cartesian product of two $*$-semigroups one
of which is (operator) perfect and the other
(operator) semiperfect is (operator) semiperfect. This
fact can be applied to $*$-semigroups $\zbb^\varkappa$
(operator perfect) and $\nfr_\pluss$ (operator
semiperfect). Employing our method in this particular
context we obtain appropriate results for mixed
polynomials
    $$
\sum_{\alpha \in \zbb^\varkappa} \sum_{m,n \Ge 0} a_{\alpha,m,n}
w^\alpha z^m \bar z^n, \quad w \in \mathbb T^\varkappa, \, z \in
\cbb.
    $$
Fortunately, there is a variety of $*$-semigroups which are
semiperfect or operator semiperfect, and for which we can
find convenient description of their dual $*$-semigroups
and Laplace transforms (see e.g.\
\cite{b-ch-r,bis-re,na-sa,js2,bis,bis3,ni-sa,fu-sa,st-sz3}).

As we now show, a finitely generated $*$-semigroup with
unit can be modelled as a quotient of the $*$-semigroup
$\zbb_\pluss^\varkappa \times \zbb_\pluss^\varkappa$, which
gives rise to yet another way of describing its dual.
Suppose that a commutative $*$-semigroup $\sfr$ with unit
is finitely generated. Then there exists a finite system
$\ebf= (e_1, \ldots,e_\varkappa)$ of elements of $\sfr$
such that the mapping
   \begin{align*}
\varPhi=\varPhi_\ebf \colon \zbb_\pluss^\varkappa \times
\zbb_\pluss^\varkappa \ni (\alpha,\beta) \mapsto
\ebf^\alpha \ebf^{*\beta} \in \sfr
   \end{align*}
is a surjection, where $\ebf^\alpha = e_1^{\alpha_1} \dots
e_\varkappa^{\alpha_\varkappa}$ and $\ebf^{*\beta} =
e_1^{*\beta_1} \dots e_\varkappa^{*\beta_\varkappa}$ with $\alpha = (\alpha_1, \ldots, \alpha_\varkappa)$ and $\beta = (\beta_1, \ldots, \beta_\varkappa)$. Let
us regard $\nfr_\varkappa \okr \zbb_\pluss^\varkappa \times
\zbb_\pluss^\varkappa$ as a $*$-semigroup with
coordinatewise defined addition as semigroup operation and
involution $(\alpha,\beta)^* = (\beta, \alpha)$ for
$\alpha, \beta \in \zbb_\pluss^\varkappa$. Then $\varPhi$
is a $*$-epimorphism of $*$-semigroups, and consequently
the relation $R=R_\varPhi$ on $\nfr_\varkappa$ defined by
   \begin{align*}
(\alpha,\beta) R (\alpha^\prime,\beta^\prime) \quad
\text{if} \quad \varPhi(\alpha,\beta) =
\varPhi(\alpha^\prime,\beta^\prime)
   \end{align*}
is a congruence. As a consequence, the $*$-semigroup $\sfr$
is $*$-isomorphic to the quotient $*$-semigroup
$\nfr_\varkappa/R$ (consult \cite[Theorem I.1.5]{hun}).
This means that the $*$-semigroups $\nfr_\varkappa/R$,
$\varkappa \Ge 1$, are model $*$-semigroups in the category
of finitely generated $*$-semigroups. We now describe the
dual $*$-semigroup of the model $*$-semigroup
$\nfr_\varkappa/R$. In what follows, we regard
$\cbb^\varkappa$ as a $*$-semigroup with coordinatewise
defined multiplication as semigroup operation and
involution $(z_1, \ldots,z_\varkappa)^* = (\bar z_1,
\ldots, \bar z_\varkappa)$. First, note that the set
   \begin{align*}
\mathfrak F_R = \{z \in \cbb^\varkappa \colon z^\alpha \bar
z^\beta = z^{\alpha^\prime} \bar z^{\beta^\prime} \text{
whenever } (\alpha,\beta) R (\alpha^\prime,\beta^\prime)\}
   \end{align*}
is a $*$-subsemigroup of $\cbb^\varkappa$. It is a matter
of routine to verify that the mapping
   \begin{align} \label{nomore}
\varPsi \colon \mathfrak F_R \to \xfr_{\nfr_\varkappa/R},
\quad \varPsi(z) ([(\alpha, \beta)]_R) = z^\alpha \bar
z^\beta \text{ for } (\alpha, \beta) \in \nfr_\varkappa, \,
z \in \mathfrak F_R,
   \end{align}
where $[(\alpha, \beta)]_R$ is the equivalence class of
$(\alpha, \beta)$ with respect to the relation $R$, is a
$*$-isomorphism of $*$-semigroups (hint:\ the mapping
$\cbb^\varkappa \ni z \mapsto \varphi_z \in
\xfr_{\nfr_\varkappa}$, where $\varphi_z (\alpha, \beta) =
z^\alpha \bar z^\beta$ for $(\alpha, \beta) \in
\nfr_\varkappa$, is a $*$-isomorphism of $*$-semigroups).
Applying the measure transport theorem to \eqref{nomore},
one can rewrite the integrals $\int_{\xfr_\sfr} \hat s \,
\D \mu$ that appear in the conditions (ii) of Theorems
\ref{main2} and \ref{main2+} in the form
$\int_{\mathfrak F_R} z^\alpha \bar z^\beta \, \D \tilde
\mu(z)$ which resembles the solution of the multidimensional
complex moment problem. However, the example of the
$*$-semigroup $\nfr_\pluss$ in Section \ref{s4} shows that
this general approach is not always efficient.
   \section*{\sc Appendix}
First, we recall the Riesz-Haviland theorem which
completely characterizes multidimensional real moment
sequences (cf.\ \cite{hav1,hav2}; see also \cite{mr}
for the earlier solution of the one-dimensional real
moment problem). Below $L_{\boldsymbol a}$ is a unique linear
functional defined on the linear space of all complex
polynomials in $\varkappa$ indeterminates given by
$L_{\boldsymbol a}(\xbf^{\boldsymbol n}) =
a_{\boldsymbol n}$, ${\boldsymbol n} \in
\zbb_\pluss^\varkappa$ (with standard multiindex notation). Likewise, $L_{\boldsymbol c}$
appearing in Theorem B is determined by
$L_{\boldsymbol c}(z^m \bar z^n) = c_{m,n}$, $(m,n)
\in \zbb_\pluss^2$.
   \begin{thmA} 
If $\boldsymbol a = \{a_{\boldsymbol n}\}_{\boldsymbol n \in
\zbb_+^\varkappa}$ is a sequence of real numbers and $Z$ is a
closed subset of $\rbb^\varkappa$, then the following conditions
are equivalent{\em :}
   \begin{enumerate}
   \item[(a)] there exists a positive Borel measure $\mu$ on
$\rbb^\varkappa$ such that
   \begin{align*}
   a_{\boldsymbol n} = \int_Z \xbf^{\boldsymbol n} \D \mu(\xbf),
\quad \boldsymbol n \in \zbb_\pluss^\varkappa,
   \end{align*}
   \item[(b)] $L_{\boldsymbol a}(p)\Ge 0$ whenever $p$ is a complex polynomial in $\varkappa$
indeterminates such that $p(\xbf) \Ge 0$ for all $\xbf \in Z$.
   \end{enumerate}
   \end{thmA}
    Next, we formulate the complex variant of the
Riesz-Haviland theorem. It can be deduced from Theorem A
(with $\varkappa=2$) in an elementary way as indicated below.
   \begin{thmB} 
If $\boldsymbol c= \{c_{m,n}\}_{m,n=0}^\infty$ is a sequence of
complex numbers and $Z$ is a closed subset of $\cbb$, then the
following conditions are equivalent{\em :}
   \begin{enumerate}
   \item[(a)] there exists a positive Borel measure $\mu$ on
$\cbb$ such that
   \begin{align*}
   c_{m,n} = \int_{Z} z^m \bar z^n \mu(\D z), \quad m,n \Ge 0,
   \end{align*}
   \item[(b)] $L_{\boldsymbol c}(p)\Ge 0$ whenever $p$ is a complex polynomial in
$z$ and $\bar z$ such that $p(z,\bar z) \Ge 0$ for all $z \in Z$.
   \end{enumerate}
   \end{thmB}

Now we relate
the two-dimensional real moment problem to the
one-dimensional complex moment problem. It will
be done much in the spirit of \cite[Appendix]{st-sz1},
though by more elementary means.
The same proof works in the multidimensional
case.

For every $(m,n) \in \zbb_\pluss^2$, there exist finite
systems $\{ \alpha_{k,l}^{m,n} \}_{k,l=0}^\infty$ and $\{
\beta_{k,l}^{m,n} \}_{k,l=0}^\infty$ of complex numbers
uniquely determined by the following identities
   \begin{align}
(x+\I y)^m (x-\I y)^n & = \sum_{k,l\Ge 0}
\alpha_{k,l}^{m,n} x^k y^l, \quad x,y\in \rbb,\label{id1}
\\
\Big(\frac{z+\bar z}2\Big)^m \Big( \frac{z-\bar z}{2\,\I}
\Big)^n & = \sum_{k,l\Ge 0} \beta_{k,l}^{m,n} z^k \bar
z^l, \quad z \in \cbb. \label{id2}
   \end{align}
It can be readily checked that
   \begin{align}\label{id3}
\sum_{i,j\Ge 0} \alpha_{i,j}^{m,n} \beta_{k,l}^{i,j} =
\sum_{i,j\Ge 0} \beta_{i,j}^{m,n} \alpha_{k,l}^{i,j} =
\delta_{m,k}\delta_{n,l}, \quad k,l,m,n \in \zbb_\pluss.
   \end{align}
Given a sequence $\boldsymbol a =
\{a_{m,n}\}_{m,n=0}^\infty \subset \cbb$, we define
$\boldsymbol c(\boldsymbol a) = \{c_{m,n}(\boldsymbol
a)\}_{m,n=0}^\infty$ via
   \begin{align*}
c_{m,n}(\boldsymbol a) = \sum_{k,l\Ge 0}
\alpha_{k,l}^{m,n} a_{k,l}.
   \end{align*}
By \eqref{id3} the mapping $\cbb^{\zbb_+ \times \zbb_+} \ni
\boldsymbol a \mapsto \boldsymbol c(\boldsymbol a) \in
\cbb^{\zbb_+ \times \zbb_+}$ is a linear isomorphism, and
   \begin{align}\label{id4}
a_{m,n} = \sum_{k,l\Ge 0} \beta_{k,l}^{m,n}
c_{k,l}(\boldsymbol a), \quad m,n \in \zbb_\pluss,
   \end{align}
for all $\boldsymbol a = \{a_{m,n}\}_{m,n=0}^\infty \subset
\cbb$. Below we identify $\cbb$ with $\rbb^2$.
   \begin{pro} \label{dar}
Let $Z$ be a subset of $\cbb$. Then a sequence $\boldsymbol
a = \{a_{m,n}\}_{m,n=0}^\infty \subset \rbb$ satisfies the
condition {\em (a)} $($resp.\ {\em (b)}$)$ of Theorem {\em A} with $\varkappa =2$ if and only if the
sequence $\boldsymbol c(\boldsymbol a)$ satisfies the
condition {\em (a)} $($resp.\ {\em (b)}$)$ of Theorem {\em B}.
   \end{pro}
   \begin{proof}
If $a_{m,n} = \int_Z x^m y^n \D \mu(x,y)$ for all
$m,n \in \zbb_\pluss$, then
   \begin{align*}
c_{m,n}(\boldsymbol a) = \sum_{k,l\Ge 0}
\alpha_{k,l}^{m,n} a_{k,l} = \int_Z
\sum_{k,l\Ge 0} \alpha_{k,l}^{m,n} x^k y^l \D \mu(x,y)
\overset{\eqref{id1}}= \int_Z z^m \bar z^n \D \mu(z).
   \end{align*}
Conversely, if $\boldsymbol c(\boldsymbol a)$ is a complex
moment sequence with a representing measure $\mu$ on
$Z$, then
   \begin{align*}
a_{m,n} \overset{\eqref{id4}}= \sum_{k,l\Ge 0}
\beta_{k,l}^{m,n} c_{k,l}(\boldsymbol a) = \int_Z
\sum_{k,l\Ge 0} \beta_{k,l}^{m,n} z^k \bar z^l \D
\mu(z) \overset{\eqref{id2}}= \int_Z x^m y^n \D
\mu(x,y),
   \end{align*}
which completes the proof of the equivalence of both conditions
(a).

Similar calculation based on \eqref{id1}, \eqref{id2} and
\eqref{id4} justifies the equivalence of the conditions
(b).
   \end{proof}
Since the positive functional $L_{\boldsymbol a}$ given by a complex sequence $\boldsymbol a = \{a_{m,n}\}_{m,n=0}^\infty$ is automatically Hermitian (i.e.\ $L_{\boldsymbol a}(\bar p) = \overline{L_{\boldsymbol a}(p)}$), the sequence $\boldsymbol a$ is necessarily real. This allows us to deduce Theorem B from Theorem A and Proposition \ref{dar}.

    \vspace{1ex} {\bf Acknowledgments.} The very early
version of this paper was designated for Kre\u{\i}n's
anniversary volume. However, due to its growing
capacity we have been exceeding all consecutive
deadlines; let us thank, by the way, Professor Vadim
Adamyan for his patience in negotiating them.

    \bibliographystyle{amsalpha}
   
   \end{document}